\documentclass[10pt]{article}
\usepackage{amsmath,amsthm,amssymb,ifsym,color,a4wide}
\usepackage[twoside]{geometry}
\usepackage{graphicx}
\usepackage{epsfig}
\usepackage{caption}
\usepackage{subcaption} 

\usepackage[most]{tcolorbox} 
\newtcolorbox{plfok}{%
     enhanced, breakable, size=minimal, parbox=false, after={\par}, 
     before upper={\indent}, colback=white, g
     overlay = {\draw[dashed, line width=2pt] (frame.north east) -|
                       ([xshift=3mm]frame.east)|-(frame.south east);},
     overlay first={\draw[dashed, line width=2pt] (frame.north east) -|
                           ([xshift=3mm]frame.south east);},
     overlay middle={\draw[dashed, line width=2pt] ([xshift=3mm]frame.north east) -- 
                              ([xshift=3mm]frame.south east);},
     overlay last={\draw[dashed, line width=2pt] ([xshift=3mm]frame.north east)|-
                          (frame.south east);},
}
\newtheorem{theorem}{Theorem}[section]

\newtheorem{lemma}[theorem]{Lemma}

\newtheorem{definition}[theorem]{Definition}

\numberwithin{equation}{section}
\numberwithin{figure}{section} 
\newtheorem{claim}{Claim}

\numberwithin{claim}{section} 
\numberwithin{equation}{section}
\newcommand \rhob {\overline \rho}

\newcommand \ub {\overline u}
\newcommand \bse {\begin{subequations}}
\newcommand \ese {\end{subequations}}

\newcommand \St {\widetilde S}

\newcommand \Ul {\overline U} 
\newcommand \Pl {\overline P}
\newcommand \Ql {\overline Q}
\newcommand \Sl {\overline S}
\newcommand \ut {\widetilde u}

\newcommand \trianglerightNEW \triangleright

\newcommand \sgn {\text{sgn}}

\newcommand \rhot {\widetilde \rho}

\newcommand \auth {\textsc}
\newcommand \Rcal {\mathcal R}

\newcommand \bei {\begin{itemize}}
\newcommand \eei {\end{itemize}}
\newcommand \be {\begin{equation}}
\newcommand \bel {\begin{equation}\label}
\newcommand \ee {\end{equation}}
\newcommand \del \partial
\newcommand \RR {\mathbb R}
\newcommand \FF {\mathcal F}

\newcommand \eps \epsilon 

\begin{document}

\title{Asymptotic structure of cosmological fluid flows in one 
\\
and two space dimensions: a numerical study} 

\author{Yangyang Cao$^1$, Mohammad A. Ghazizadeh$^2$, 
%
and Philippe G. LeFloch$^1$}

\date{December 2019}

\footnotetext[1]{Laboratoire Jacques-Louis Lions \& Centre National de la Recherche Scientifique,
Sorbonne Universit\'e, 4 Place Jussieu, 75252 Paris, France.
Email: {\sl caoy@ljll.math.upmc.fr, contact@philippelefloch.org.}
\newline
$^2$ Department of Civil Engineering, University of Ottawa, Ottawa, ON K1N 6N5, Canada. Email: sghaz023@uottawa.ca. 
} 

\maketitle
 
\begin{abstract} We consider an isothermal compressible fluid evolving on a cosmological background which may be either expanding or contracting toward the future. The Euler equations governing such a flow consist of two nonlinear hyperbolic balance laws which we treat in one and in two space dimensions. 
We design a finite volume scheme which is fourth-order accurate in time and second-order accurate in space. This scheme allows us to compute weak solutions containing shock waves and, by design, is well-balanced in the sense that it preserves exactly a special class of solutions. 
Using this scheme, we investigate the asymptotic structure of the fluid when the time variable approaches infinity (in the expanding regime) or approaches zero (in the contracting regime). We study these two limits by introducing a suitable rescaling of the density and velocity variables and, in turn, we analyze the effects induced by the geometric terms (of expanding or contracting nature) and the nonlinear interactions between shocks. Extensive numerical experiments in one and in two space dimensions are performed in order to support our observations. 
\end{abstract}

%


\section{Introduction}
\label{thesection:1}

One of the challenges we address  in this paper is the design of a geometry-preserving method that accurately compute fluid flows in presence of non-homogeneous geometrical effects. We build upon earlier work on structure-preserving numerical methods for nonlinear hyperbolic problems, for instance in Chertock et al. \cite{Chertock2}, Michel-Dansac  et al. \cite{MVCS2016}, and Russo \cite{Russo1,Russo2}
and the many references cited therein. 
Specifically, we are interested in the global dynamics of compressible relativistic fluids evolving on a curved cosmological background of expanding or contracting type. Our model is essentially the Euler system posed on a FLRW background (after Friedmann--Lema\^{i}tre--Robertson--Walker) and describe fluids evolving on a homogeneous and isotropic cosmology. Recall that shock wave solutions to
nonlinear hyperbolic equations (such as the Euler equations) are defined in the {\sl forward time} direction only and, since the geometry under consideration is singular at the time normalized to be $t = 0$, we distinguish between two formulations of the initial value problem, corresponding to different ranges of the time variable.   

In the range $t \in (0, + \infty)$, the background is expanding toward the future and a defining geometric coefficient $a=a(t)$ (related to the speed of cosmological expansion or contraction, introduced below) increases monotonically to $+ \infty$. In this regime, we prescribe the initial data at some positive time $t_0>0$. 
On the other hand, in the range $t \in (-\infty,0)$, the background is contracting toward the future and the coefficient $a=a(t)$ decreases monotonically until it reaches $a(0)=0$. In this regime, we prescribe the initial data at some negative $t_0<0$. 
 A  typical choice for the defining geometric coefficient is 
the function $a(t) = a_0 (t / t_0)^\alpha$, often normalized by taking $a_0 = 1$ and $t_0 = \pm 1$ and for some rate of contraction/expansion $\alpha \in(0, 1)$. 
In our study, we also allow the background spacetime to be {\sl spatially non-homogeneous}.

Our aim is to design a numerical algorithm adapted to this problem and based on the finite volume methodology, and next to investigate the late-time asymptotics of solutions in the expanding direction as well as the behavior of the flow as one approaches the cosmological singularity. 
The model of interest here is given by the Euler equations and, for the sake of simplicity, the fluid sound speed is assumed to be a constant. 
By distinguishing between two cases, whether the background is expanding or contracting toward the future, we find that a fine structure arises which consists of non-interacting shock waves that move periodically in time. Importantly, our scheme is sufficiently robust so that this fine structure is correctly captured by our algorithm at a reasonable computational cost. 
The proposed method relies on a high-order Runge-Kutta discretization in the time variable and a structure-preserving technique for the spatial discretization.  

Our aim is to investigate the fine structure of the solutions in the expanding direction $t \to +\infty$ as well as in the contracting direction $t \to 0$. We expect that the flow structure will somehow ``simplifies'' asymptotically, and we are interested in the competition taking place between the shock propagation and the background geometry. By working with periodic boundary conditions, we expect that the shocks will interact until only a simple pattern is left, typically the so-called $N$-wave profile that is well-known for nonlinear systems of conservation laws. However, due to the non-homogeneous terms, the global dynamics of the flow turns out to be very complex. 

For the theory of weak solutions to hyperbolic conservation laws on geometric background, we refer to \cite{ALO,BLF,BL,BLM1} and \cite{CLO,Gie09,GM14,PLF1,LM13}. The numerical computation for such equations has also recently received some attention; see for instance \cite{PLFM,LFX,LFX2}.

An outline of this paper is as follows. In Section~\ref{thesection:2}, we present the model we study in this paper and introduces some formal asymptotics by distinguishing between the expanding and contracting cases.  In Section~\ref{thesection:3}, the finite volume methodology is explained before presenting the construction of our numerical algorithm in Section~\ref{thesection:4}. The numerical results are presented and discussed in Section~\ref{thesection:5} (concerned with the global dynamics on a future-expanding background)
 and Section~\ref{thesection:6} (concerned with the global dynamics on a future-contracting background)
and the paper ends with definite conclusions concerning the asymptotics of the solutions. 


\section{A model of cosmological fluid flows}
\label{thesection:2}

\subsection{The equations of interest}

The Euler equations in two spatial variables $x,y \in  [0,1]$ (i.e.~the torus with periodic boundary conditions) read as follows: 
\bse
\label{EulerAbc700}
\bel{EulerAbc700-a}
\aligned
& \del_t \Big(\rho (1 + \eps^4 k^2 V^2) \Big) 
 + \del_x \Big( \rho u ( 1 + \eps^2 k^2) \Big) 
 + \del_y \Big( \rho v ( 1 + \eps^2 k^2) \Big) 
=  S_0, 
\\
& \del_t \Big( \rho u (1+ \eps^2 k^2) \Big)
 + \del_x \Big( ( 1 + \eps^2 k^2) \rho u^2  + k^2 \rho ( 1 - \eps^2 V^2)
 \Big) 
 +  \del_y \Big( ( 1 + \eps^2 k^2) \rho u v   \Big) 
 = S_1, 
\\
& \del_t \Big( \rho v (1+ \eps^2 k^2) \Big)
 +  \del_x \Big( ( 1 + \eps^2 k^2) \rho u v   \Big) 
 +  \del_y \Big(( 1 + \eps^2 k^2) \rho v^2 + k^2 \rho ( 1 - \eps^2 V^2)
\Big) 
 = S_2, 
\endaligned
\ee
with 
\bel{EulerAbc700-b}
\aligned
& S_0 = - {\del_t a \over a} \rho \Big( 1 + 3 \eps^2 k^2 + (1 - \eps^2 k^2) \eps^2 V^2 \Big), 
\\
& S_1 =  2 \rho 
\Big( k^2 \, {\del_x b \over b} (1 - \eps^2 V^2) - {\del_t a \over a}  (1+ \eps^2 k^2) u
\Big),
\qquad
 S_2 = 2 \rho 
\Big( k^2 \, {\del_y b \over b} (1 - \eps^2 V^2) - {\del_t a \over a}  (1+ \eps^2 k^2) v
\Big). 
\endaligned
\ee
\ese
Here, the main unknowns are the (suitably normalized) density $\rho = \rho(t, x,y) \geq 0$ and the velocity components $(u,v) = (u,v)(t, x, y)$ with $V^2 = u^2 + v^2  < 1/\eps^2$. The coefficient $k \in (0, 1/\eps)$ represents the sound speed, while the light speed is $1/\eps$. 
Periodic boundary conditions are imposed, that is, 
\bel{BounC}
(\rho, u, v)(t, 0) = (\rho, u, v)(t, 1). 
\ee

Moreover, the functions $a = a(t)>0$ and  $b = b(x,y) >0$ are prescribed and describe the background geometry. Two regimes for the time variable are considered: 
\bel{eq:tmms9} 
t \in 
\begin{cases} 
[1, + \infty),
 \qquad 
& \text{ future-expanding,}
\\
[-1,0), 
& \text{ future-contracting,}
\end{cases}
\ee 
a typical function to be considered below being 
\be
a(t) = |t|^\kappa, 
\ee
where $\kappa >0$ is a parameter. 
With obvious notation, we rewrite \eqref{EulerAbc700} in the form of a system of balance laws: 
\bel{equa:Euler-Form}
\del_t U + \del_x F(U) + \del_y G(U) = S(U,t, x, y). 
\ee 
This system is non-homogeneous and involves a nonlinear source depending on all of the independent variables $(t,x,y)$. 
We are interested in solving the initial value problem numerically, 
when an initial condition denoted by $U_0$ is prescribed at some time $t_0 \neq 0$: 
\bel{InitialC}
U(t_0, \cdot) = U_0. 
\ee
Due to the presence of shocks in the problem we can solve in the forward-time direction 
and we impose periodic spatial boundary conditions. 
 

\bse 
\label{equa-27ab}
We first consider the system in one space dimension, which reads as follows: 
\bel{EulerAbc700-1D}
\aligned
& \del_t \Big(\rho (1 + \eps^4 k^2 u^2) \Big) 
 + \del_x \Big( \rho u ( 1 + \eps^2 k^2) \Big)  
=  S_0, 
\\
& \del_t \Big( \rho u (1+ \eps^2 k^2) \Big)
 + \del_x \Big( \rho (u^2  + k^2) \Big)  
 = S_1,  
\endaligned
\ee
with 
\be
\aligned
& S_0 = - {\del_t a \over a} \rho \Big( 1 + 3 \eps^2 k^2 + (1 - \eps^2 k^2) \eps^2 u^2 \Big), 
\qquad 
S_1 =  2 \rho \, 
\Big( k^2 \, {\del_x b \over b} (1 - \eps^2 u^2) - {\del_t a \over a}  (1+ \eps^2 k^2) u
\Big). 
\endaligned
\ee
\ese
It is tedious but straightforward to check that these equations are strictly hyperbolic and admit the following two (distinct) wave speeds: 
\bel{equa:eigenv}
\lambda_1(u) = {u - k \over 1 - \eps^2 k u}, 
\qquad
\lambda_2(u) = {u + k \over 1 + \eps^2 k u}.
\ee

We also emphasize that the principal part of the system is expressed in {\sl normalized density and velocity components,} as follows. From the standard expression $T^{\mu\nu} = (1+ k^2) \rhob u^\mu u^\nu + k^2 \rhob \delta^{\mu\nu}$ of the stress-energy tensor for isothermal fluids with physical density $\rhob$
and unit, future-oriented velocity vector $u^\mu$ with $\mu, \nu=0,1,2$ for a flow in two space dimensions, we write $\del_\mu T^{\mu\nu} =0$. We express the Euler equations in terms of the two velocity components $u,v \in \RR$, as follows: 
$$
u^0 =  \eps(1- \eps^2 V^2)^{-1/2}, \qquad 
u^1 =  \eps u (1- \eps^2 V^2)^{-1/2}, \qquad 
u^2 =  \eps v (1- \eps^2 V^2)^{-1/2},  
$$
in which $V = (u^2+ v^2)^{1/2}$. It is then clear that 
$$
\eps^{-2} (u^0)^2 - (u^1)^2 - (u^2)^2
= (1- \eps^2 V^2)^{- 1} - \eps^2 u^2 (1- \eps^2 V^2)^{-1} - \eps^2 v^2 (1- \eps^2 V^2)^{-1}
= 1, 
$$ 
as expected. We also rescale the density by setting $\rhob =: \rho (1- \eps^2 V^2)$ and, in turn, we have arrived at the following homogeneous version of the Euler equations: 
\bel{Eule0-a}
\aligned
& \del_t \Big(\rho (1 + \eps^4 k^2 V^2) \Big) 
 + \del_x \Big( \rho u ( 1 + \eps^2 k^2) \Big) 
 + \del_y \Big( \rho v ( 1 + \eps^2 k^2) \Big) 
=  S_0, 
\\
& \del_t \Big( \rho u (1+ \eps^2 k^2) \Big)
 + \del_x \Big( ( 1 + \eps^2 k^2) \rho u^2  + k^2 \rho ( 1 - \eps^2 V^2) \Big) 
 +  \del_y \Big( ( 1 + \eps^2 k^2) \rho u v   \Big) 
 = 0, 
\\
& \del_t \Big( \rho v (1+ \eps^2 k^2) \Big)
 +  \del_x \Big( ( 1 + \eps^2 k^2) \rho u v   \Big) 
 +  \del_y \Big(( 1 + \eps^2 k^2) \rho v^2 + k^2 \rho ( 1 - \eps^2 V^2) \Big) 
 = 0.
\endaligned
\ee
Note also in passing that taking $\eps=0$ leads us to the standard Euler equations for non-relativistic flows, as expected. 


\subsection{Heuristics on an expanding background}

In our numerical investigations, the asymptotics will be found to be $\rho \to 0$ and $u \to 0$, and in this limit the eigenvalues  in \eqref{equa:eigenv} becomes constant and of opposite sign, i.e. 
\bel{equa:eigenv2}
\lambda_1(u) \to - k, 
\qquad
\lambda_2(u)\to k,
\ee
and the flow propagates at (plus or minus) the sound speed. 
It is natural to search for suitably rescaled unknowns of the form 
\bel{eq: Reu01}
\rho =: t^{-\alpha} \rhot, \qquad 
u =: t^{-\beta} \ut, 
\ee
such that the new unknown have finite limits on the singularity when $t \to +\infty$. The exponents $\alpha, \beta >0$ are determined formally from \eqref{EulerAbc700}, by writing with the choice $a(t) = t^\kappa$ and $b=b(x)$ being a general coefficient, 
\bse
\bel{EulerAbc700-a-sim}
\aligned
& \del_t \Big( t^{-\alpha} \rhot (1 + \eps^4 k^2 t^{-2\beta} \ut^2) \Big) 
 + \del_x \Big( t^{-\alpha} \rhot t^{-\beta} \ut ( 1 + \eps^2 k^2) \Big)  
\simeq S_0, 
\\
& \del_t \Big( t^{-\alpha} \rhot t^{-\beta} \ut (1+ \eps^2 k^2) \Big)
 + \del_x \Big( ( 1 + \eps^2 k^2) t^{-\alpha} \rhot t^{-2\beta} \ut^2  
+ k^2 t^{-\alpha} \rhot ( 1 - \eps^2 t^{-2\beta} \ut^2) \Big) 
\simeq S_1,  
\endaligned
\ee
with 
\bel{EulerAbc700-b-sim}
\aligned
& S_0 \simeq - {\kappa \over t} t^{-\alpha} \rhot \Big( 1 + 3 \eps^2 k^2 + (1 - \eps^2 k^2) \eps^2 t^{-2\beta} \ut^2 \Big), 
\\
& S_1 \simeq  2 t^{-\alpha} \rhot \, 
\Big( k^2 \, {\del_x b \over b} (1 - \eps^2 t^{-2\beta} \ut^2) - {\kappa \over t}  (1+ \eps^2 k^2) t^{-\beta} \ut
\Big), 
\endaligned
\ee
\ese
therefore 
$$
\aligned
&  \del_t \big( t^{-\alpha}  \rhot  \big) 
+  k^2 \eps^4  \del_t \Big( t^{-\alpha-2\beta} \rhot  \ut^2 \Big) 
 + t^{-\alpha-\beta}  ( 1 + \eps^2 k^2)  \del_x \big( \rhot \ut \big) 
\\
& 
\simeq
 - t^{-\alpha-1}  \kappa \, \big( 1 + 3 \eps^2 k^2 \big)  \rhot
 - t^{-\alpha-1-2\beta}  \kappa \,  (1 - \eps^2 k^2) \eps^2 \rhot  \ut^2, 
\endaligned
$$
and 
$$
\aligned 
&   (1+ \eps^2 k^2)  \del_t \big( t^{-\alpha-\beta}  \rhot \ut \big)
 +  t^{-\alpha}  \del_x \Big( k^2 \rhot + t^{-2\beta} \rhot \ut^2   \Big) 
\\
&
\simeq 
2 t^{-\alpha} \rhot \, 
\Big(
 k^2 \, {\del_x b \over b}  
-   t^{-2\beta} k^2 \, {\del_x b \over b} \eps^2 \ut^2 
- t^{-\beta-1} \kappa  (1+ \eps^2 k^2)  \ut
\Big).
\endaligned
$$
The first equation leads us to the asymptotic equation $t^{-\alpha}  \del_t  \rhot = 0 $ for the density, 
thus $\rhot$ depends on $x$ only, as expected. The second equation gives us 
$
 t^{-\alpha}  \del_x \Big( k^2 \rhot \Big) 
= 
2 t^{-\alpha} \rhot \,  k^2 \, {\del_x b \over b}, 
$
therefore 
$\rhot(x) = C_1 b(x)^2$, 
which relates the asymptotic density profile to the underlying geometry. 

Postulating also that $\ut$ depends upon $x$ only, we can then rewrite our system in the form
\bse
\bel{EulerAbc45}
\aligned
&  -\alpha t^{-\alpha-1} b^2 
   - (\alpha+2\beta)  k^2 \eps^4 t^{-\alpha-2\beta-1} b^2  \ut^2  
 + t^{-\alpha-\beta}  ( 1 + \eps^2 k^2)  \del_x \big(b^2  \ut \big) 
\\
& 
\simeq
 - t^{-\alpha-1}  \kappa \, \big( 1 + 3 \eps^2 k^2 \big) b^2 
 - t^{-\alpha-1-2\beta}  \kappa \,  (1 - \eps^2 k^2) \eps^2 b^2  \ut^2, 
\endaligned
\ee
\be
\aligned
&   -(\alpha+\beta) (1+ \eps^2 k^2)  t^{-\alpha-\beta-1}  b^2 \ut  
 +  t^{-\alpha-2\beta}  \del_x \Big(  b^2 \ut^2   \Big) 
\\
&
\simeq 
2 t^{-\alpha} b^2 \, 
\Big( 
-   t^{-2\beta} k^2 \, {\del_x b \over b} \eps^2 \ut^2 
- t^{-\beta-1} \kappa  (1+ \eps^2 k^2)  \ut
\Big). 
\endaligned
\ee
\ese
It follows that $\rhot$ and $\ut$ must satisfy 
\bel{EulerAbc45}
\aligned
&  -\alpha t^{-\alpha-1} b^2  
 + t^{-\alpha-\beta}  ( 1 + \eps^2 k^2)  \del_x \big(b^2  \ut \big) 
\simeq
 - t^{-\alpha-1}  \kappa \, \big( 1 + 3 \eps^2 k^2 \big) b^2,
\\ 
&   -(\alpha+\beta) (1+ \eps^2 k^2)  t^{-\alpha-\beta-1}  b^2 \ut  
 +  t^{-\alpha-2\beta}  \del_x \big( b^2 \ut^2 \big) 
\\
&
\simeq 
2 t^{-\alpha} b^2 \, 
\Big( 
-   t^{-2\beta} k^2 \, {\del_x b \over b} \eps^2 \ut^2 
- t^{-\beta-1} \kappa  (1+ \eps^2 k^2)  \ut
\Big). 
\endaligned
\ee

We analyze this system as follows: if $\beta \in (0,1)$ then from the first equation we deduce that 
$
t^{-\alpha-\beta}  ( 1 + \eps^2 k^2)  \del_x \big(b^2  \ut \big) =0, 
$
thus for some constant $C_2$ 
we have $
\ut(x) = C_2 b^{-2}(x), x \in [0,1]. 
$
On the other hand,  if $\beta =1$ then  from the first equation we deduce that 
$$
-\alpha t^{-\alpha-1} b^2  
 + t^{-\alpha-1}  ( 1 + \eps^2 k^2)  \del_x \big(b^2  \ut \big) 
=
 - t^{-\alpha-1}  \kappa \, \big( 1 + 3 \eps^2 k^2 \big) b^2,
$$
thus 
$$
b^2(x)  \ut (x) = b^2(0)  \ut(0)
+ \Big( \alpha 
 -  \kappa \, \big( 1 + 3 \eps^2 k^2 \big) \Big) / ( 1 + \eps^2 k^2) \int_0^x b^2(y) dy, 
$$
which however is not a periodic function, as required. 
Finally, if $\beta >1$ then from the first equation we deduce that 
$
 -\alpha t^{-\alpha-1} b^2  
= 
 - t^{-\alpha-1}  \kappa \, \big( 1 + 3 \eps^2 k^2 \big) b^2.
$
This leads us to the condition
$
\alpha = \kappa \, \big( 1 + 3 \eps^2 k^2 \big), 
$
which appears to be the only consistent choice and, therefore, provides us with one of the two exponents we are searching for. 

From the second equation we obtain 
$$
  -(\alpha+\beta) (1+ \eps^2 k^2)  t^{-1}  b^2 \ut  
 +  t^{-\beta}  \del_x \big( b^2 \ut^2 \big) 
\simeq 
2  b^2 \, 
\Big( 
-   t^{-\beta} k^2 \, {\del_x b \over b} \eps^2 \ut^2 
- t^{-1} \kappa  (1+ \eps^2 k^2)  \ut
\Big). 
$$
If $\beta >1$, then we obtain 
$
(\alpha+\beta) = 2 \, \kappa;
$
thus 
$
\beta = \kappa ( 1 -  3 \eps^2 k^2).
$
In conclusion, the asymptotic limit \eqref{eq: Reu01} on an expanding background satisfies  
\bse
\label{equa-urho-exp}
\be
\rhob(x) = \lim_{t \to +\infty} \rhot(t, x) = C_1 b(x)^2, 
\qquad
\ub(t,x) = \lim_{t \to +\infty} \ut(x) = C_2 b^{-2}(x), \qquad x \in [0,1]
\ee
where $C_1>0$ and $C_2$ are constants and the exponents are
\be
\alpha = \kappa \, \big( 1 + 3 \eps^2 k^2 \big), \qquad
\beta = \kappa ( 1 -  3 \eps^2 k^2).
\ee
Since we want $\beta$ to be positive we assume that the sound speed is not too large, in the sense that
\be
k < {1 \over \eps \sqrt{3}}.
\ee 
\ese


\subsection{Heuristics on a contracting background}

When $t \to 0$, we will see numerically that $\rho \to +\infty$ while the velocity component $u \to \pm 1/\eps$ or $u \to 0$.
In the latter case eigenvalues becomes constant and opposite, as stated in \eqref{equa:eigenv2}, 
while in the limit $u \to \pm 1/\eps$ we find 
\bel{equa:eigenv3} 
\lambda_1(u) = {\pm 1/\eps - k \over 1 \mp \eps k} = \pm 1/\eps, 
\qquad
\lambda_2(u) = {\pm 1/\eps + k \over 1 \pm \eps k} =  = \pm 1/\eps, 
\ee
so that the eigenvalues coincide with the light speed. 
Considering first the case $u \to 1/\eps$, we postulate the ansatz 
\bel{equa:resc499} 
\rho =: |t|^{-\alpha} \rhot, \qquad 
u =: 1/\eps - |t|^{\beta} \ut, 
\ee
with positive exponents $\alpha, \beta >0$, in which $\rhot$ and $\ut$ should have a finite limit on the singularity. We obtain  
\bse
\bel{EulerAbc700-1D0}
\aligned
& \del_t \Big(|t|^{-\alpha} \rhot \big(1 + \eps^4 k^2 (1/\eps^2 - 2|t|^{\beta} \ut/\eps) \big)\Big) 
 + |t|^{-\alpha} \del_x \Big(  \rhot (1/\eps - |t|^{\beta} \ut) ( 1 + \eps^2 k^2) \Big)  
\simeq S_0, 
\\
& \del_t \Big( |t|^{-\alpha} \rhot (1/\eps - |t|^{\beta} \ut) (1+ \eps^2 k^2) \Big)
 + |t|^{-\alpha}  
\del_x \Big( \rhot ( k^2 +  1/\eps^2 - 2|t|^{\beta} \ut/\eps) \Big)  
\simeq S_1,  
\endaligned
\ee
with  (recalling that $t<0$) 
\be
\aligned
& S_0 \simeq
 - \kappa |t|^{-\alpha-1} \rhot \Big( 1 + 3 \eps^2 k^2 + (1 - \eps^2 k^2) \eps^2 (1/\eps^2 - 2 |t|^{\beta} \ut/\eps) \Big), 
\\
& S_1 \simeq
  2 |t|^{-\alpha} \rhot \, 
\Big( k^2 \, {\del_x b \over b} \big(1 - \eps^2 (1/\eps^2 - 2t^{\beta} \ut/\eps)\big) - {\kappa \over t}  (1+ \eps^2 k^2) (1/\eps - |t|^{\beta} \ut)
\Big). 
\endaligned
\ee
\ese
We rewrite these equations as 
\bse
\bel{E0-1D0} 
\aligned
& \del_t \Big(|t|^{-\alpha} \rhot \big(1 + \eps^2 k^2 - 2 \eps^3 k^2 |t|^{\beta} \ut) \big)\Big) 
 + |t|^{-\alpha} ( 1 + \eps^2 k^2)  \del_x \Big(  \rhot \big(1/\eps - |t|^{\beta} \ut \big) \Big)  
\simeq S_0, 
\\
& \del_t \Big( |t|^{-\alpha} \rhot (1/\eps - |t|^{\beta} \ut) (1+ \eps^2 k^2) \Big)
 + |t|^{-\alpha}  
\del_x \Big( \rhot \big( k^2 +  1/\eps^2 - 2|t|^{\beta} \ut/\eps \big) \Big)  
\simeq S_1,  
\endaligned
\ee
with
\be
\aligned
& S_0 \simeq
  - \kappa |t|^{-\alpha-1} \rhot \Big( 2 + 2 \eps^2 k^2  - 2 \eps(1 - \eps^2 k^2)  |t|^{\beta} \ut \Big), 
\\
& S_1 \simeq
  2 |t|^{-\alpha} \rhot \, 
\Big( 2 \eps k^2 \, {\del_x b \over b}  t^{\beta} \ut
 - {\kappa \over t}  (1+ \eps^2 k^2) (1/\eps - |t|^{\beta} \ut)
\Big). 
\endaligned
\ee
\ese 

We content here with the analysis at the order $|t|^{-\alpha-1}$ and we arrive at the conclusion that 
$$
\alpha |t|^{-\alpha-1} \rhot (1 + \eps^2 k^2)  
\simeq  
 \kappa |t|^{-\alpha-1}\rhot \Big( 2 + 2 \eps^2 k^2 \Big), 
$$
which shows that the exponent for the density variable is 
$
\alpha (1 + \eps^2 k^2)  
=
\kappa  \Big( 2 + 2 \eps^2 k^2  \Big), 
$
leading us to the coefficient 
\be
\alpha = 2 \, \kappa. 
\ee


\section{The finite volume methodology} 
\label{thesection:3}

\subsection{Discretization of the homogeneous system}  

\paragraph*{A four-state Riemann solver} 
\label{FVm}

Throughout this section, we assume $a(t) \equiv 1$ and we present a well-balanced scheme that takes the spatial effects of the geometry into account.
We begin by designing a finite volume method for the {\sl homogeneous version} of the model \eqref{equa:Euler-Form}, based on a four-state approximate Riemann solver defined as follows. 
We use the HLL methodology introduced by Harten, Lax, and van Leer \cite{HLL} (but in a generalized form) in order to approximate the solution of 
\bel{eq:Euler-Riemann}
\aligned
& \del_ t U + \del_x F(U) = 0,  \qquad t > 0, \, x \in \RR, 
\\
& U(0, x)=U_0(x) = 
\begin{cases} 
U_L, 
\quad 
& x  < 0, 
\\ 
U_R, 
\quad 
& x > 0,
\end{cases}
\endaligned
\ee
where $U_L$ and $U_R$ are constant states. The exact solution of \eqref{eq:Euler-Riemann} denoted by $\Rcal_\infty = \Rcal_\infty (\xi;  U_L, U_R)$ only depends on the self-similar variable 
$\xi :=x /t$, and $U_L$ and $U_R$. 

We introduce an approximate four-state Riemann solver defined as
\bel{r-s12}
\Rcal_4 = \Rcal_4(\xi;  U_L, U_R)
= 
\begin{cases}  
U_L,  
\qquad           
& \xi < \lambda_L, 
\\  
U_M^-, 		
& \lambda_L < \xi < 0,
\\  
U_M^+,    		
& 0 < \xi  < \lambda_R,
\\  
U_R,    		   
& \lambda_R < \xi,  
\end{cases}
\ee
where the two state vectors $U_M^-$ and $U_M^+$, as well as the approximate speeds $\lambda_L$ and $\lambda_R$ need to be defined. By assumption, $\lambda_L$ is chosen to be negative, and $\lambda_R$ to be positive.

For the (homogeneous) Euler system it is natural to choose the following expressions for the wave speeds: 
\bel{speed0}
\aligned
\lambda_R 
&= \max\Big(0, {u_L + k \over 1+ \eps^2 k u_L}, {u_R+ k \over 1+ \eps^2 k u_R} \Big), 
\qquad
\lambda_L
 =
\min\Big(0, {u_L- k \over 1- \eps^2 k u_L}, {u_R- k \over 1 -  \eps^2 k u_R} \Big).
\endaligned
\ee
The values of the intermediate states $U_M^-$ and $U_M^+$ (4 scalar unknowns) must be found in order to solve the approximate Riemann solver \eqref{r-s12}. Hence, we address some properties of the approximate Riemann solver in the following section.


\paragraph*{Consistency with the divergence part}
\label{sec: Consis}

The approximate Riemann solver $\Rcal_4(\xi;  U_L, U_R)$ defined in \eqref{r-s12}  should satisfy the consistency condition in divergence form that was proposed by Harten, Lax, and van Leer in \cite{HLL}.
The average of the approximate Riemann solver over a cell coincides with the average of the exact solution $\Rcal_\infty (\xi;  U_L, U_R)$ of the Riemann problem \eqref{eq:Euler-Riemann} over the same cell. 

We consider the Riemann problem \eqref{eq:Euler-Riemann} posed in a control volume $[x_L, x_R]\times[0, \tau]$, which satisfies
\be
x_L \leq \tau \lambda_L, \qquad x_R \leq \tau \lambda_R,
\ee
where $\lambda_L$ and  $\lambda_R$ are wave speeds given by \eqref{speed0}, $\tau$ is a chosen time. 
Therefore, the consistency condition can be written in the following form:
\bel{eq:consistency00}
\int_{x_L}^{x_R}\Rcal_4(\xi;  U_L, U_R) dx
= \int_{x_L}^{x_R} \Rcal_\infty (\xi;  U_L, U_R) dx.
\ee

After integrating \eqref{eq:Euler-Riemann} over the control volume $[x_L, x_R]\times[0, \tau]$, the exact Riemann solution satisfies 
\bel{eq:ExactRiemann}
 \int_{x_L}^{x_R}\Rcal_\infty (\xi;  U_L, U_R) dx
 = x_R U_R - x_L U_L - \tau(F(U_R)- F(U_L)).
\ee
We then use the expression \eqref{r-s12} of the approximate Riemann solver $ \Rcal_4(\xi;  U_L, U_R)$, and integrate it over the cell $[x_L, x_R]$:
\bel{eq:ApproximateRiemann}
\aligned
\int_{x_L}^{x_R} \Rcal_4(\xi;  U_L, U_R) dx 
&=  \int_{x_L}^{\tau \lambda_L} U_L \, dx + \int_{\tau \lambda_L}^0 U_M^- \, dx+  \int_0^{\tau \lambda_R} U_M^+ \, dx + \int_{\tau \lambda_R}^{x_R} U_R \, dx
\\
& = (\tau \lambda_L -x_L) U_L -\tau \lambda_L U_M^- + \tau \lambda_R U_M^+ +  (x_R- \tau \lambda_R) U_R.
\endaligned
\ee
Substituting \eqref{eq:ExactRiemann} and \eqref{eq:ApproximateRiemann} into \eqref{eq:consistency00} yields two scalar conditions:
\bel{eq: consist}
\lambda_R U_M^+ - \lambda_L U_M^-
 = \lambda_R U_R - \lambda_L U_L
- \Big(F(U_R) - F(U_L)\Big).
\ee

\begin{claim}
The approximate Riemann solver with two intermediate state vectors $U_M^+$ and $U_M^-$ is consistent with the conservation laws \eqref{eq:Euler-Riemann} if
\bel{eq: hll1}
\lambda_R U_M^+ - \lambda_L U_M^-
= \lambda_R U_R - \lambda_L U_L - \Big(F(U_R) - F(U_L)\Big),
\ee
where $\lambda_R$ and $\lambda_L$ are given by \eqref{speed0}.
\end{claim}

Thus far, two scalar conditions for  four scalar unknowns are obtained. Additional conditions for the approximate Riemann solver are needed. For the homogeneous model \eqref{eq:Euler-Riemann},  if we impose the continuity at the interface so that we only use one intermediate state, that is 
\bel{eq:UM1}
U_M^+ = U_M^-, 
\ee
which yields
\bel{eq:UM12}
U_M^+ = U_M^- = {\lambda_R U_R - \lambda_L U_L - \Big(F(U_R) - F(U_L)\Big) \over \lambda_R - \lambda_L}.
\ee
Hence, the approximate Riemann solver \eqref{r-s12} is fully defined. We will construct a Godunov-type scheme for the homogeneous system based on the approximate Riemann solver \eqref{r-s12}.


\paragraph*{A Godunov-type scheme}
\label{sec: G-type}

We now provide a finite volume discretization of the the  homogeneous system \eqref{eq:Euler-Riemann}.  The discretization in time and space is based on two mesh lengths
 $\Delta t$ and $\Delta x$ and relies on the cells $(x_{i -1/2}, x_{i +1/2})$ for $i = 0, 1, \cdots,$ 
with  
\be
x_i = i \Delta x, 
\qquad 
x_{i +1/2} = (i + 1/2) \Delta x.
\ee
Furthermore, $\Delta t$ and $\Delta x$ satisfy the CFL condition 
\bel{eq:CFL0}
{\Delta t  \over \Delta x }  \max(|\lambda_L |, |\lambda_R |) < {1 \over  2}.
\ee

We approximate the exact solution $U(t, x)$ of \eqref{eq:Euler-Riemann} by a constant value $U_i^n$ at time $t^n$.
The average value of $U(t, x)$ over the cell $(x_{i -1/2}, x_{i +1/2})$ is
\be
U_i^n ={1 \over \Delta x} \int_{x_{i - 1/2}}^{x_{i + 1/2}} U(t^n, x) dx,  
\qquad 
x \in (x_{i -1/2}, x_{i +1/2}). 
\ee
In particular, for the initial data we set
\be
U_i^0 = {1 \over \Delta x} \int_{x_{i - 1/2}}^{x_{i + 1/2}} U_0(y) dx.
\ee 

To approximate the solution of the homogeneous system \eqref{eq:Euler-Riemann}, we consider the Riemann problem at each interface $y = y_{i+1/2}$.
By using the approximate Riemann solver \eqref{r-s12} to the Riemann problem \eqref{eq:Euler-Riemann},  the updated solution at $t^{n+1}$ reads as:
\bel{eq:RS}
\aligned
U_i^{n+1}  
& = {1\over \Delta x} \int_{x_{i-1/2}}^{x_i} \Rcal_4\Big({x-x_{i-1/2}\over \Delta t};  U_{i-1}^n, U_i^n\Big) \, dx 
+ {1\over \Delta x} \int_{x_i}^{x_{i+1/2}} \Rcal_4\Big({x-x_{i+1/2}\over \Delta t}; U_i^n, U_{i+1}^n\Big) \, dx.
\\
&= U_i^n + {\Delta t \over \Delta x} 
\Big(
\lambda_{R, i-1/2}^n \big(U_{M,i-1/2}^{n, +} - U_i^n\big) -  \lambda_{L,i +1/2}^n \big( U_{M,i+1/2}^{n, -} - U_i^n \big)
\Big),
\endaligned
\ee
where $\lambda_{L,i +1/2}^n$ and $\lambda_{R,i-1/2}^n$ are the approximate wave speeds at each interface, and $U_{M,i-1/2}^{n, +}$ and $U_{M,i+1/2}^{n, -}$ are the intermediate states.


With the intermediate states \eqref{eq:UM12}, we obtain the following conservative form 
of the updated solution at $t^{n+1}$:
\bel{eq: FVH}
U_i^{n+1} = U_i^n - {\Delta t \over \Delta x} 
\Big( \FF( U_i^n,U_{i+1}^n) - \FF(U_{i-1}^n, U_i^n) \Big),
\ee 
where $\FF(U_L, U_R)$ is the numerical flux. If we impose the condition given by \eqref{eq:UM1}, the corresponding Riemann flux can be obtained as
\bel{eq:FluxM}
\FF(U_L, U_R) = {\lambda_R F(U_L) - \lambda_L F(U_R) \over  \lambda_R - \lambda_L} + {\lambda_R \lambda_L  (U_R - U_L) \over \lambda_R - \lambda_L}.
\ee


\subsection{A finite volume discretization of the non-homogeneous system}  

\paragraph*{A four-state Riemann solver} 

We return to the four-state approximate Riemann solver in  \eqref{r-s12} to approximate the solution of the non-homogeneous system \eqref{equa:Euler-Form}.
Recall that $a \equiv 1$, therefore, the source term  in \eqref{equa:Euler-Form} is independent of $t$. 
To evolve the solution in time, we consider the following Riemann problem
\bel{eq:Euler-Riemannq}
\aligned
& \del_ t U + \del_x F(U) = S(U, x),  \qquad t > 0, \, x  \in \RR, 
\\
& U(0, x)=U_0(x) = 
\begin{cases} 
U_L, 
\quad 
& x < 0, 
\\ 
U_R, 
\quad 
& x > 0,
\end{cases}
\endaligned
\ee
with two constant states $U_L$ and $U_R$.
Unlike the homogeneous case $(S(U,x) = 0)$, the exact solution of \eqref{eq:Euler-Riemannq} is no longer self-similar, which is denoted by $\Rcal_\infty = \Rcal_\infty (t, x, U_L, U_R)$.
For the the approximate Riemann solver of the non-homogeneous system, we choose the same expression as \eqref{r-s12}. 
The effects of the source term will appear in the construction of the intermediate states $U_M^{\pm}$.
The approximate speeds $\lambda_L$ and $\lambda_R$ are chosen from \eqref{speed0}. The two intermediate states $U_M^-$ and $U_M^+$ remain to be determined.


\paragraph*{The consistency condition}
\label{sec: Consnon}

To derive the values of intermediate states $U_M^-$ and $U_M^+$, we first consider the consistency condition for the approximate solver. 
We still consider the Riemann problem \eqref{eq:Euler-Riemannq}  posed in the control volume $[x_L, x_R]\times[0, \tau]$ as in Section \ref{sec: Consis}
Recall that the consistency condition has the following form introduced in Section 
\ref{sec: Consis}:
\bel{eq:consistency11}
\int_{x_L}^{x_R}\Rcal_4(\xi;  U_L, U_R) dx
= \int_{x_L}^{x_R} \Rcal_\infty (t, x;  U_L, U_R) dx,
\ee
where $\Rcal_4(\xi;  U_L, U_R)$ denotes the approximate Riemann solver of  \eqref{eq:Euler-Riemannq}, which has the form  given by \eqref{r-s12}, and $\Rcal_\infty (t, x;  U_L, U_R)$ denotes the  exact solution.

By integrating \eqref{eq:Euler-Riemannq} over the rectangle $[x_L, x_R] \times[0, \tau]$ the exact Riemann slover satisfies
\bel{eq:ExactRiemann1}
\aligned
\int_{x_L}^{x_R} \Rcal_\infty (t, y;  U_L, U_R) dx
& = x_R U_R - x_L U_L - \tau (F(U_R)- F(U_L))
+  \int_{x_L}^{x_R} \int_0^{\tau} S \Big(\Rcal_\infty (t, x;  U_L, U_R), y \Big) dt dx,
\endaligned
\ee
which is very similar to the case of the homogeneous system, see \eqref{eq:ExactRiemann}.

To simplify the notation, we let $\St(\delta, \tau; U_L, U_R)$ stand for the approximation of the source term, which is defined as follows: 
\bel{eq: souce}
\St(\tau, \delta; U_L, U_R) = {1 \over \Delta x}  {1 \over \Delta t}   \int_{x_L}^{x_R} \int_0^{\tau} S \Big(\Rcal_\infty (t, x;  U_L, U_R), x \Big) dt dx,
\ee
where we introduced the notation  
\bel{eq:dey}
\delta := x_R- x_L.
\ee
An easy calculation shows us that the approximate solver satisfies 
\bel{eq:ApproximateRiemann1}
\aligned
& \int_{x_L}^{x_R}\Rcal_4(\xi;  U_L, U_R) dx
  =   (\tau \lambda_L -x_L) U_L -\tau \lambda_L U_M^- + \tau \lambda_R U_M^+ +  (x_R- \tau \lambda_R) U_R.
\endaligned
\ee
Therefore, \eqref{eq:consistency11} can be rewritten as follows:
\bel{eq:HLLC1}
\lambda_R U_M^+ - \lambda_L U_M^-
= \lambda_R U_R - \lambda_L U_L - \big(F(U_R) - F(U_L)\big)  + \delta \, \St(\tau, \delta; U_L, U_R).
\ee
Recall that the original HLL scheme is based on an approximate Riemann solver
containing only one intermediate state, say denoted by $U_M$, which is obtained by solving the above equations when the source term is identically vanishing. 
For the following, it will be convenient to introduce
\bel{eq: HLL1M}
U_M = {\lambda_R U_R - \lambda_L U_L - \Big(F(U_R) - F(U_L)\Big) \over \lambda_R - \lambda_L}. 
\ee

\begin{claim}
The approximate Riemann solver with two intermediate state vectors $U_M^+$ and $U_M^-$ is consistent with the  balance laws \eqref{eq:Euler-Riemannq} in a chosen  control volume $[x_L, x_R]\times[0, \tau]$,
 if
\bel{eq:hll1}
\lambda_R U_M^+ - \lambda_L U_M^-
= (\lambda_R - \lambda_L) U_M+ \delta \, \St(\tau, \delta; U_L, U_R),
\ee
where $U_M$ is given by \eqref{eq: HLL1M}, $\lambda_R$ and $\lambda_L$ are given by \eqref{speed0}, $\delta$ is defined by \eqref{eq:dey}.
 $\St(\tau, \delta; U_L, U_R)$ is the approximation of the source term average.
\end{claim}

We refer this as the consistency conditions for the non-homogeneous system, where $U_M^+$ and $U_M^-$ are two intermediate state vectors of HLL solver. To summarize, we have obtained two scalar conditions \eqref{eq:hll1} for four scalar unknowns. Thus, we need to look for another two conditions to determine the intermediate states. 
We next construct a Godunov-type scheme based on the approximate Riemann solver defined in \eqref{r-s12}.


\paragraph*{A Godunov-type scheme}

With the finite volume  discretization in Section~\ref{sec: G-type}, we update the approximation of solution at  time $t = t^{n+1}$,  as follows:
\bel{eq:t+1}
\aligned
U_i^{n+1} 
& 
= {1 \over \Delta x} 
 \int_{x_{i-1/2}}^{x_{i+1/2}} U(t^{n + 1}, y) dx
\\
& = {1 \over \Delta x} \int_{x_{i-1/2}}^{x_i}\Rcal_4 \Big( {x- x_{i-1/2}\over \Delta t}; U_{i-1}^n, U_i^n \Big) \, dx 
+  {1 \over \Delta x} \int_{x_i}^{x_{i+1/2}} \Rcal_4 \Big( {x - x_{i+1/2}\over \Delta t}; U_i^n, U_{i+1}^n \Big) \, dx.
\endaligned
\ee
By using the approximation of Riemann solver $\Rcal_4(\xi;  U_L, U_R)$ given by \eqref{r-s12}  at each interface, we obtain
\bel{Taa0}
\aligned
U_i^{n+1} 
& 
= {1 \over \Delta x} 
\int_{x_{i-1/2}}^{x_{i-1/2} +  \lambda_{R,i-1/2}^n  \Delta t}  U_{M,i-1/2}^{n, +} dx 
+ {1 \over \Delta x}
\int_{x_{i-1/2} + \lambda_{R,i-1/2}^n \Delta t}^{x_{i +1/2} + \lambda_{L,i +1/2}^n \Delta t} U_i^n dx
+ \int_{x_{i +1/2} + \lambda_{L,i +1/2}^n \Delta t}^{x_{i +1/2}}  U_{M,i+1/2}^{n, -} dx,
\endaligned
\ee
which leads us to the following scheme:
\bel{Taa}
U_i^{n+1} =
U_i^n + {\Delta t \over \Delta x} 
\Big(
\lambda_{R,i-1/2}^n \big(U_{M,i-1/2}^{n, +} - U_i^n\big) -  \lambda_{L,i +1/2}^n \big( U_{M,i+1/2}^{n, -} - U_i^n \big)
\Big),
\ee
where $\lambda_{L,i +1/2}^n$ and $\lambda_{R,i-1/2}^n$ are the approximate wave speeds at each interface, which are given in \eqref{speed0}, 
$U_{M,i-1/2}^{n, +}$ and $U_{M,i+1/2}^{n, -}$ are the intermediate states to be defined.

Combining the integral consistency condition expressed in \eqref{eq:hll1}, we rewrite  \eqref{Taa} in the following conservative form:
\bel{Wbs0}
U_i^{n+1} = U_i^n - {\Delta t \over \Delta x} 
\big( \FF_{i + 1/2}^n - \FF_{i-1/2}^n \big) + {\Delta t \over 2}  \big( \St_{i + 1/2}^n + \St_{i - 1/2}^n \big),
\ee 
where $\FF_{i + 1/2}^n$ is the numerical flux at the interface $y_{i+1/2}$, defined by
\bel{FluxWbs1}
\aligned
\FF_{i + 1/2}^n = \FF(U_i^n, U_{i+1}^n)
& = {1 \over 2} \Big( F(U_i^n) + F(U_{i+1}^n)
\\
& \qquad
 + {\lambda_{L,i+1/2}^n} \big(U_{M,i+1/2}^{n, -} - U_i^n\big) + {\lambda_{R,i+1/2}^n} \big(U_{M,i+1/2}^{n, +} - U_{i+1}^n\big)\Big), 
\endaligned
\ee
and $\St_{i + 1/2}^n$ is the numerical source term at the interface $x_{i+1/2}$  given by
\bel{Nusource}
\St_{i + 1/2}^n = \St(\Delta t, \Delta x; U_i^n, U_{i+1}^n).
\ee

In summary, \eqref{Wbs0} defines a scheme, together with \eqref{FluxWbs1}, once we have chosen $U_{M,i+1/2}^{n,\pm}$ and $\St_{i + 1/2}^n$, such that 
\eqref{eq:hll1} holds.
For example,  we choose an approximate source term as follows:
\be
\hat{S}_{i+1/2}^n (\Delta t, \Delta x; U_i^n, U_{i+1}^n)  = 
\bigg(
\begin{array}{cccc}
0\\
{{\rho_i^n+ \rho_{i+1}^n \over \Delta x} \, 
 k^2 \, {\ln {b_{i+1} \over b_i}} \Big({1 - \eps^2  \big({u_i^n + u_{i+1}^n \over 2}\big)^2} \Big)}
\end{array} 
\bigg).
\ee
Moreover,  we impose the continuity at the interface so that we only use one intermediate state, that is 
\[
U_M^- = U_M^+.
\]
Together with \eqref{eq:hll1}, the intermediate states can be defined as follows:
\bel{eq:U22}
U_M^- = U_M^+ = U_M + { \hat{S}(\Delta t, \Delta x; U_L, U_R) \Delta x \over \lambda_R -\lambda_L },
\ee
where $U_M$ is given by \eqref{eq: HLL1M}.
Now, $U_{M,i+1/2}^{n,\pm}$ and $\hat{S}_{i + 1/2}^n$ have  been chosen. Thus, \eqref{Wbs0} defines a general scheme.


\section{A well-balanced finite volume scheme for cosmological fluid flows}
\label{thesection:4}

\subsection{The well-balanced property}

We require a well-balanced property for the scheme, that is, 
smooth steady state solutions of the Euler model should be preserved at the discrete level. The design of our well-balanced scheme is motivated by the earlier work \cite{MVCS2016} on the shallow water equations.
The steady state solutions satisfy the following system of ordinary differential equations:
\bse
\label{equa:static}
\bel{equilibrium-000}
\del_x F(U) = S(U, x),
\ee
where
\bel{equat-666}
\aligned
U  & = 
\bigg(
\begin{array}{cccc}
U_0 
\\
U_1
\\
\end{array} 
\bigg) 
 = \bigg(
\begin{array}{cccc}
\rho  (1 + \eps^4 k^2 u^2) 
\\
\rho u (1+ \eps^2 k^2)
\\
\end{array} 
\bigg),
\qquad \quad
F (U)  
& = 
\bigg(
\begin{array}{cccc}
F_0(U)
\\
F_1(U)
\\
\end{array} 
\bigg) 
= \bigg(
\begin{array}{cccc}
\rho u ( 1+ \eps^2 k^2)
\\
\rho (u^2 + k^2) 
\\
\end{array}\bigg),
\endaligned
\ee
and the source term is written as  
\be
S(U, x)  = \bigg(
\begin{array}{cccc}
S_0(U, x)\\
S_1(U,  x)\\
\end{array} 
\bigg) 
= \Bigg(
\begin{array}{cccc}
0
\\
2 \rho \, 
 k^2 \, {\del_x b \over b} (1 - \eps^2 u^2)
\end{array}
\Bigg).
\ee
\ese
We first assume that the rescaled density is positive such that $\rho_i^n > 0$ and $\rho_{i+1}^n > 0$.
From the scheme \eqref{Taa}, we observe that the solution is stationary that is $U_i^{n+1} = U_i^n$, 
if we have
\be
U_{M,i+1/2}^{n,-} = U_i^n, \qquad  U_{M,i-1/2}^{n, +} = U_i^n.
\ee

Therefore, after shifting $i \to i+1$ in the second condition, 
we look for the intermediate states $U_{M, i+1/2}^{n, -}$ and $U_{M, i+1/2}^{n, +}$ in the approximate Riemann solver satisfying
\be
U_{M, i+1/2}^{n, -} = U_i^n,  \qquad U_{M,i + 1/2}^{n, +} = U_{i+1}^n,
\ee
whenever $U_i^n$ and $U_{i+1}^n$ can be connected by a (continuous) steady state solution. After removing superscript $n$ and subscript $i\pm 1/2$ and shifting $i \to L$, $i+1 \to R$, we have the following property.

\begin{definition} [The well-balanced property] 
\label{Def: WB}
The scheme \eqref{Taa} is {\bf{well-balanced}} provided the intermediate states $U_{M}^-$ and $U_{M}^+$  are chosen to be
\be
U_{M}^- = U_L,  \qquad U_{M}^+ = U_R, 
\ee
whenever $U_L$ and $U_R$ can be connected by a (continuous) steady state solution to the system \eqref{equa:static}.
\end{definition}


\paragraph*{Construction of  $U_{1,M}^-$ and $U_{1,M}^+$}

We now look for the two intermediate state vectors denoted by
\be
 {U_M^-}= \bigg(
\begin{array}{cccc}
U_{0,M}^-
\\
U_{1,M}^-
 \\
\end{array}
\bigg),
\qquad 
{U_M^+} = \bigg(
\begin{array}{cccc}
U_{0,M}^+
\\
U_{1,M}^+
\\
\end{array}
\bigg),
\ee
which should satisfy the well-balanced property expressed in  Definition \ref{Def: WB}.

We begin by determining $U_{1,M}^-$ and $U_{1,M}^+$.
Note that the first component of the source term is $0$, and $F_0(U) = U_1$ for the non-homogeneous system. It is natural to impose the following 
condition:
\[
U_{1,M}^- = U_{1,M}^+,
\]
and we denote this value by $U_{1,M}^{\pm}$. Thus, with the consistency condition \eqref{eq:hll1} we obtain
\bel{U2}
\aligned
U_{1, M}^{\pm}
& = U_{1,M}^- = U_{1,M}^+ 
  = U_{1,M} + { \St_{1}(\Delta t, \Delta x;U_L, U_R)\Delta x \over \lambda_R - \lambda_L},
\endaligned
\ee
where $U_{1,M}$ is the second component of $U_M$ given by \eqref{eq: HLL1M}, $\lambda_R$ and $\lambda_L$ are given by $\eqref{speed0}$, the notation for the 
approximation of the source term $\St_{1}(\Delta t, \Delta x;U_L, U_R)$ is given by  \eqref{eq: souce}.

Therefore, the second component of intermediate state $U_{1,M}^- $ and  $U_{1,M}^+$ can be defined once the 
approximation of the source term is chosen.


\paragraph*{Construction of  $U_{0,M}^-$ and $U_{0,M}^+$}

We next turn to determine  $U_{0,M}^-$ and $U_{0,M}^+$.
We assume that  $U_L = U(\rho_L, u_L)$
and $U_R = U(\rho_R, u_R)$ can be connected by a steady state solution, with  $\rho_L > 0$ and $\rho_R > 0$. Under this assumption we derive the missing conditions of the scheme. 
We thus impose the following relations: 
\bel{Euler-Steady}
\aligned
\rho_R u_R (1+ \eps^2k^2) -\rho_L u_L (1+ \eps^2k^2)
& = 0,
\\
\rho_R \big(u_R^2 +k^2\big) - \rho_L \big(u_L^2 +k^2\big) 
& = \St_{1}(\Delta t, \Delta x;U_L, U_R) \Delta x. 
\endaligned
\ee

Based on \eqref{Euler-Steady}, we introduce the following state   
\bel{US1}
U_{1,LR} = \rho_R u_R (1+ \eps^2k^2) =\rho_L u_L (1+ \eps^2k^2).
\ee
By combining \eqref{US1} and the second equation in \eqref{Euler-Steady}, we obtain
\bel{US2}
\Bigg(k^2 - \Big({U_{1,LR} \over 1+ \eps^2k^2}\Big)^2 {1 \over \rho_{R} \rho_L}\Bigg)(\rho_R - \rho_L) 
=  \St_{1}(\Delta t, \Delta x;U_L, U_R)  \Delta x,
\ee
which is equivalent to
\bel{US3}
\Big(k^2 - u_L u_R\Big) \Big(\rho_R - \rho_L \Big) 
=  \St_{1}(\Delta t, \Delta x;U_L, U_R) \Delta x.
\ee


To determine $U_{0,M}^- $ and  $U_{0,M}^+$,
we need to look for the relation between $U_{0, L}$ and $U_{0, R}$ which can be obtained by utilizing \eqref{equat-666} and \eqref{US1} where
\bel{U1LR}
\aligned
U_{0,R} - U_{0, L}
& 
= \rho_R (1+ \eps^4 k^2 u_R^2) - \rho_L (1+ \eps^4 k^2 u_L^2) 
\\
&
= \Bigg(1 - \eps^4k^2 \Big({U_{1,LR} \over 1+ \eps^2k^2}\Big)^2 {1 \over \rho_R \rho_L}\Bigg)(\rho_R - \rho_L)
= \Big(1 - \eps^4k^2 u_L u_R\Big) \Big(\rho_R - \rho_L\Big).
\endaligned
\ee

Since $k\in(0, 1/\eps),$ therefore, $u_L, u_R \in (- 1/\eps, 1/\eps)$, and $1 - \eps^4k^2 u_L u_R > 0$. Thus, the relation between $\rho_R$ and $\rho_L$ can be written as
\bel{U-11}
\rho_R - \rho_L = {1\over 1 - \eps^4k^2 u_L u_R} \Big(U_{0, R} - U_{0, L}\Big).
\ee
Substituting \eqref{U-11} into \eqref{US3} yields a relation between $U_{0, L}$ and $U_{0, R}$,
\bel{bta}
{k^2 - u_L u_R \over 1 - \eps^4k^2 u_L u_R} \Big(U_{0,R} - U_{0,L}\Big) 
=   \St_{1}(\Delta t, \Delta x;U_L, U_R)  \Delta x.
\ee

We introduce a new function 
\bel{beta1}
\Lambda(u_L, u_R) = {k^2
 - u_L u_R 
 \over 1 - \eps^4k^2 u_L u_R}.
\ee
To ensure the well-balanced property, let us enforce the following condition for the approximate Riemann solver $U_{0,M}^- $ and  $U_{0,M}^+$:
\bel{eq:UMLR}
U_{0,M}^+ - U_{0,M}^- = U_{0,R} - U_{0,L}.
\ee
We thus can extend \eqref{bta} to the intermediate states $U_{0,M}^- $ and  $U_{0,M}^+$, which gives:
\bel{U1LR*}
\Lambda(u_L, u_R) (U_{0,M}^+ - U_{0,M}^-) =  \St_{1}(\Delta t, \Delta x;U_L, U_R)  \Delta x.
\ee 

The states $U_{0,M}^- $ and  $U_{0,M}^+$ are constructed as follows. 
\bei

\item If $ \Lambda(u_L, u_R) \neq 0$. From \eqref{eq:hll1} and \eqref{U1LR*}, we obtain
\bel{U12LR}
\aligned
U_{0,M}^+ &= U_{0,M} - {\lambda_{L} \over \lambda_R - \lambda_L}  { \St_{1}(\Delta t, \Delta x;U_L, U_R)  \Delta x \over \Lambda(u_L, u_R)}, 
\\
U_{0,M}^- &= U_{0,M} - {\lambda_{R} \over \lambda_R - \lambda_L}{ \St_{1}(\Delta t, \Delta x;U_L, U_R)  \Delta x \over \Lambda(u_L, u_R)},
\endaligned
\ee
where $\lambda_R$ and $\lambda_L$ are chosen in $\eqref{speed0}$.

\item If $ \Lambda(u_L, u_R) = 0$, that is $k^2 = u_L u_R$.

According to \eqref{US3}, we observe that $\St_{1}(\Delta t, \Delta x;U_L, U_R) =0$ in this case.
Instead of using  \eqref{U1LR*}, we utilize \eqref{U-11} and \eqref{eq:UMLR}, which gives 
\[ 
U_{0,M}^+ - U_{0,M}^- = (1 - \eps^4k^4) (\rho_R - \rho_L).
\]
Together with \eqref{eq:hll1}, we obtain 
\bel{U12LR0}
\aligned
U_{0,M}^+ &= U_{0,M} - {\lambda_{L} \over \lambda_R - \lambda_L} (1 - \eps^4k^4) (\rho_R - \rho_L), 
\qquad
U_{0,M}^- = U_{0,M} - {\lambda_{R} \over \lambda_R - \lambda_L} (1 - \eps^4k^4) (\rho_R - \rho_L).
\endaligned
\ee

\eei

\noindent In summary, the expressions for the intermediate states $U_M^- $ and  $U_M^+$ have been chosen, and we are only left with defining the discretization of the (second) source-term which has already been expressed in term of an auxiliary variable $U_{LR}$. 


\subsection{Property of the intermediate state}

We recall that the first component given in \eqref{equat-666} is non-negative if the weighed density $\rho \geq 0$. However, the expressions of the intermediate state \eqref{U12LR} and \eqref{U12LR0} may lead 
to no-positive  $U_{0,M}^- $ and  $U_{0,M}^+$. We thus need a modification to ensure that the positivity of  $U_{0,M}^- $ and  $U_{0,M}^+$.

We introduce a small parameter $\theta \geq 0$, which will be fixed in the numerical experiments later. We modify the intermediate state as follows:
\bei

\item If $U_{0,M}^+  \leq \theta$, we take $U_{0,M}^+ = \theta$ and, from \eqref{eq:hll1}, we get
\bse
\label{eq:Pos}
\bel{PosUM-}
U_{0,M}^-
=  \Big(
1- {\lambda_R \over \lambda_L}
\Big)  
U_{0, M} + {\lambda_R \over \lambda_L} \theta.
\ee

\item If $U_{0,M}^- \leq \theta$, we take $U_{0,M}^- = \theta$. We get
\bel{posUM+}
U_{0,M}^+
=  \Big(1- {\lambda_L \over \lambda_R }
\Big)  
U_{0, M}+ {\lambda_L \over \lambda_R} \theta .
\ee
\ese

\item Otherwise, we do not apply this positivity procedure.
\eei
Such a correction procedure ensures that $U_{0,M}^+ \geq 0$ and $U_{0,M}^- \geq 0$.

\begin{lemma}
Given any two constants $U_L$ and $U_R$, 
the intermediate states $U_{M}^+ $ and $U_{M}^-$ given by \eqref{U12LR0}-\eqref{eq:Pos} satisfy the following properties:

\bei

\item {\bf Consistency property.} The  intermediate state $U_M^+ $ and $U_M^-$ satisfy  the consistency condition \eqref{eq:hll1}.

\item {\bf Well-balanced property.} If $U_L$ and $U_R$ are  connected by a continuous steady state solution avoiding the sonic point, then 
\be
U_M^{-} = U_L,  \qquad U_M^+ = U_R.
\ee

\item {\bf Positivity property.} If $U_{0,L}$ and $U_{0,R}$ are non-negative, the intermediate states  $U_{0,M}^+ $ and $U_{0,M}^-$ given by 
\eqref{U12LR0}-\eqref{eq:Pos} are non-negative.
\eei
\end{lemma}


\subsection{Choice of the discretization of the source}
\label{source}

In this section, we give a suitable expression for the function $\St_{1}(\Delta t, \Delta x;U_L, U_R) $, when the exact source term $S_1$  is given as follows:
\bel{SourceNoa1}
S_1(U,  x)  = 
2 \rho k^2 { b_x \over b} (1 - \eps^2 u^2),
\qquad \text{ when } a(t) \equiv 1. 
\ee

We derive our expression by restricting attention first to data $U_L = U(\rho_L, u_L)$ 
and $U_R = U(\rho_R, u_R)$ that can be connected by smooth steady state solutions.
Recall that the steady state solution satisfy the ODE \eqref{equilibrium-000}, we have
\bel{EulerSS2}
\Big({U_{1,LR} \over 1+ \eps^2 k^2} \Big)^2 \Big({1\over \rho }\Big)_x + k^2 \rho_x = 2 \rho k^2 {b_x \over b} (1 - \eps^2 u^2),
\ee
where $U_{1,LR}$ has been introduced in \eqref{US1}:
\[
U_{1,LR} = \rho_R u_R (1+ \eps^2k^2) =\rho_L u_L (1+ \eps^2k^2).
\]

Dividing by $\rho$, \eqref{EulerSS2} becomes:
\bel{EulerSS22}
{1\over 2}\Big({U_{1,LR} \over 1+ \eps^2 k^2} \Big)^2 \Big({1\over \rho^2 }\Big)_x + k^2 (\ln \rho)_x = 2  k^2 {b_x \over b} (1 - \eps^2 u^2).
\ee

Integrating \eqref{EulerSS2} and \eqref{EulerSS22} over $[x_L, x_R]$, we obtain  the following algebraic relations:
\bel{Eulerflrw5}
\aligned
\Big({U_{1,LR} \over 1+ \eps^2 k^2} \Big)^2 \Big({1\over \rho_R} - {1\over \rho_L}\Big)
+ k^2 (\rho_R - \rho_L) 
& = \St_{1}(\Delta t, \Delta x;U_L, U_R) \Delta x, \\
{1\over 2}\Big({U_{1,LR} \over 1+ \eps^2 k^2} \Big)^2 \Big({1\over \rho_R^2} - {1\over \rho_L^2}\Big)
+ k^2  \ln{{\rho_R} \over \rho_L} 
&= 2 k^2 (1-\eps^2 u_{LR}^2) \ln{b(x_R) \over b(x_L)},
\endaligned
\ee
where the parameter $u_{LR}$ denotes  the approximation of the mean value of $u$, which is  consistent with $u$.
We choose the parameter $u_{LR}$ of the form
\bel{pmv}
u_{LR} = {u_L + u_R \over 2}.
\ee

Eliminating  $U_{1,LR}$ in the first equation of \eqref{Eulerflrw5} yields
\bel{SourceAppr}
\St_{1}(\Delta t, \Delta x;U_L, U_R)  \Delta x
= 2 k^2 \Big(1-\eps^2 u_{LR}^2 \Big) {2 \rho_L \rho_R \over \rho_L + \rho_R}  \ln{b(x_R) \over b(x_L) }
- 
k^2 {2 \rho_L \rho_R \over \rho_L + \rho_R}  \ln{\rho_R \over \rho_L}  
+
k^2 (\rho_R - \rho_L),
\ee
which provides us with one algebraic relation for the function $\St_{1}(\Delta t, \Delta x;U_L, U_R)$.

To shorten the notation, we introduce the following new functions:
\bel{eq: wrho}
w = w(\rho)= {1\over \rho}, \qquad B =  B(x) = \ln b(x).
\ee
Therefore, \eqref{SourceAppr} can be rewritten as follows:
\bel{eq:SourceAppr}
\St_{1}(\Delta t, \Delta x;U_L, U_R)  \Delta x
= {2 k^2 \over w_{LR}}\Big(1-\eps^2 u_{LR}^2 \Big)  ( B_R- B_L) 
- k^2 W_{LR},
\ee
where 
\bel{eq:wLR}
w_{LR} = {w_L + w_R \over 2},
\ee
and 
\bel{eq:WLR}
A_{LR} = A(w_L, w_R) = { 1 \over w_{LR} }\ln{w_L \over w_R}  
-
\Big({1\over w_R} - {1\over w_L}\Big).
\ee

The approximation of the source term has been determined when $U_L$ and $U_R$ are connected by the steady state solutions. However, we note that the expression \eqref{SourceAppr} only depends on the left and right states,  thus we can extend it for unsteady states.

\begin{definition} The following expression 
\bel{SourceT}
\St_{1}(\Delta t, \Delta x;U_L, U_R)  \Delta x
= {2 k^2 \over w_{LR}}\big(1-\eps^2 u_{LR}^2 \big)  ( B_R- B_L)  
- k^2 A_{LR}
\ee
is called the well-balanced source for the Euler system,
where $v_{LR}$ and  $w_{LR}$ are given by \eqref{pmv} and \eqref{eq:wLR}, $A_{LR}$ is  given by \eqref{eq:WLR},  $w$ and $B$ are introduced in \eqref{eq: wrho}.
\end{definition}

\begin{lemma}[Consistency of the approximation of the source term]
\label{ConsistencySouce}
For smooth solution $\rho = \rho(x)$ and $u = u(x)$, and given function $b= b(x)$, the expression of $\St_{1}(\Delta t, \Delta x;U_L, U_R)$ given by \eqref{SourceT} is consistent with the source term $S_1(U, x) = 2 \rho k^2 { b_x \over b} (1 - \eps^2 u^2)$.
\end{lemma}

\begin{proof}
For any smooth solution $\rho = \rho(x)$, $u = u(x)$ and any smooth function $b(x)$, we take $\rho_L = \rho(x)$ and $\rho_R = \rho(x + \Delta x)$,  $u_L = u(x)$ and $u_R = u(x + \Delta x)$.
With the functions $w$ and $B$ introduced in \eqref{eq: wrho},  we know that $w$ and $B$ are smooth functions. Hence, we let
$w_L = w(x)$, $w_R = w(x + \Delta x)$, $B_L = B(x)$ and $B_R = B(x + \Delta x)$. We then use the Taylor's expansion
\be
\aligned
& 
w_R = w + \del_x w \Delta x  + \mathcal{O}(\Delta x^2),
\qquad
u_R = u + \del_x u \Delta x  +  \mathcal{O}(\Delta x^2),
\\
&
B_R = B + \del_x B \Delta x  +  \mathcal{O}(\Delta x^2).
\endaligned
\ee

In \eqref{SourceT}, we have
\[
w_{LR} = { 2 w +  \mathcal{O}(\Delta x) \over 2}, \qquad u_{LR} = {2 u +  \mathcal{O}(\Delta x)\over 2},
\]
and
\[
B_R - B_L = B_x \Delta x +  \mathcal{O}(\Delta x^2).
\]

Moreover, for the second part $W_{LR}$ of \eqref{SourceT},
\bel{TaWLR}
\aligned
A_{LR} 
&
= { 1 \over w_{LR} }\ln{w_L \over w_R}  
-
\Big({1\over w_R} - {1\over w_L}\Big)
\\
&
= {2 \over 2 w +  \mathcal{O}(\Delta x)} \Big({w_x\over w} \mathcal{O}(\Delta x) + \mathcal{O}(\Delta x^2)\Big) - 
\Big({w_x \over w^2} \mathcal{O}(\Delta x) + \mathcal{O}(\Delta x^2)\Big)
  = \mathcal{O}(\Delta x^2). 
\endaligned
\ee

Above all, we have
\be
\St_{1}(\Delta t, \Delta x;U_L, U_R)  = 2k^2 \rho {b_x(x) \over b(x)}(1- \eps^2 u^2) +  \mathcal{O}(\Delta x),
\ee
which is consistent with $S_1(U, x)$.
\end{proof}


From the proof of Lemma~\ref{ConsistencySouce}, we note that the approximation of source term given by \eqref{SourceT} is  just consistent with the source term
for the smooth solution. Indeed, for discontinuous solutions, the second term $W_{LR}$ could not be consistent with $0$. To handle such an inconsistent term, we modify \eqref{SourceT} as follows:

\bei

\item If $\Big| {\St_{1}(\Delta t, \Delta x;U_L, U_R) \Delta x
\over {2 k^2 \over w_{LR}}\big(1-\eps^2 u_{LR}^2 \big)  ( B_R- B_L) }\Big| < \alpha$,  we use the relation \eqref{SourceT},
where $\alpha$ is a positive constant to be fixed in the numerical experiments.

\item Otherwise, we set 
\bel{Source-co}
\St_{1}(\Delta t, \Delta x;U_L, U_R) \Delta x
= {2 k^2 \over w_{LR}}\big(1-\eps^2 u_{LR}^2 \big)  ( B_R- B_L), 
\ee
\eei
Such modification ensure that the consistency of the approximation source term.

We note that the definition for the approximation of source term \eqref{SourceT} will not work  if $\rho_L = 0$ and $\rho_R = 0$.
Thus we take $\St_{1}(\Delta t, \Delta x;U_L, U_R)  \Delta x = 0$ if $\rho_L = 0$ and $\rho_R= 0$.  In this case, the intermediate states will be $0$. This completes the description of the algorithm (summarized in the next section). 


\subsection{A summary of our construction}

When $a(t)  \equiv 1$, the finite volume scheme for the Euler model takes the following form:
\bel{eq:1Wbs0}
U_i^{n+1} = U_i^n - {\Delta t \over \Delta x} 
\big( \FF_{i + 1/2}^n - \FF_{i-1/2}^n \big) + {\Delta t \over 2}  \big( \St_{i + 1/2}^n + \St_{i - 1/2}^n \big),
\ee 
where $\FF_{i + 1/2}^n$ is the numerical flux at the interface $y_{i+1/2}$, defined by
\bel{eq:FluxWbs1}
\aligned
\FF_{i + 1/2}^n 
& = \FF(U_i^n, U_{i+1}^n)
\\
& = {1 \over 2} \Big( F(U_i^n) + F(U_{i+1}^n)
 + {\lambda_{L,i+1/2}^n} \big(U_{M,i+1/2}^{n, -} - U_i^n\big) + {\lambda_{R,i+1/2}^n} \big(U_{M,i+1/2}^{n, +} - U_{i+1}^n\big)\Big).
\endaligned
\ee

Here, the intermediate states $U_{M,i+1/2}^{n, \pm}$ are chosen by \eqref{U2} and \eqref{U12LR}-\eqref{eq:Pos}. The wave speeds
are
\bel{eq:wavespeed1}
\aligned
\lambda_{R,i +1/2}^{n} 
&= \max\Bigg(0, {u_i^n+ k \over 1+ \eps^2 k u_i^n}, {u_{i + 1}^n+ k \over 1 +  \eps^2 k u_{i + 1}^n}\Bigg), 
\qquad
\lambda_{L,i +1/2}^{n} = \min\Bigg(0, {u_i^n- k \over 1- \eps^2 k u_i^n}, {u_{i + 1}^n- k \over 1 -  \eps^2 k u_{i + 1}^n} \Bigg).
\endaligned
\ee

And, $\St_{i + 1/2}^n$ is the numerical source term at the interface $y_{i+1/2}$  given by
\bel{eq:Nusource}
\St_{i + 1/2}^n = \St(\Delta t, \Delta x; U_i^n, U_{i+1}^n), 
\ee
which is defined by \eqref{SourceT} and \eqref{Source-co}.

We also  assume that the wave speeds satisfy the CFL condition:
\bel{CFL00}
 {\Delta t \over  \Delta x} \max\Big(|\lambda_{L,i+1/2}^{n}|, |\lambda_{R,i+1/2}^{n}|\Big)  < {1 \over 2 },
\ee
insuring  that no wave interaction takes place within one  time interval. The algorithm is thus based on the following steps:
\bei

\item  Firstly, for given the initial data $(\rho_i^n, u_i^n)$, we compute the conservative and flux variables: 
\bel{U11}
\aligned
& 
U_{0,i}^n = \rho_i^n  (1+ \eps^4k^2(u_i^n)^2), & \quad U_{1,i}^n = \rho_i^n u_i^n (1+ \eps^2 k^2),
\\
& 
F_{0,i}^n = \rho_i^n u_i^n (1+ \eps^2 k^2),  
& F_{1,i}^n = \rho_i^n ((u_i^n)^2 +k^2).
\endaligned
\ee

\item Secondly, by using the scheme \eqref{eq:1Wbs0}, the $U_{0,i}^{n+1}$  and $U_{1,i}^{n+1}$ values can be calculated.

\item Finally, we get  the primitive variables  $\rho_{i}^{n+1}$  and $u_{i}^{n+1}$ from 
the conservative variables $U_{0,i}^{n+1}$  and $U_{1,i}^{n+1}$. We have the following cases:

--- If $U_{0,i}^{n+1} = 0$, we take $\rho_i^{n+1} = 0$ and $u_i^{n+1} = 0$.

--- If $U_{0,i}^{n+1} \neq 0$ and $U_{1,i}^{n+1} = 0$, we take $\rho_i^{n+1} = U_{0,i}^{n+1}$ and $u_i^{n+1} = 0$.

--- If $U_{0,i}^{n+1} \neq 0$ and $U_{1,i}^{n+1} \neq 0$, we have 
\bel{Uv11}
\aligned
&
u_{i}^{n+1} = {1+ \eps^2k^2 - \sqrt{(1+ \eps^2k^2)^2 - 4\eps^4 k^2(U_{1,i}^{n+1}/U_{0,i}^{n+1})^2} 
\over 2\eps^4 k^2(U_{1,i}^{n+1}/U_{0,i}^{n+1})}, 
\qquad
\rho_{i}^{n+1} = {U_{0,i}^{n+1} \over 1+ \eps^4 k^2 (u_i^{n+1})^2 }.
\endaligned
\ee
\eei 


\subsection{Second-order accuracy in space}

In order to increase the accuracy in the numerical experiments, we construct a second-order scheme based on the above first-order scheme. We use a piecewise linear approximation of solution instead of the piecewise constant approximation of solution.

To shorten the notation, we denote by $q$ the state vector of primitive variables, i.e.  $q = (\rho, u)^T$.
For given primitive variables $q_i^n = (\rho_i^n, u_i^n)^T$ at the center of the cells $[x_{i-1/2}, x_{i+1/2}]$,  we now construct a piecewise linear approximation of solution $q_i^n(x)$, that is 
\bel{SPPL0}
q_i^n(x)= q_i^n + (x - x_i) \delta_i^n,
\ee
where $\delta_i^n$ is a local slope of the solution $q_i^n(x)$ in each cell.

We choose the minmod slope limiter under the following form:
\bel{SLMCL0}
\delta_i^n  = \begin{cases} \sgn(q_{i+1}^n - q_{i-1}^n) \min \Big( { |q_i^n - q_{i-1}^n| \over \Delta x},  { |q_{i+1}^n - q_i^n| \over \Delta x} \Big),  \quad
& \eta_i^n >0,
\\
0,
& \text{otherwise},
\end{cases}
\ee
where $\eta_j^n$ is defined by
\bel{eta1}
\eta_i^n = (q_{i+1}^n - q_i^n) (q_i^n - q_{i-1}^n).
\ee

Thus, the left and right values $q_{i+1/2}^{n,LR}$ at each interface $x = x_{i+1/2}$ can be obtained, as follows:
\bel{SWinterface0}
q_{i+1/2}^{n,L} = q_i^n + {\Delta x\over 2} \delta_i^n, 
\quad 
q_{i+1/2}^{n,R} = q_{i+1}^n - {\Delta x\over 2} \delta_{i+1}^n.
\ee

From \eqref{SWinterface0}, the conservative variables  $U_{i+1/2}^{n,L}$ and $U_{i+1/2}^{n,R}$ can be obtained. Finally, we use the following formula to update the approximate solution:
\bel{WbS20}
U_i^{n+1} = U_i^n - {\Delta t \over \Delta x} 
\big( \FF_{i + 1/2}^n - \FF_{i-1/2}^n \big) + {\Delta t \over 2}  \big( \St_{i + 1/2}^n + \St_{i- 1/2}^n \big),
\ee
where the numerical flux and source term are defined as
\bel{eq:Fluxsecond}
\FF_{i+ 1/2}^n = \FF(U_{i+1/2}^{n,L}, U_{i+1/2}^{n,R}),
\ee
and
\bel{eq:nusource2}
\St_{i+ 1/2}^n = \St(\Delta t, \Delta x; U_{i+1/2}^{n,L}, U_{i+1/2}^{n,R}).
\ee


The reconstructed scheme is not able to preserve all steady state solutions. We need to use the following reconstruction:
\bel{SWinterface1}
q_{i+1/2}^{n,L} = q_i^n +  {\Delta x\over 2} \delta_i^n \phi_i^n, 
\quad 
q_{i+1/2}^{n,R} = q_{i+1}^n -  {\Delta x\over 2} \delta_{i+1}^n \phi_{i+1}^n,
\ee
where $0 \leqslant \phi_i^n \leqslant 1$ is a parameter of the reconstruction. We note that when $\phi_i^n = 0$, the scheme is a first-order one, which preserves the steady state solution. Furthermore, when $\phi_i^n = 1$ the scheme is a standard second-order one. To define $ \phi_i^n$, we first need the following parameter:
\bel{psi1}
\psi_{i+1/2}^n = \rho_{i+1}^n (u_{i+1}^n + k^2) - \rho_i^n (u_i^n + k^2) -  \St_{1, i+ 1/2}^n \Delta x.
\ee
We can observe that if $q_{i+1}^n$ and $q_i^n$  are connected by the steady state solutions, $\psi_{i+1/2}^n$ will vanish.

Next, we define a function denoted by $\varphi_i^n$ to evaluate the deviation with respect to steady state solutions:
\bel{ups1}
\varphi_i^n = \Bigg\| \left(\begin{array}{c} \rho_{i+1}^n u_{i+1}^n -  \rho_i^n u_i^n  \\ \psi_{i+1/2}^n \end{array}\right)\Bigg\|_2
+ \Bigg\| \left(\begin{array}{c} \rho_i^n u_i^n - \rho_{i_1}^n u_{i_1}^n \\ \psi_{i-1/2}^n \end{array}\right)\Bigg\|_2.
\ee

We now define the parameter $\phi_i^n$ as follows:
\be
\phi_i^n =
 \begin{cases}
 0, \quad 
 & \varphi_i^n < m \Delta x,
 \\
 { \varphi_i^n - m \Delta x \over  M \Delta x - m \Delta x},
 
 &  m \Delta x < \varphi_i^n  < M \Delta x,
 \\
 1,
 & \varphi_i^n > M \Delta x,
 \end{cases}
\ee
where $0< m < M$ are numerical parameter will be given later.
If $q_{i+1}^n$ and $q_i^n$, $q_{i-1}^n$ are connected by a steady state solution, $\varphi_i^n$ will be very small, in this case the well-balanced scheme is used.  If $\varphi_i^n$ is
large enough, we use the  second-order  scheme.


\subsection{Taking the expanding or contracting effects $\boldsymbol {a(t)}$ into account}

We now consider the system with $a = a(t)$, and begin by providing 
a summary of the scheme presented in the previous section. 

We construct the HLL scheme for the full model \eqref{equa:Euler-Form}, that is, 
\bel{eq:3Euler-Formwitha}
\del_t U + \del_x F (U) = S(U,  x,  t).
\ee
We also recall the expressions of the conservative and flux variables: 
\be
\aligned
U  & = 
\bigg(
\begin{array}{cccc}
U_0 
\\
U_1 
\\
\end{array} 
\bigg) 
 = \bigg(
\begin{array}{cccc}
\rho  (1 + \eps^4 k^2 u^2) 
\\
\rho u (1+ \eps^2 k^2)
\\
\end{array} 
\bigg),
\qquad
F (U)  
& = 
\bigg(
\begin{array}{cccc}
F_0(U)
\\
F_1(U)
\\
\end{array} 
\bigg) 
= \bigg(
\begin{array}{cccc}
\rho u ( 1+ \eps^2 k^2)
\\
\rho (u^2 + k^2) 
\\
\end{array}\bigg),
\endaligned
\ee
while the source term has the form 
\be
S(U, x, t)  = \bigg(
\begin{array}{cccc}
S_0(U, x, t)\\
S_1(U,  x, t)\\
\end{array} 
\bigg) 
= \Bigg(
\begin{array}{cccc}
- {\del_t a \over a} \rho \Big( 1 + 3 \eps^2 k^2 + (1-\eps^2 k^2) \eps^2 u^2 \Big) 
\\
2 \rho \, 
\Big( k^2 \, {\del_x b \over b} (1 - \eps^2u^2) - {\del_t a \over a}  (1+ \eps^2 k^2) u
\Big)
\end{array}
\Bigg).
\ee

We split the source term into two parts as follows:
\be
S(U, x, t) =  P(U, x) + Q(U, x, t)
\ee
where
\be 
P(U, x)  = \bigg(
\begin{array}{cccc}
0\\
P_1(U, x)\\
\end{array} 
\bigg) 
= \Bigg(
\begin{array}{cccc}
0
\\
2 \rho k^2 \, {\del_x b \over b} (1 - \eps^2 u^2)
\end{array}
\Bigg),
\ee
and 
\be 
Q(U, x, t)  = \bigg(
\begin{array}{cccc}
Q_0(U, x, t)\\
Q_1(U, x, t)\\
\end{array} 
\bigg) 
= \Bigg(
\begin{array}{cccc}
- {\del_t a \over a} \rho \Big( 1 + 3 \eps^2 k^2 + (1-\eps^2 k^2) \eps^2 u^2 \Big) 
\\
-2 {\del_t a \over a} \rho \, 
 (1+ \eps^2 k^2) u
\end{array}
\Bigg).
\ee

We use the finite volume methodology  to discretize
the model \eqref{eq:3Euler-Formwitha}.
We denote by $\Ul_i$ and $\Sl_i$ the cell average of the solution $U(x, t)$ and the source term
$S(U, x, t)$ over a cell $[x_{i-1/2}, x_{i+1/2}]$ at time $t$:
\be
\Ul_i = {1\over \Delta x} \int_{x_{i-1/2}}^{x_{x + 1/2}} U(x, t) dx, 
\qquad
\Sl_i =  {1\over  \Delta x} \int_{x_{i-1/2}}^{x_{i + 1/2}}  S(U, x,  t) dx,
\ee
and
let $\Pl_i$ and $\Ql_i$ be the approximation of the average of source term $P(U, x)$ and $Q(U, x, t)$,
\be
\Pl_i = {1\over \Delta x} \int_{x_{i-1/2}}^{x_{i + 1/2}} P(U, x) dx, \qquad \Ql_j =  {1\over  \Delta x} \int_{x_{i-1/2}}^{x_{i + 1/2}}  Q(U, x,  t) dx.
\ee

Integrating \eqref{eq:3Euler-Formwitha} over the space cell $[x_{i-1/2}, x_{i+1/2}]$, we obtain the following semi-discrete equations:
\bel{Semi-discrete1}
{d\Ul_i \over dt} = - {1 \over \Delta x} 
\Big( \FF_{i+ 1/2} - \FF_{i-1/2}  \Big) +  \Pl_i + \Ql_i.
\ee

For the choice of the numerical flux $\FF_{i+1/2}$ and the source term $\Pl_i$, we use the well-balanced discretization presented in the previous section. 
The midpoint for $\Ql_i$  is chosen as follows:
\bel{SourceLF1}
\Ql_i = Q(U_i, x_i, t)  = 
 \Bigg(
\begin{array}{cccc}
- {\del_t a \over a} \rho_j \Big( 1 + 3 \eps^2 k^2 + (1-\eps^2 k^2) \eps^2 (u_i)^2 \Big) 
\\
-2 k^2 \rho_j \, 
{\del_t a \over a}  (1+ \eps^2 k^2) u_i
\end{array}
\Bigg).
\ee

To increase the accuracy in the numerical experiments, we use the piecewise linear reconstructions in space and a fourth-order Runge-Kutta solver in time.

  
\section{Global dynamics on a future-expanding background}
\label{thesection:5}

\subsection{Flows on a spatially homogeneous background in one space dimension} 

We present several numerical examples for the cosmological fluid equations \eqref{equa-27ab} in one space dimension. 
We assume $a(t) = t^\kappa$ and we begin with a uniform geometry $b(x) \equiv 1$, so that the source term reads  
\be
\aligned
& S_0 = - {\kappa \over t} \rho \Big( 1 + 3 \eps^2 k^2 + (1 - \eps^2 k^2) \eps^2 u^2 \Big), 
\qquad
S_1 =  -2 \rho \, {\kappa \over t}  (1+ \eps^2 k^2) u. 
\endaligned
\ee

\paragraph*{Test 1: Initial density with two constant states.}

We choose the following initial data posed at $t_0 = 1$ and defined in the domain $[0,1]$: 
\be
\big(\rho_0(x), u_0(x)\big) =  \begin{cases} (1, 0), \quad 
& \qquad   0 \leq x \leq 0.5, 
\\
(0.9, 0), &  \qquad 0.5 < x \leq 1. 
\end{cases}
\ee

This is a single jump discontinuity and we solve the initial value problem numerically with a periodic boundary condition. We choose here the exponent $\kappa =2$, and $t \in [1, +\infty)$, and the sound speed $k= 0.7$, and the light speed to be a unit. We denote by $N$ the total number of grid cells in space. 

In the first numerical result,  $N= 5000$ is chosen in order have a very fine grid. 
We view this solution as the ``reference solution". 
We compute the solution using $N = 100$ uniformly placed grid cells and compare it with the
reference solution and CFL $= 0.3$. The numerical solution obtained by using the standard HLL scheme at several order of accuracy: first-order in time and first-order in space ($1T1S$),
fourth-order in time and second-order in space ($4T2S$) at $t = 1.1$, which are presented in Figures~\ref{Fig: Euler-HLL-ex1}. We observe that the fourth-order in time and second-order in space
discretization significantly provides a better accuracy for the solution.

In Figure~\ref{Fig: Euler-HLL-ex10}, we plot the solutions for $N = 50, 100, 200, 400$ at $t = 1.1$, respectively. The results demonstrate that the approximate solutions approach the reference solution as 
$N$ increases.


\begin{figure}[htbp]
\centering
\epsfig{figure = 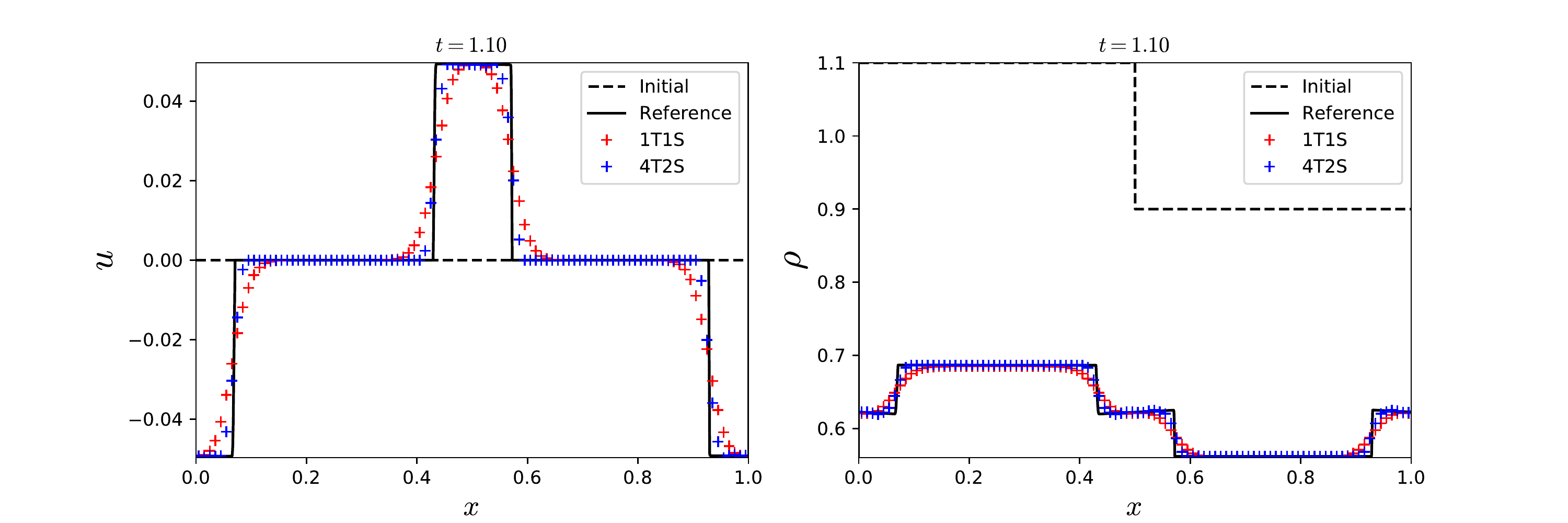,height = 2.2 in}  
\caption{First-order in time and first-order in space ($1T1S$) compared to fourth-order in time and second-order in space ($4T2S$).}
\label{Fig: Euler-HLL-ex1}
\end{figure}


\begin{figure}[htbp]
\centering
\epsfig{figure = 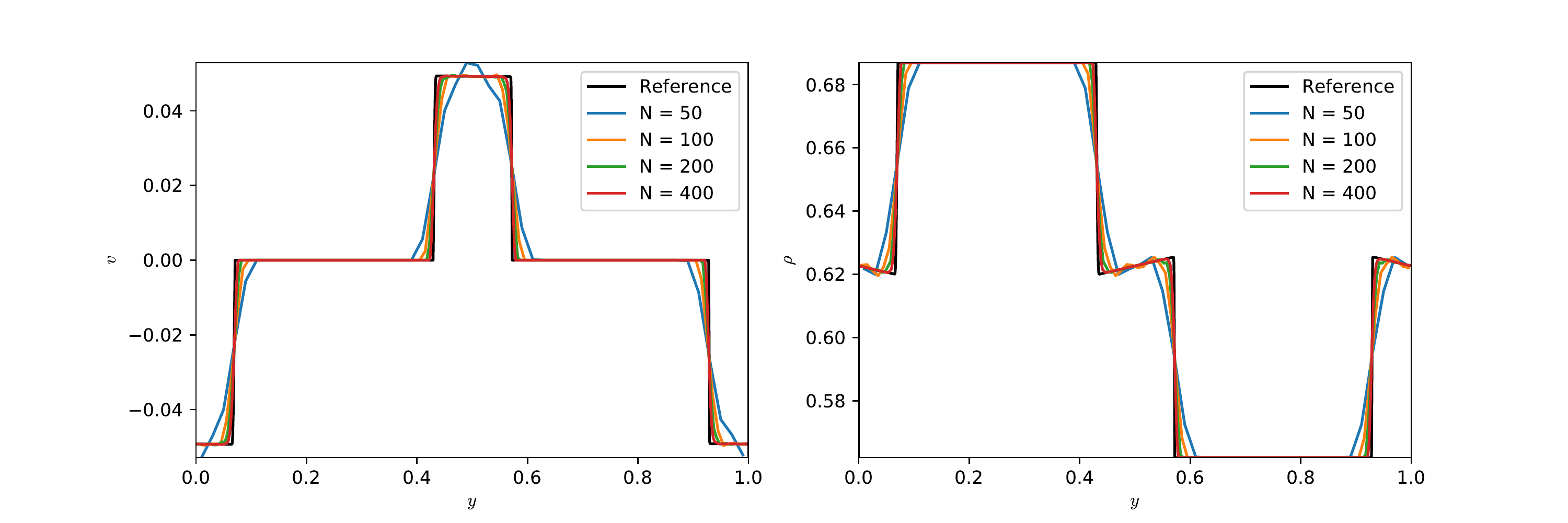,height = 2.2 in}  
\caption{Fourth-order in time and second-order in space scheme for different space grid cells.}
\label{Fig: Euler-HLL-ex10}
\end{figure}

\paragraph*{Test 2: Initial density with oscillations.} We now choose the initial data to be
\be
u_0(x)=0, \qquad  \rho_0(x) = 1+ \sin\Big({6\over7} \pi x\Big) \cos\Big({7\over 2} \pi x\Big),
\ee
which has a variable density $\rho$ and a vanishing velocity. In this test, the exponent is $\kappa =2$,  light speed is chosen to be unit, the sound speed is $k = 0.5$ with CFL $= 0.3$. 

The evolution of the solution $u$ and $\rho$ as $t$ increases is shown in Figures~\ref{fig:exp1} to \ref{fig:exp2}, where we use $N =500$.
We observe that the solution $\rho \to 0$ and $u \to 0$ as $t$ increases. 
Moreover, the figures show that initially the solution $u$ evolves from the initial data in to a sawtooth wave, 
which is a piecewise linear function. This transition happens on a relatively short scale.  Then, the waves interact until there are only two N-waves left, that structure preserves for a very long time.

\begin{figure}
\centering 
\begin{minipage}{0.5\linewidth}
  \centerline{\includegraphics[width=5.6in]{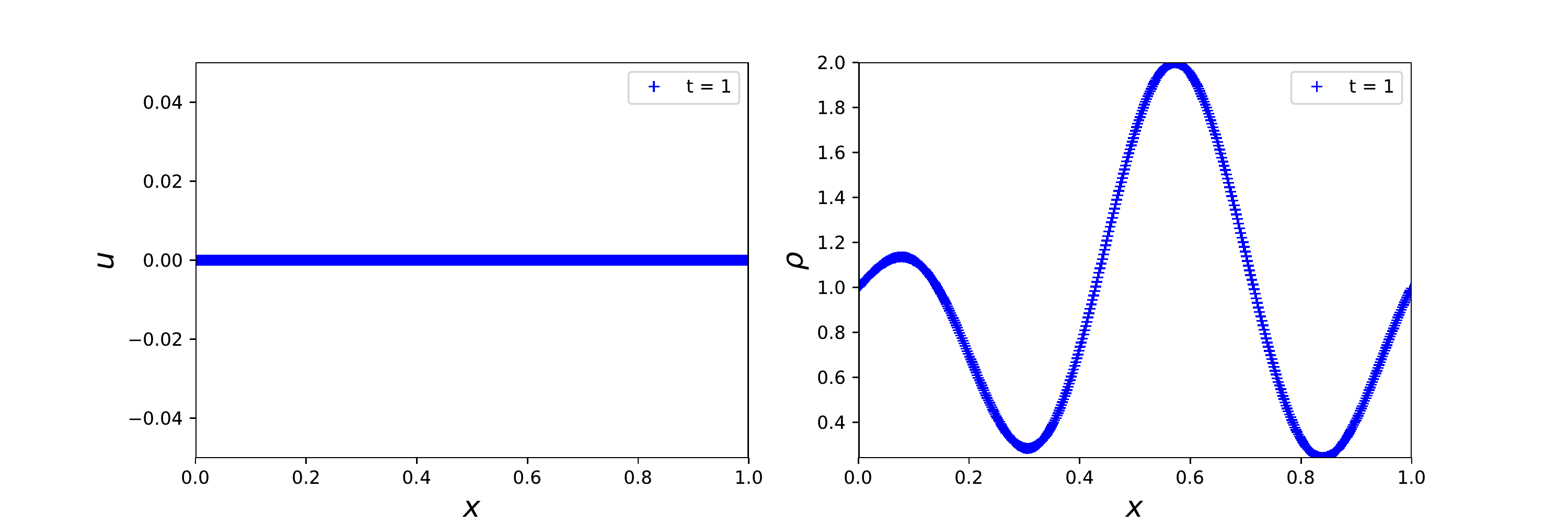}} 
\end{minipage}
\vfill
\begin{minipage}{.5\linewidth}
  \centerline{\includegraphics[width=5.6in]{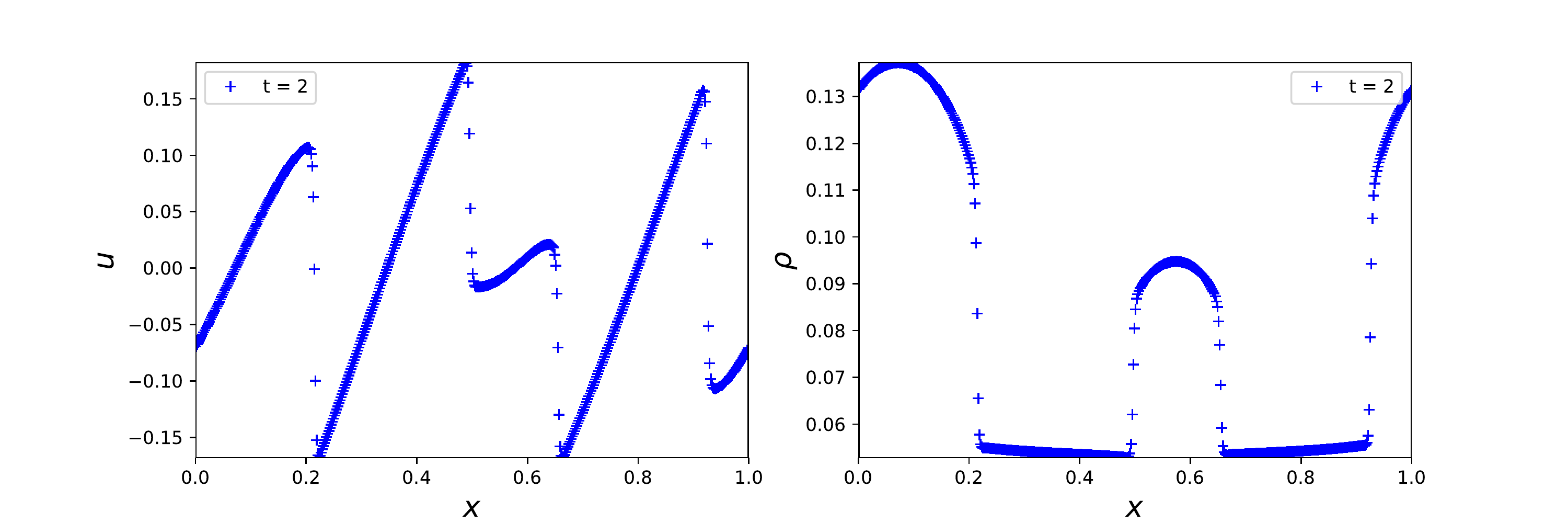}} 
\end{minipage}

\vfill
\begin{minipage}{.5\linewidth}
  \centerline{\includegraphics[width=5.6in]{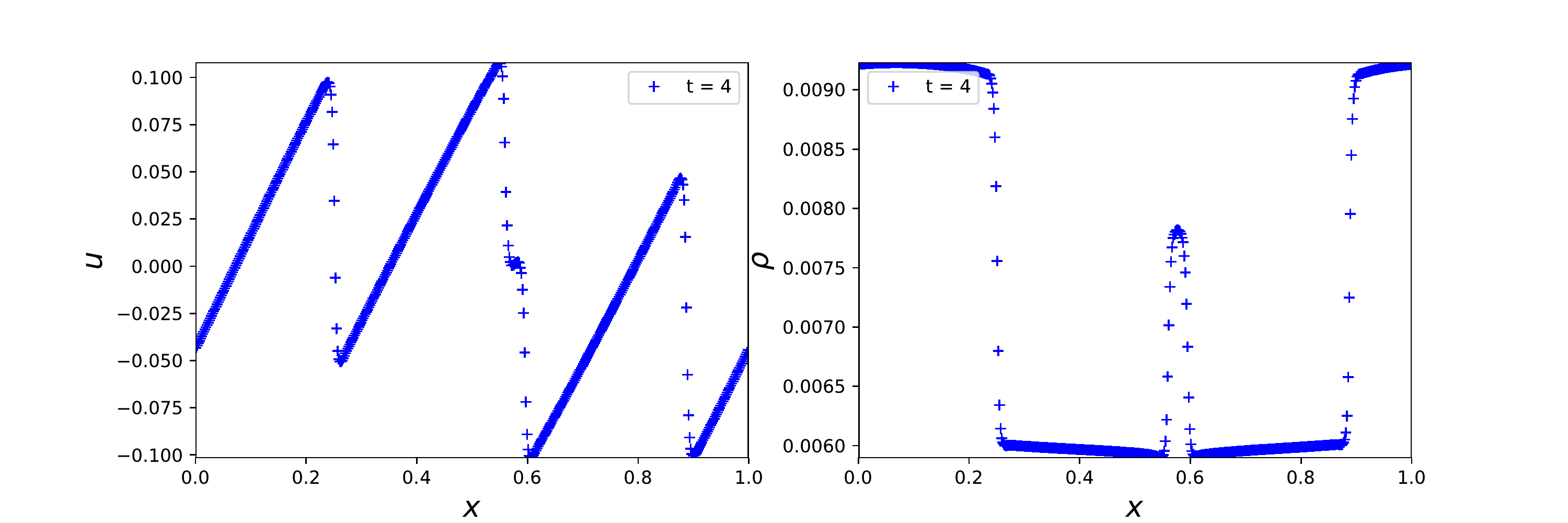}} 
\end{minipage}
\vfill
\begin{minipage}{.5\linewidth}
  \centerline{\includegraphics[width=5.6in]{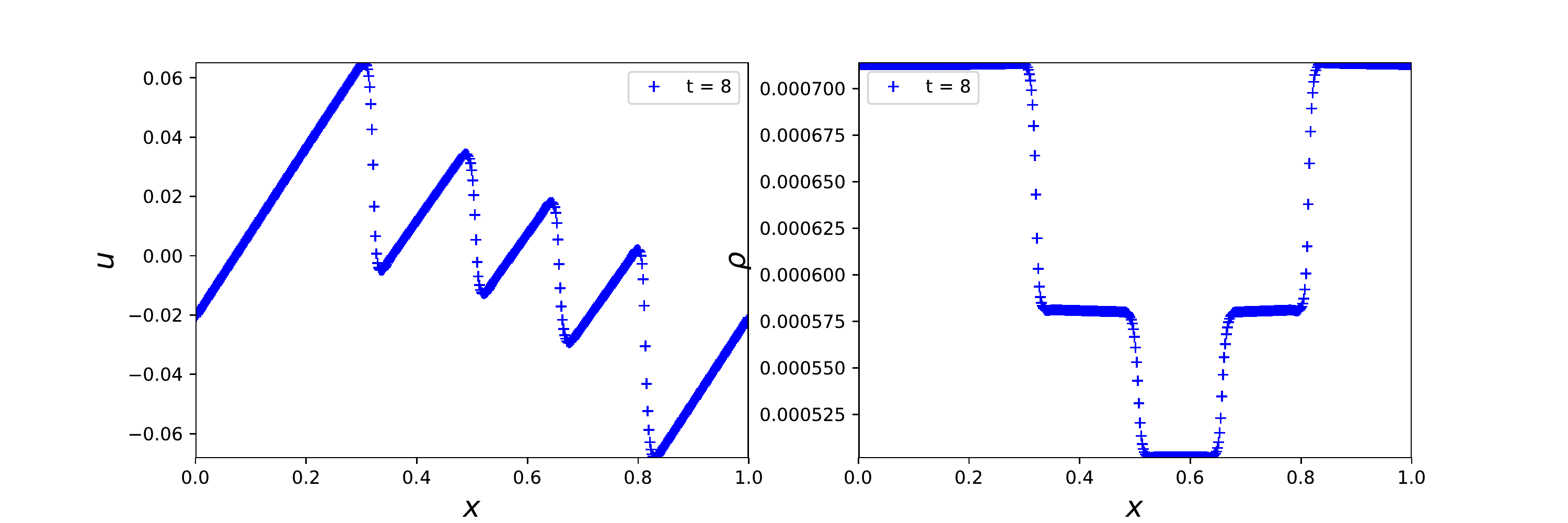}} 
\end{minipage}

\caption{The evolution of solution $u$ and $\rho$ as $t$ increases on an expanding background.}
\label{fig:exp1}
\end{figure}

\begin{figure}
\centering 
\begin{minipage}{0.5\linewidth}
  \centerline{\includegraphics[width=5.6in]{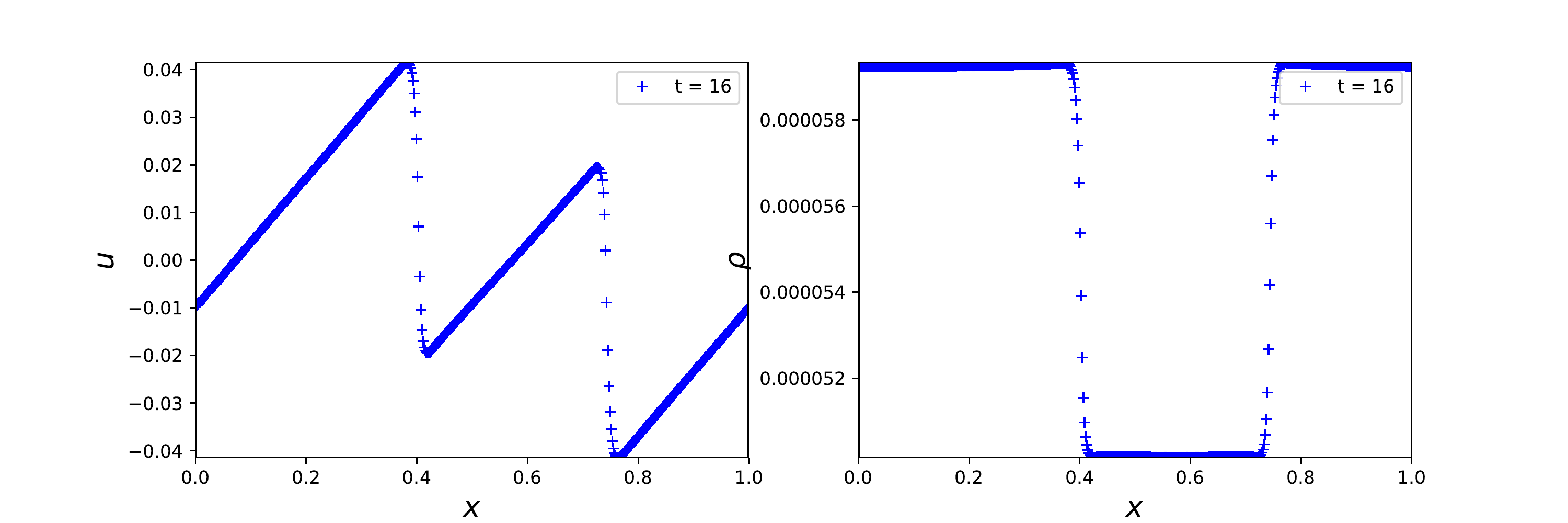}} 
\end{minipage}
\vfill
\begin{minipage}{.5\linewidth}
  \centerline{\includegraphics[width=5.6in]{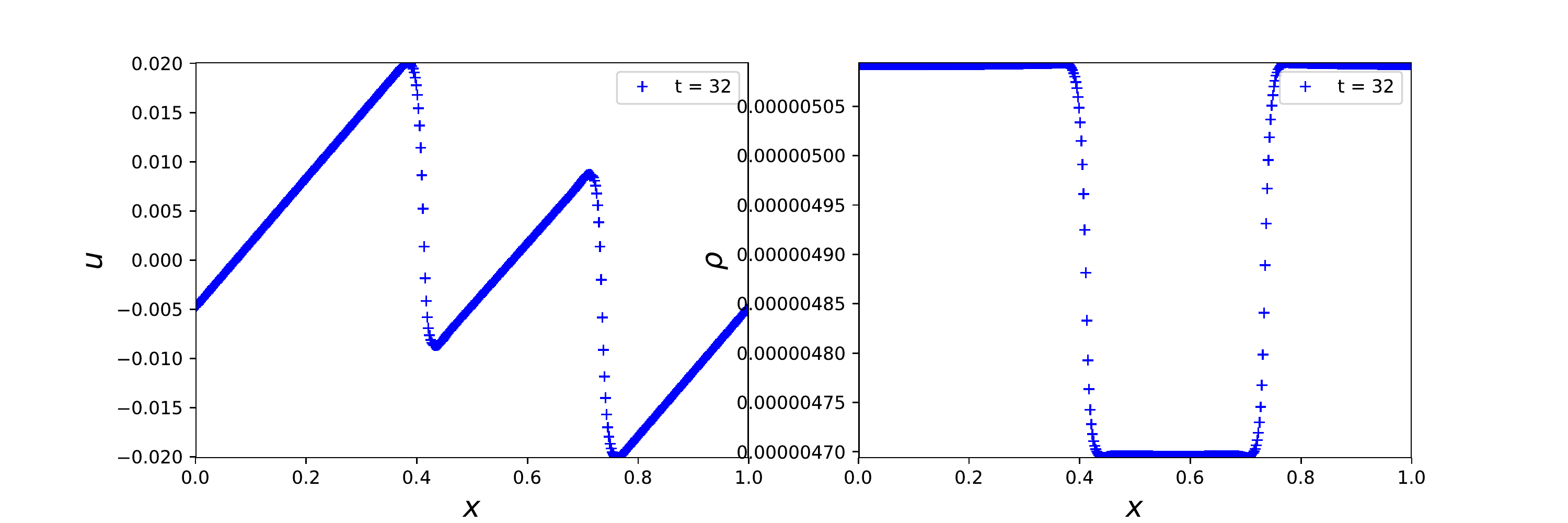}} 
\end{minipage}

\vfill
\begin{minipage}{.5\linewidth}
  \centerline{\includegraphics[width=5.6in]{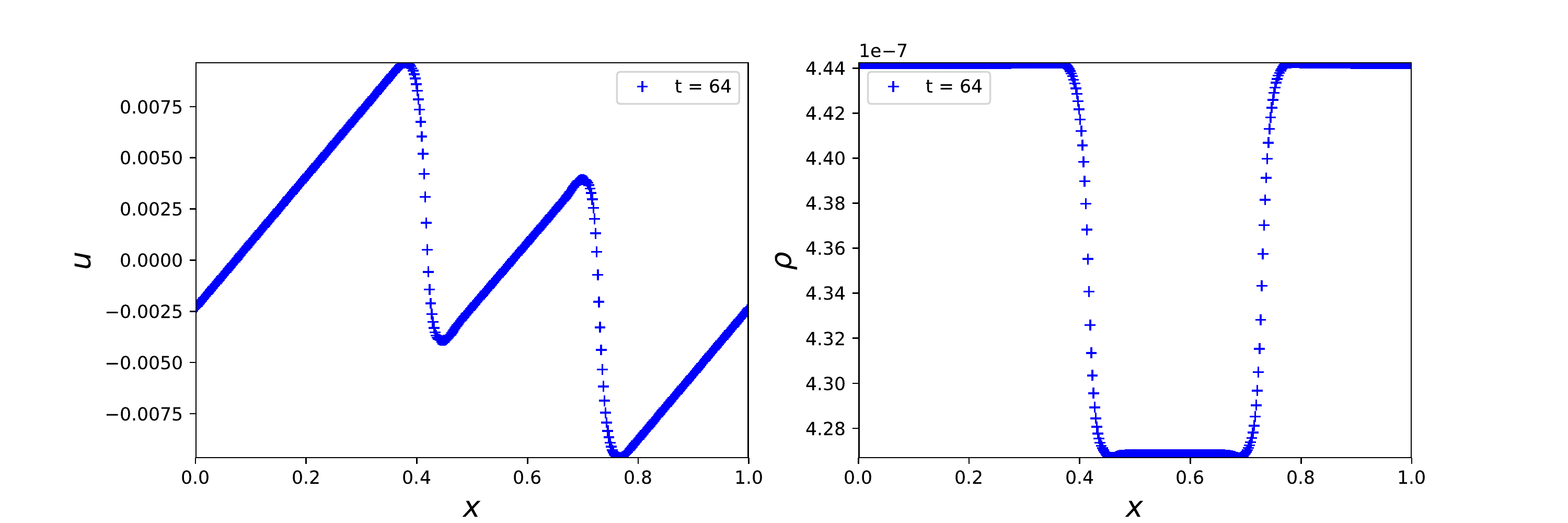}} 
\end{minipage} 

\caption{The evolution of solution $u$ and $\rho$ as $t$ increases on an expanding background.}
\label{fig:exp2}
\end{figure}


\paragraph{Rescaling the numerical solution}

We now display the rescaled solution $\ut$ and $\rhot$ defined in \eqref{eq: Reu01}; see Figure~\ref{fig:exp2}. We observe that the asymptotic solution only contains two linear pieces with two jumps,
and eventually converges to a constant.

\begin{figure}
\centering 
\begin{minipage}{0.5\linewidth}
  \centerline{\includegraphics[width=5.6in]{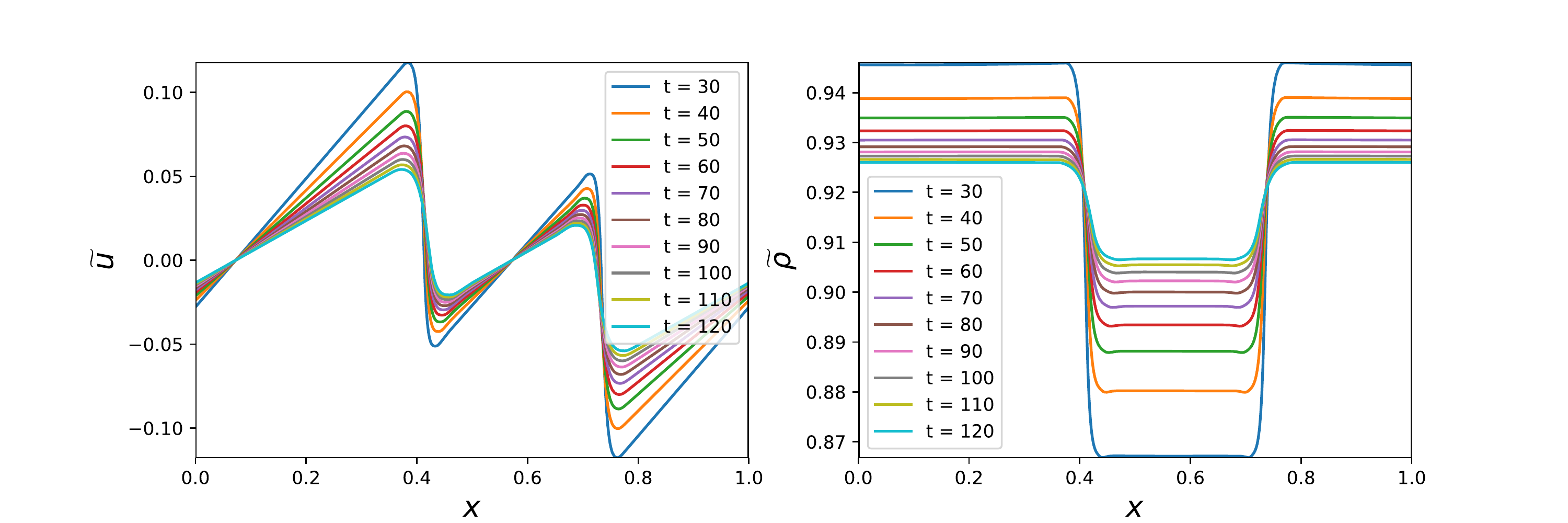}} 
\end{minipage}
\vfill
\begin{minipage}{.5\linewidth}
  \centerline{\includegraphics[width=5.6in]{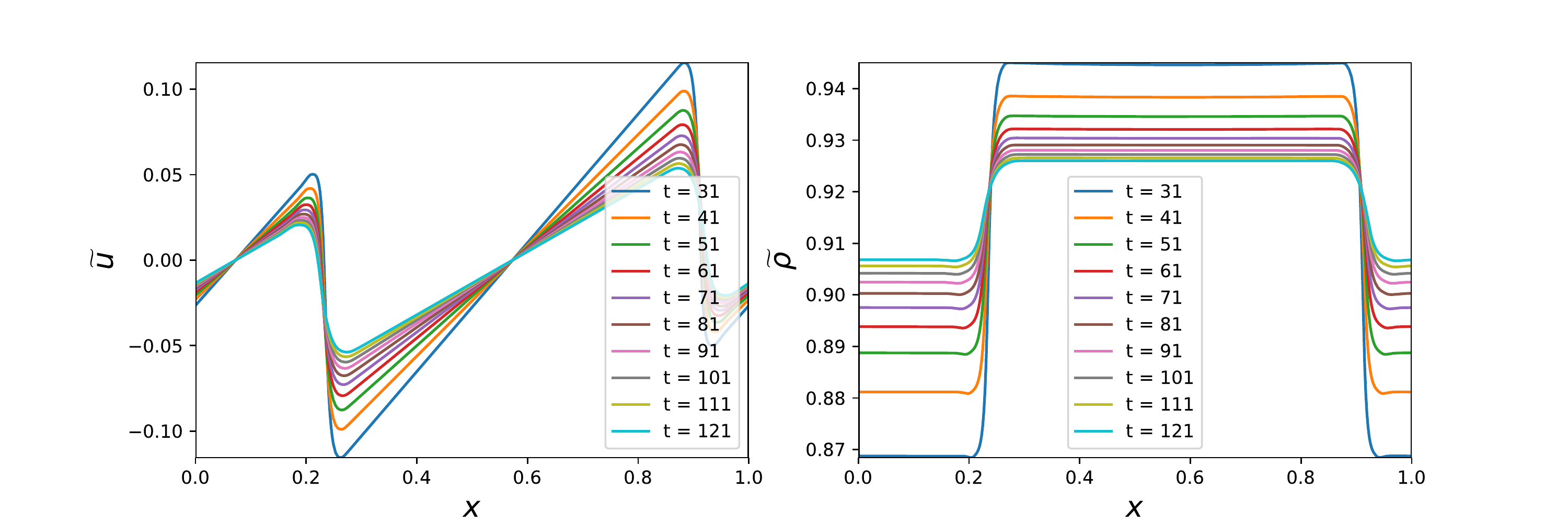}} 
\end{minipage} 

\vfill
\begin{minipage}{.5\linewidth}
  \centerline{\includegraphics[width=5.6in]{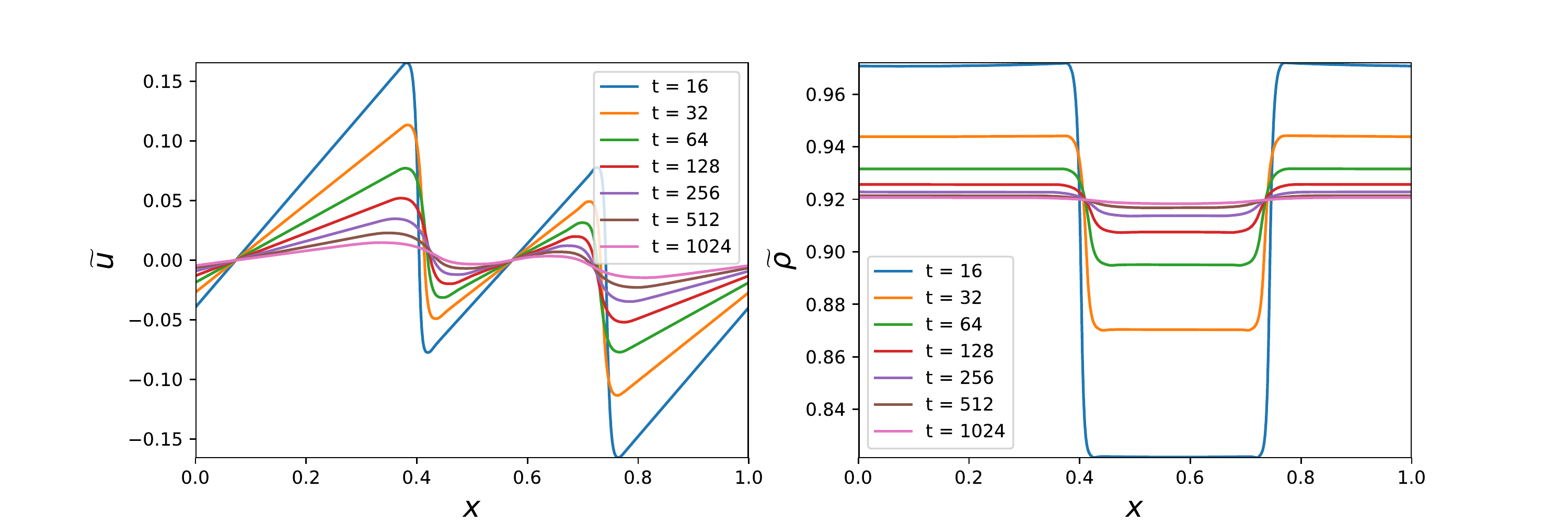}} 
\end{minipage} 

\caption{Rescaled solution on an expanding background with $k = 0.5$.}
\label{fig:exp2}
\end{figure}


\subsection{Flows on a spatially homogeneous background in two space dimensions} 

The proposed algorithm in one space dimension is applied direction by direction on a Cartesian mesh. We checked that our two-dimensional code is ``consistent'' with the results provided in one dimension, and in the typical test chosen above a very similar asymptotic rescaled density is recovered. We then performed genuinely two-dimensional tests, as now presented. 

Similar to the one-dimensional tests we assume that $a(t) = t^\kappa$ and we begin with a uniform geometry $b(x,y) \equiv 1$, so that the source term reads  
\bel{}
\aligned
& S_0 = - {\kappa \over t} \rho \Big( 1 + 3 \eps^2 k^2 + (1 - \eps^2 k^2) \eps^2 V^2 \Big), 
\qquad
 S_1 =  -2 \rho  {\kappa \over t}  (1+ \eps^2 k^2) u,
\qquad
 S_2 = -2 \rho {\kappa \over t}  (1+ \eps^2 k^2) v. 
\endaligned
\ee

\paragraph*{Test 1: Symmetrical initial data.}

We choose the following initial data posed at $t_0 = 1$ and defined in the domain $[0,1]\times[0,1]$: 
\be
\rho_0(x,y) = 0.1 + 0.1 e^{-20 {(x - 0.5)}^2 - 20 {(y - 0.5)}^2}, \qquad u_0(x,y) = 0, \qquad v_0(x,y) = 0.
\ee

\begin{figure}[htbp]
\centering
 {\includegraphics[height=2.1in]{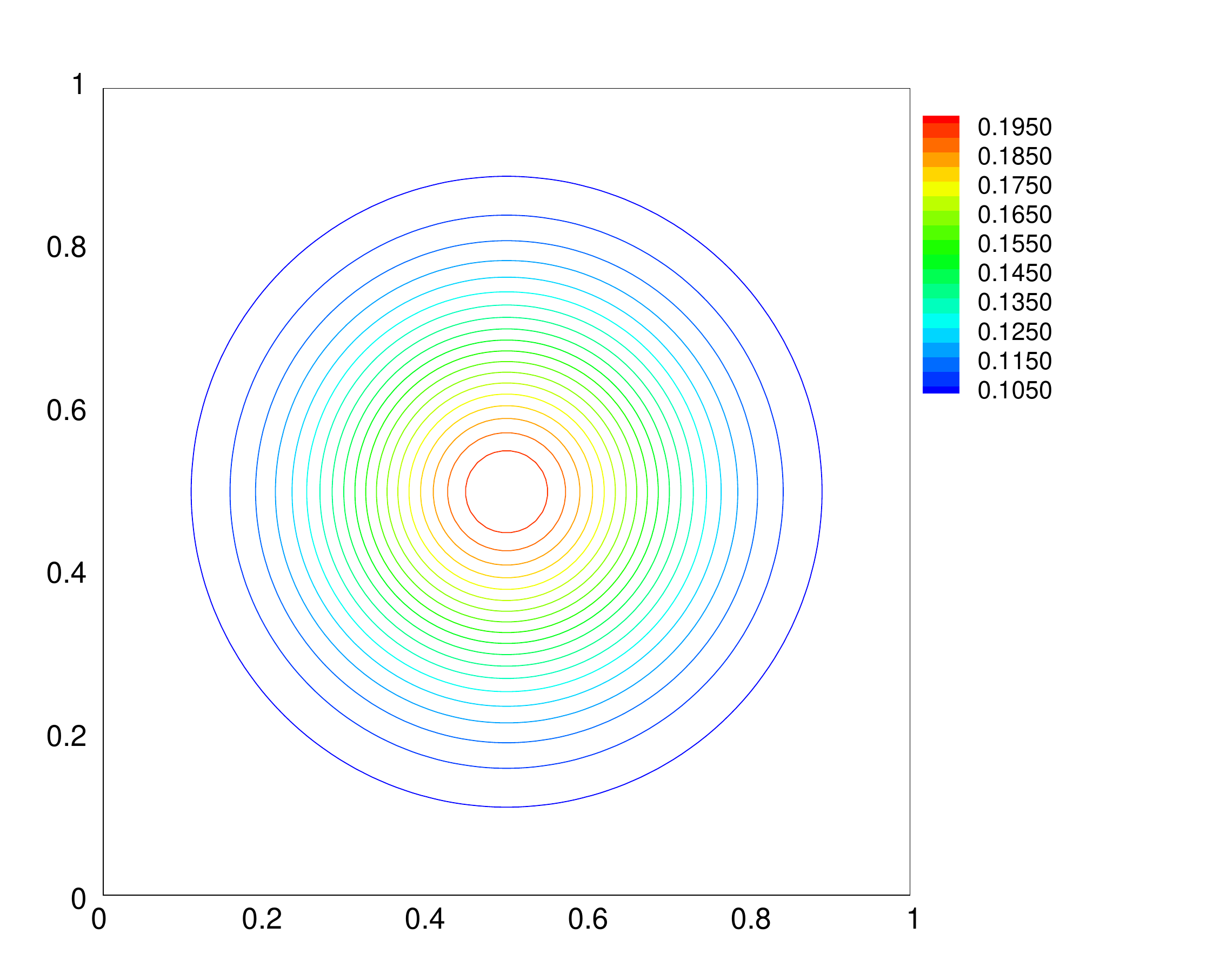}}
    
\caption{Initial data for the spatially homogenous background in two space dimensions.}
\label{2d-initial1}
\end{figure}

We choose this initial data (Figure~\ref{2d-initial1}) to be able to observe that the solution preserves its symmetry. In addition, in all two-dimensional tests from here, the exponent $\kappa =2$, the sound speed $k= 0.5$, CFL $= 0.5$, the light speed to be a unit, and the grid is $[100\times100]$. For the expanding test cases $t \in [1, +\infty)$ is chosen. Furthermore, we solve the 2-D system with a third-order strong stability preserving (SSP) Runge-Kutta solver.

In Figure~\ref{a(t)-ex}, we plot the rescaled solution $\rhot$ and velocity magnitude $V$ at $t=8$, $16$, $50$, $60$, respectively. The results which are obtained by the standard HLL scheme demonstrate that the solution $\rho \to 0$ and $V \to 0$ as $t$ increases. The rescaled solution $\rhot$ also converges to a periodically constant state. 

\begin{figure}[htbp]
\centering
 {\includegraphics[height=2.1in]{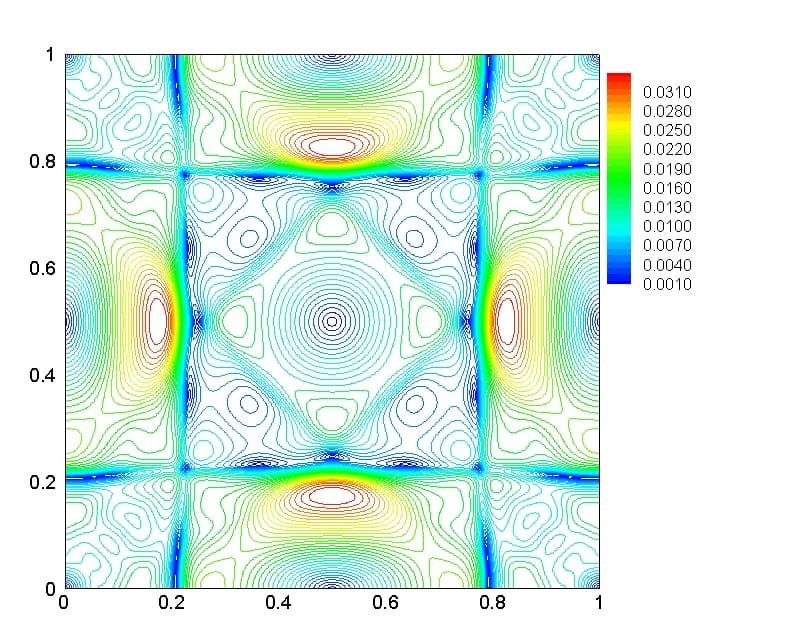}}
 {\includegraphics[height=2.1in]{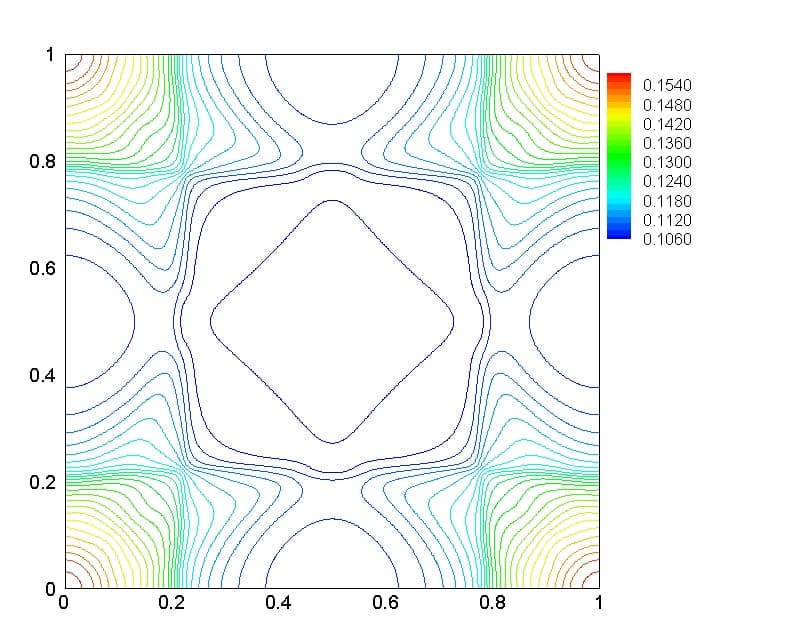}} \\
  
\centering
{\includegraphics[height=2.1in]{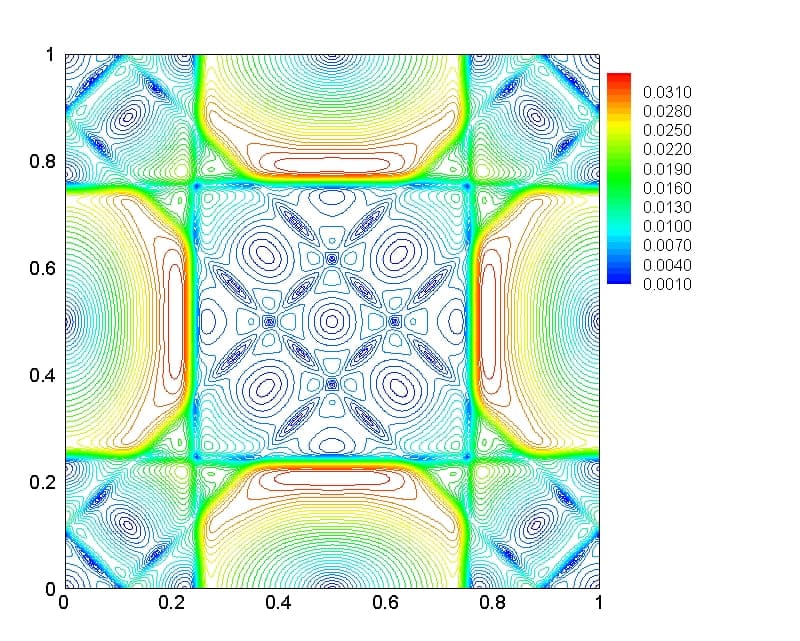}}
 {\includegraphics[height=2.1in]{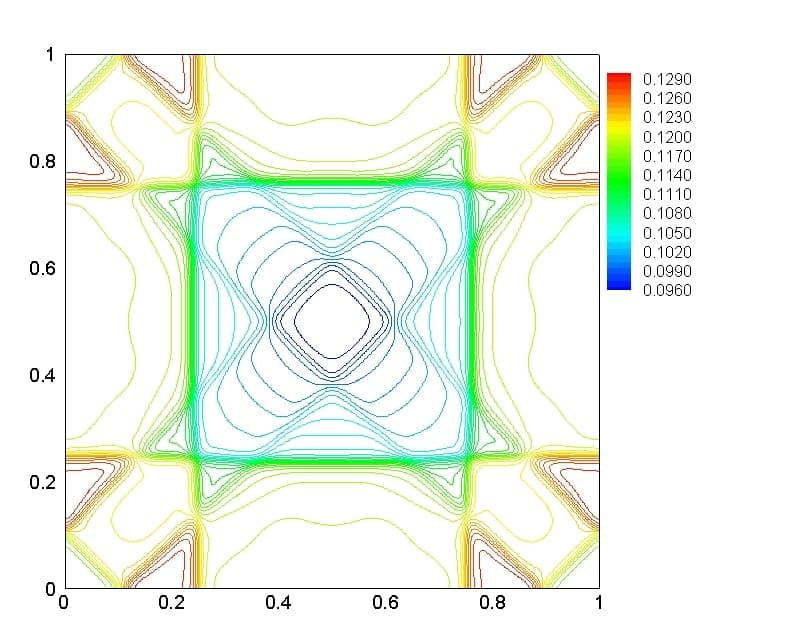}} \\

\centering
{\includegraphics[height=2.1in]{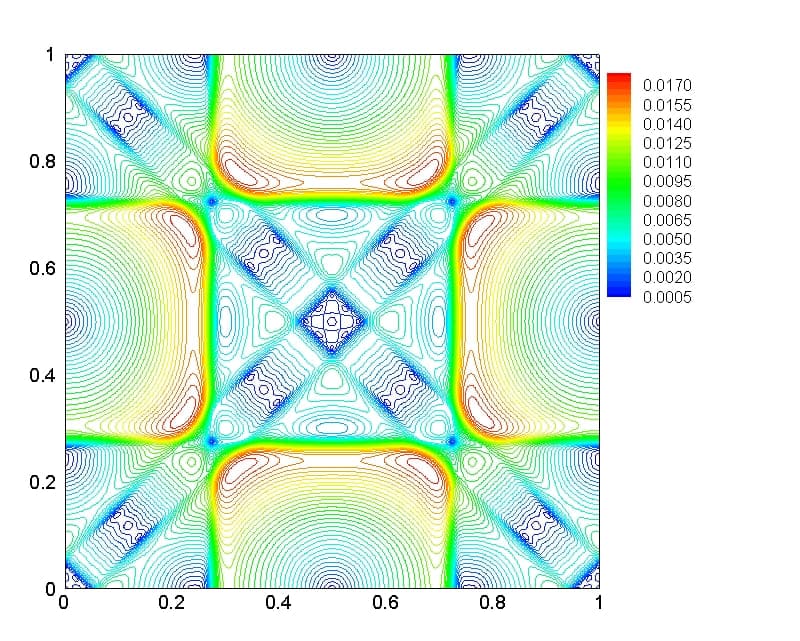}}
 {\includegraphics[height=2.1in]{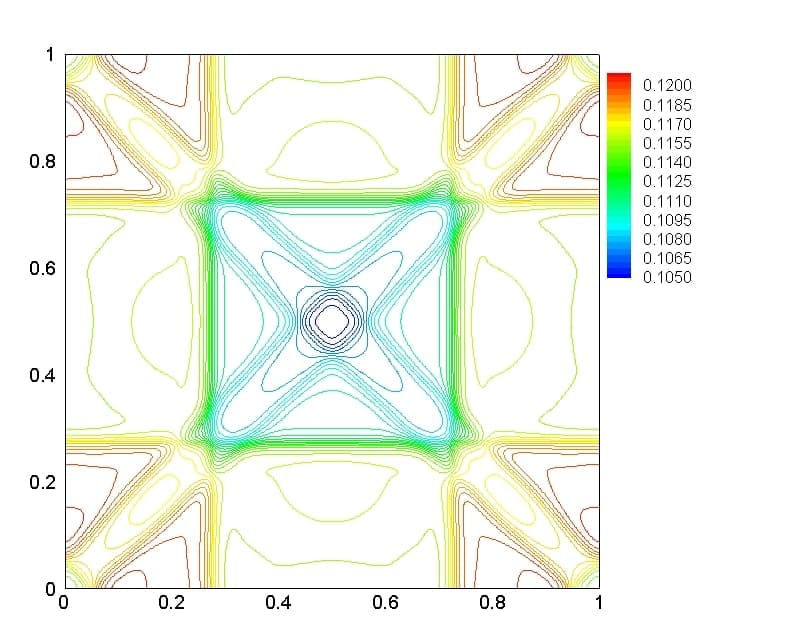}} \\

\centering
{\includegraphics[height=2.1in]{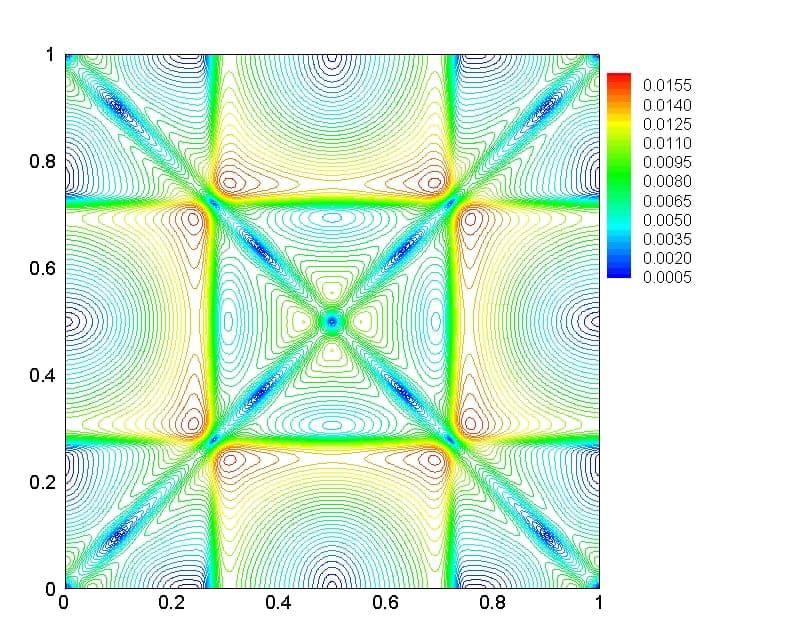}}
 {\includegraphics[height=2.1in]{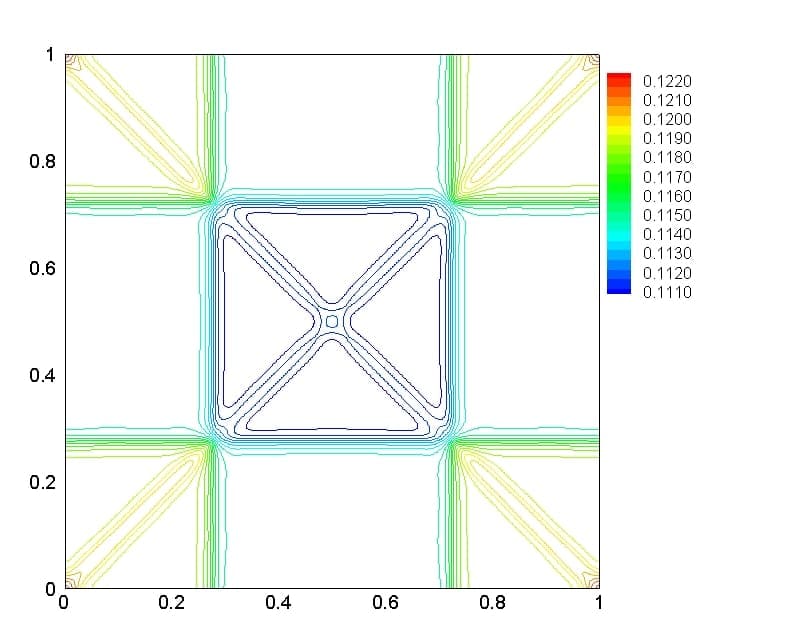}} 
    
\caption{Solutions of the 2-D spatially homogeneous system in the expanding background at $t=8$, $16$, $50$, $60$. Right column: Rescaled solution $\rhot$. Left column: Velocity magnitude $V$.}
\label{a(t)-ex}
\end{figure}


\subsection{Flows on an non-homogeneous background in one space dimension} 

We demonstrate here that the asymptotics of the solutions at the fine scale level is driven by the underlying background geometry.

\paragraph*{Test 1: Stationary initial data.}  We first validate  the well-balanced property of the  scheme, that is, it can preserve and capture smooth steady solutions of  the Euler-FLRW model  when $a(t)\equiv 1$. 
We first consider the special steady state solutions that is
\bel{eq: sssv000}
u= u(x) \equiv 0, \qquad \rho = \rho(x) = C b^2(x),
\ee
where $C$ is a constant. We use different schemes to show that the modified HLL scheme is  well-balanced.
In this test, we choose the initial data to be
\bel{eq: sssv000}
u_0=  0, \qquad \rho_0(x) = b^2(x).
\ee
at $t = 1$,
and the  function $b(x)$ to be
\bse
\bel{eq:bxa}
b(x) = 1+ 0.01 \sin(2 \pi x),
\ee 
\bel{eq:bxb}
b(x) = 1+ 0.01 \big(\sin(6 \pi x) + \cos(2\pi x)\big),
\ee
\ese
respectively.
We take $N =100$, $\eps =1$, $k =0.5$ and  $CFL = 0.6$.  We plot the solution by using our proposed HLL scheme and the standard HLL scheme when $t = 10$ (see Figure~\ref{Fig: Euler-WBHLL-EX0}-\ref{Fig: Euler-WBHLL-EX1}).
We observe that the modified HLL scheme is able to exactly preserve the steady state solution while the standard HLL scheme cannot preserve it. Note that oscillations appear for the solution of the velocity $u$.

\begin{figure}[htbp]
\centering
\epsfig{figure = 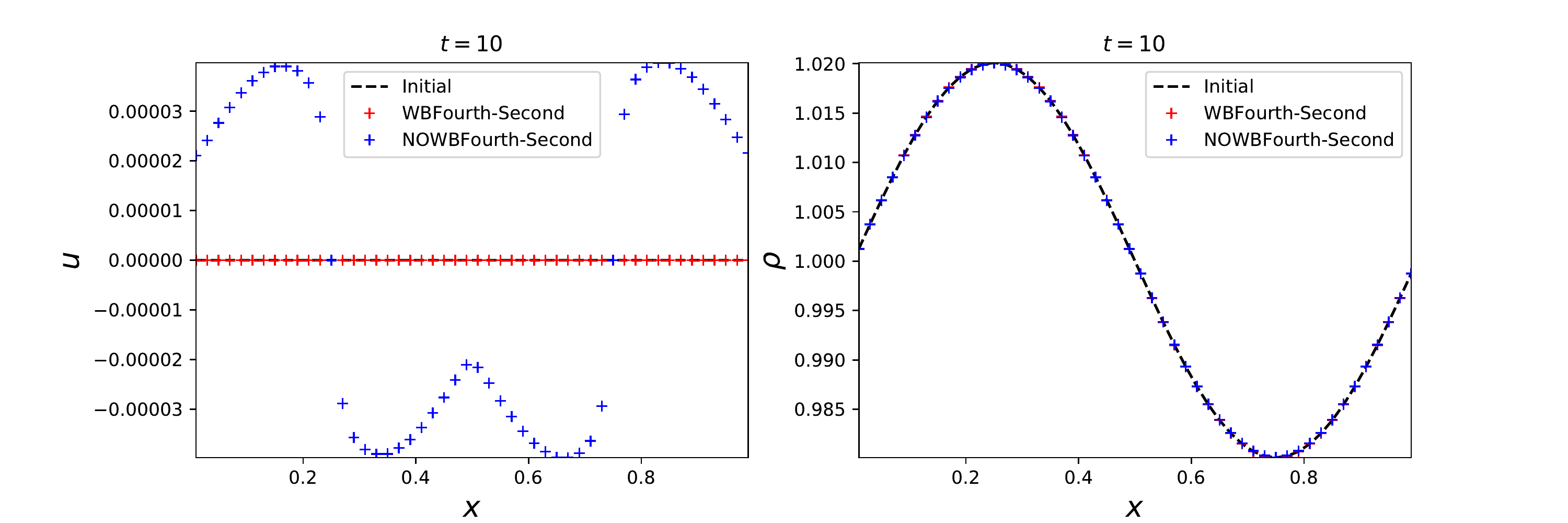,height = 2.1 in}  
\caption{The solution on an expanding background at $t = 10$ with  \eqref{eq:bxa}.}
\label{Fig: Euler-WBHLL-EX0}
\end{figure}

\begin{figure}[htbp]
\centering
\epsfig{figure = 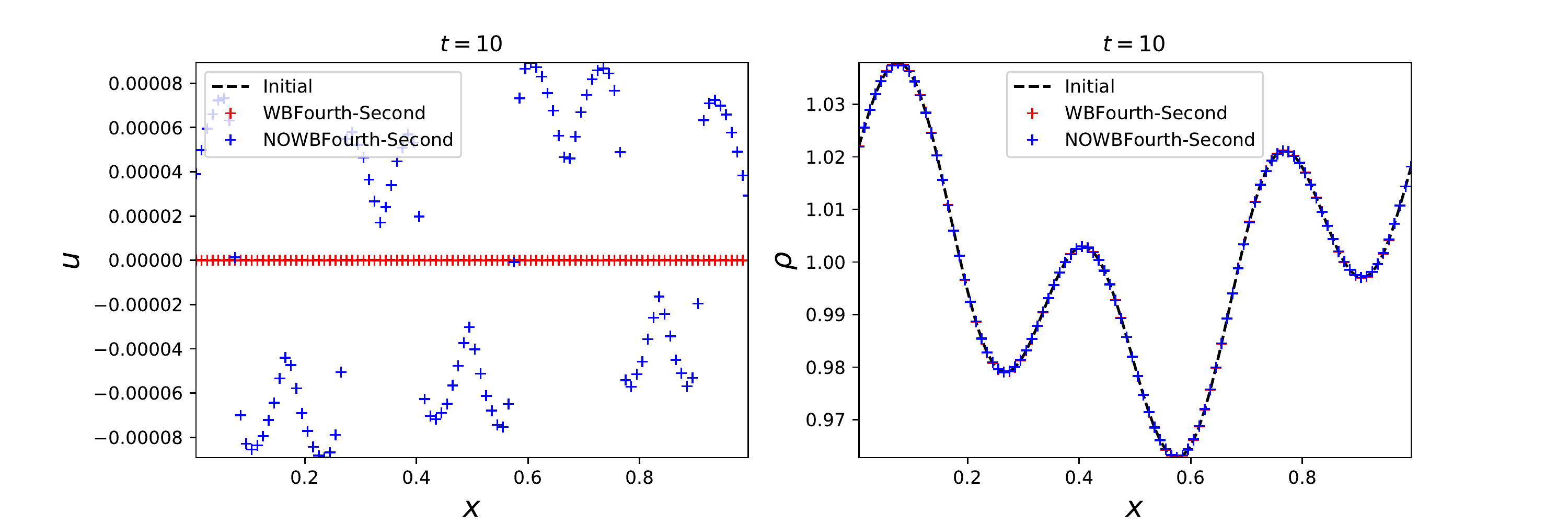,height = 2.1 in}  
\caption{The solution on an expanding background at $t = 10$ with \eqref{eq:bxb}.}
\label{Fig: Euler-WBHLL-EX1}
\end{figure}

\paragraph*{Test 2: Perturbed initial data.} 
In this test,  we choose the following initial data
\bel{eq: sssv11}
u_0=  0, \qquad \rho_0(x) =  \begin{cases} b^2(x) + 0.02 \cos(30 \pi x), \quad 
& \qquad   0.2 \leq x \leq 0.7, 
\\
b^2(x), &  \qquad \text{otherwise},
\end{cases}
\ee
which has a perturbed density and a vanishing velocity, where $b(x)$ is given by \eqref{eq:bxb}.  The initial data and the solution at $t =10$ are plotted in the Figure~\ref{fig:expand}. We observe that even if the initial data is not 
a steady state solution, the solution converges to the steady state one. This shows the ability of the scheme to capture the steady state solutions.
\begin{figure}
\centering 
\begin{minipage}{0.5\linewidth}
  \centerline{\includegraphics[width=5.6in]{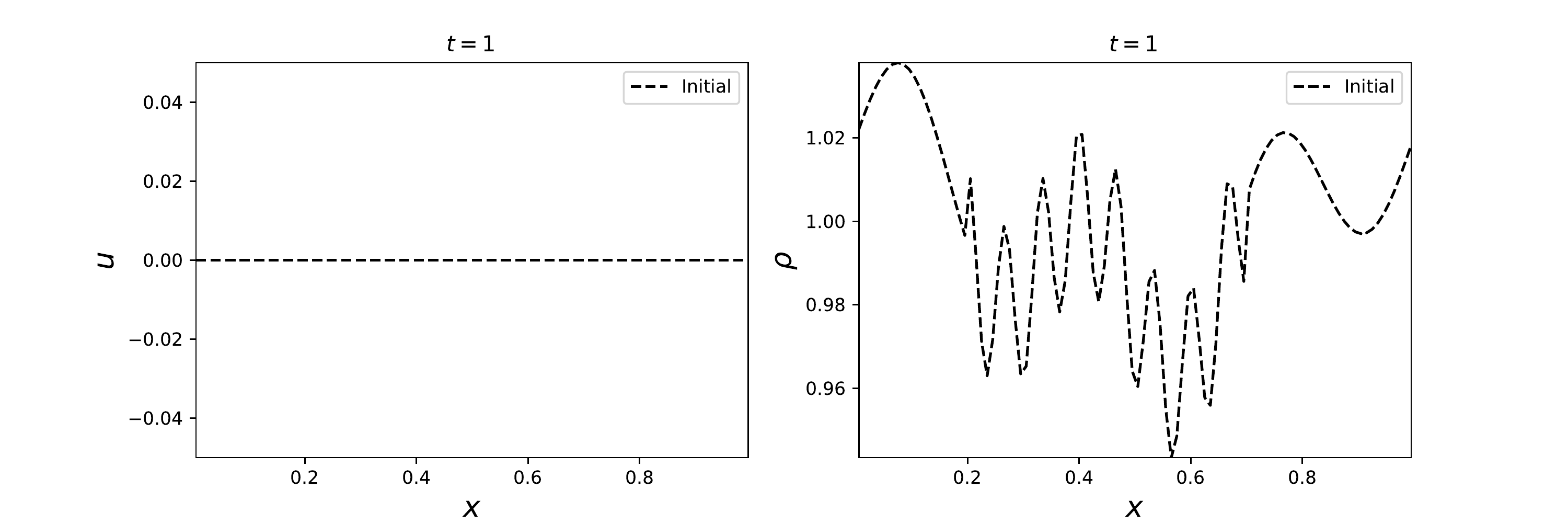}} 
\end{minipage}
\vfill
\begin{minipage}{.5\linewidth}
  \centerline{\includegraphics[width=5.6in]{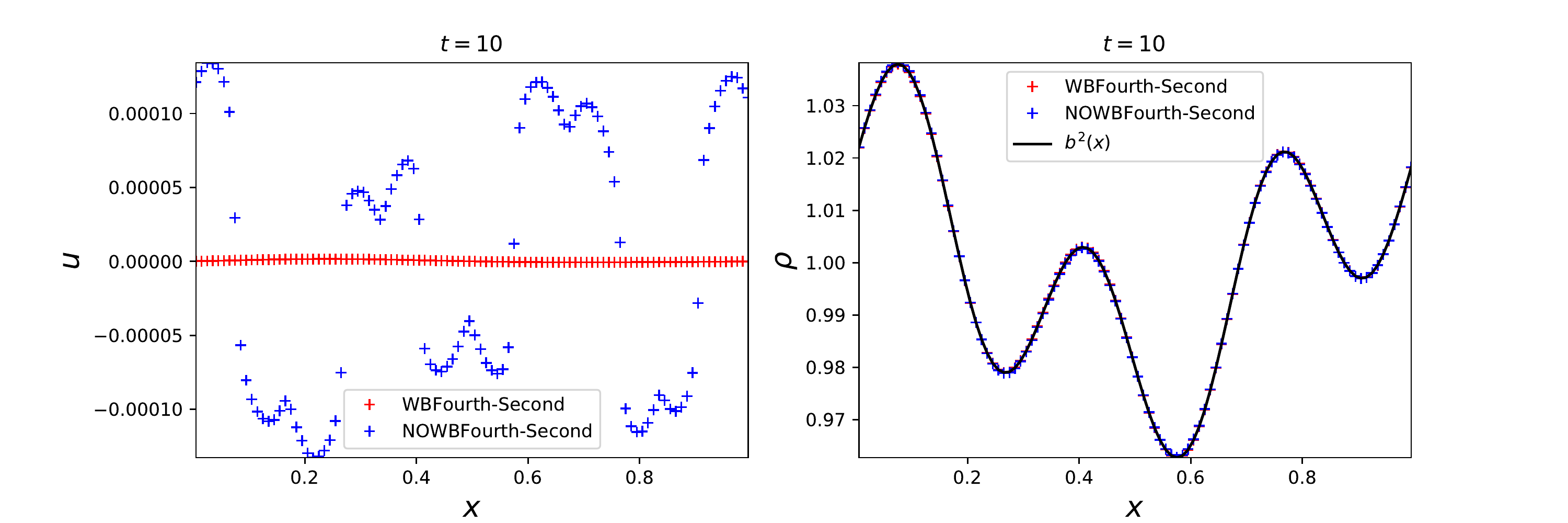}} 
\end{minipage}

\caption{The solution on an expanding background.}
\label{fig:expand}
\end{figure}

\paragraph*{Test 3: Perturbed initial data.} 
We now turn to  consider the Euler model when $a(t) \neq 1$. We take here $\kappa =2$,
and we choose the same initial data as \eqref{eq: sssv11}. We plot the solution when $t = 10$ and $t = 20$ (see Figure \ref{fig:expand111}). We find that the solution
$\rho \to 0$ and $u \to 0$, as $t$ increases.

\begin{figure}
\centering 
\begin{minipage}{0.5\linewidth}
  \centerline{\includegraphics[width=5.6in]{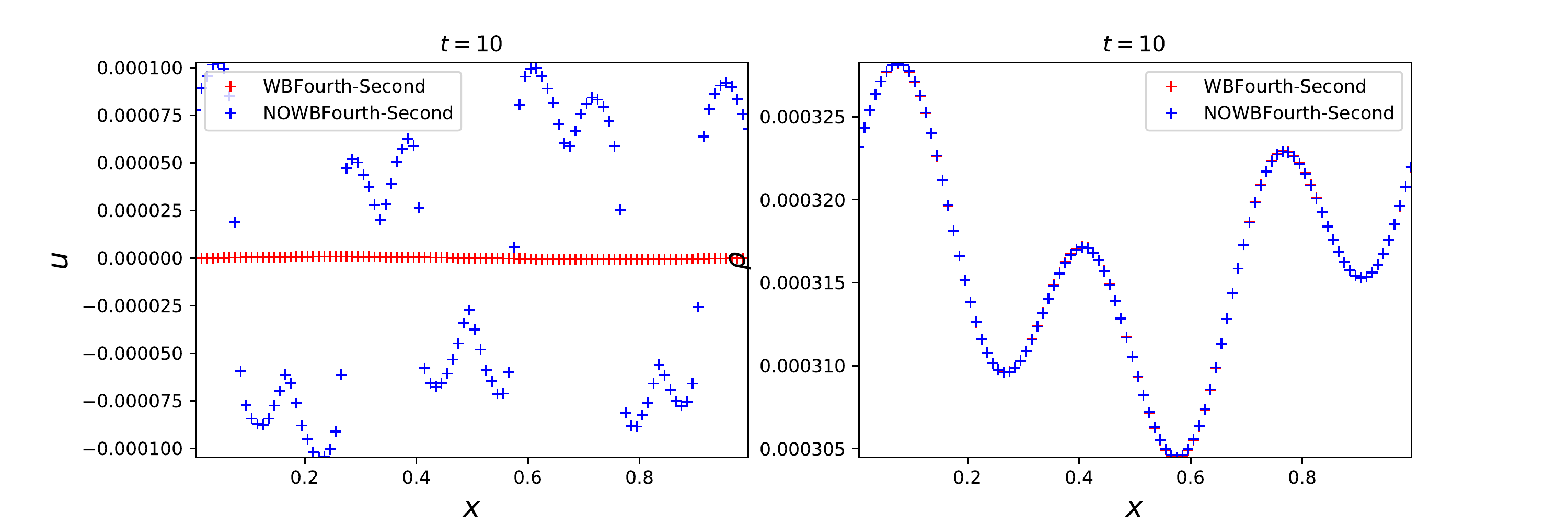}} 
\end{minipage}
\vfill
\begin{minipage}{.5\linewidth}
  \centerline{\includegraphics[width=5.6in]{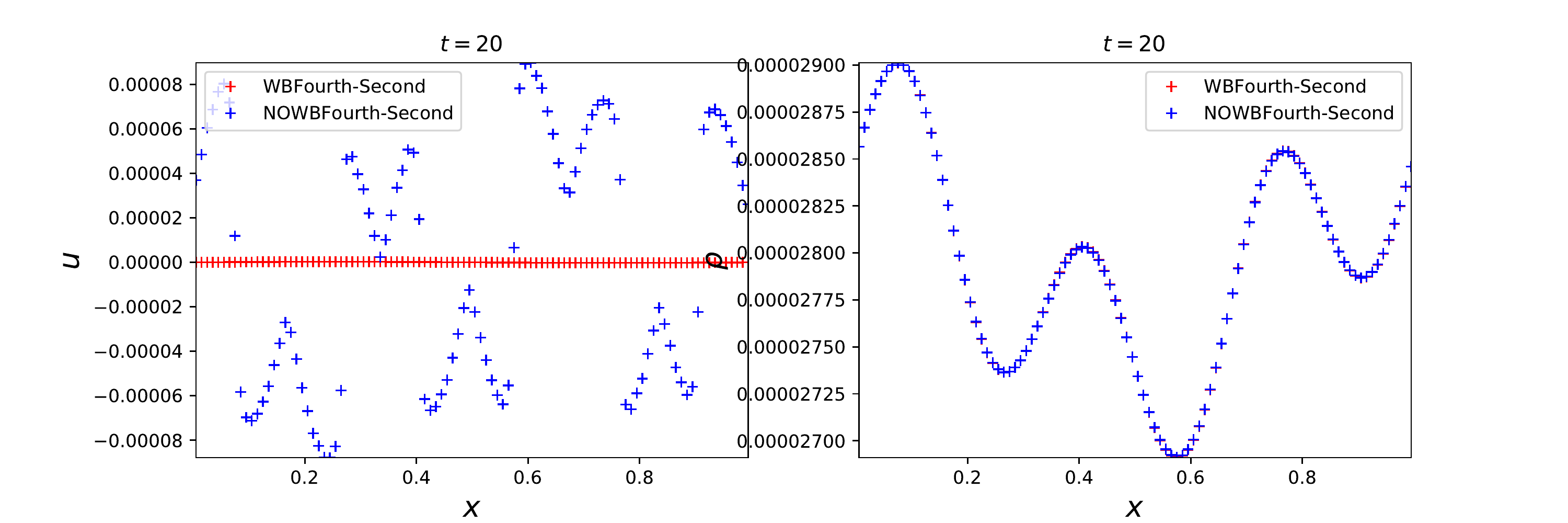}} 
\end{minipage}

\caption{Solution on an expanding background with $\kappa  =2$.}
\label{fig:expand111}
\end{figure}


\paragraph{Rescaling the numerical solution}

We now display the rescaled solution $\ut$ and $\rhot$ defined in \eqref{eq: Reu01}; see Figure~\ref{Fig: Euler-Wb-Res}. 
We observe that the asymptotic density converges to $b^2(x)$, while the rescaled velocity converges to $0$.

\begin{figure}[htbp]
\centering
\epsfig{figure = 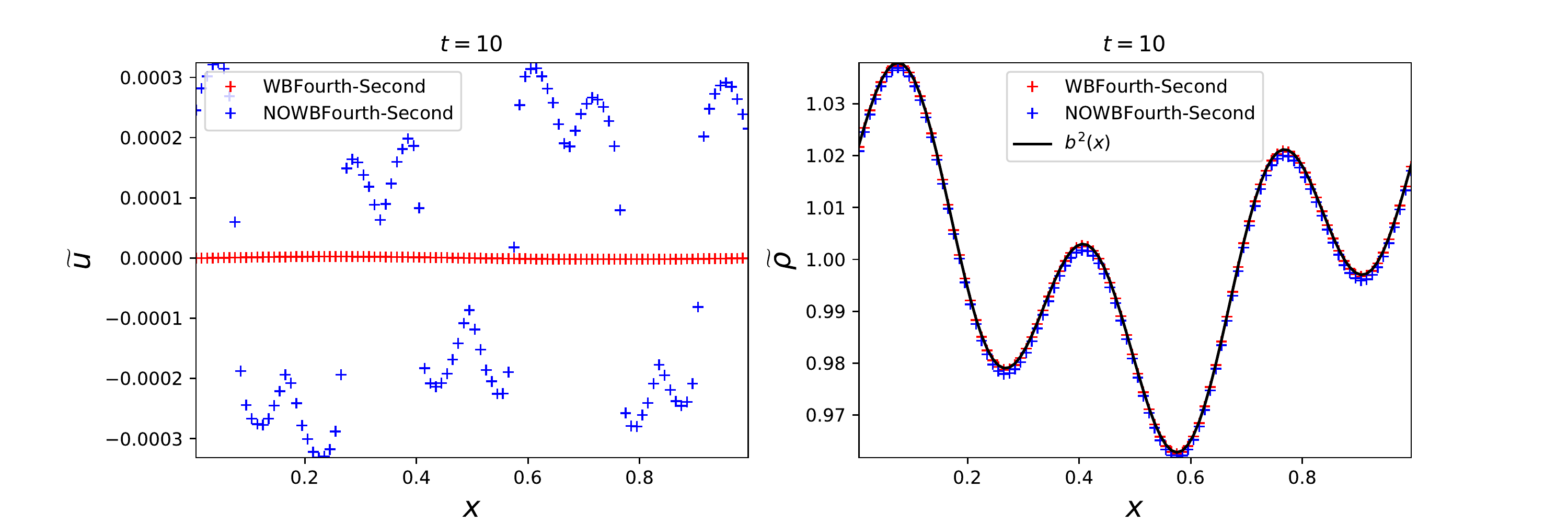,height = 2.1 in}  
\caption{Rescaled solution on an expanding background at $t = 10$.}
\label{Fig: Euler-Wb-Res}
\end{figure}



\subsection{Flows on an non-homogeneous background in two space dimensions} 

Similar to the one-dimensional tests, again we expect that the asymptotics of the solutions at the fine scale level is driven by the underlying background geometry which we now assume to be non-homogeneous in both spatial directions. 

\paragraph*{Test 1: Point symmetrical initial data with $\boldsymbol{a(t)\equiv 1}$.}

We choose the following initial data posed at $t_0 = 1$ and defined in the domain $[0,1]\times[0,1]$: 
\be
\rho_0 = 1 + 0.01 \big(\sin(2  \pi x)  \cos(2  \pi x) \sin(2  \pi y)  \cos(2  \pi y)\big), \qquad u_0(x,y) = 0, \qquad v_0(x,y) = 0,
\label{eq:initial2D}
\ee
which is shown in Figure~\ref{2d-initial2}. Moreover, the background geometry $b(x,y)$ is chosen as
\be
b(x,y) = 0.1 + 0.01 e^{-20 {(x - 0.5)}^2 - 20 {(y - 0.5)}^2}.
\label{bxy2d}
\ee

\begin{figure}[htbp]
\centering

 {\includegraphics[height=2.1in]{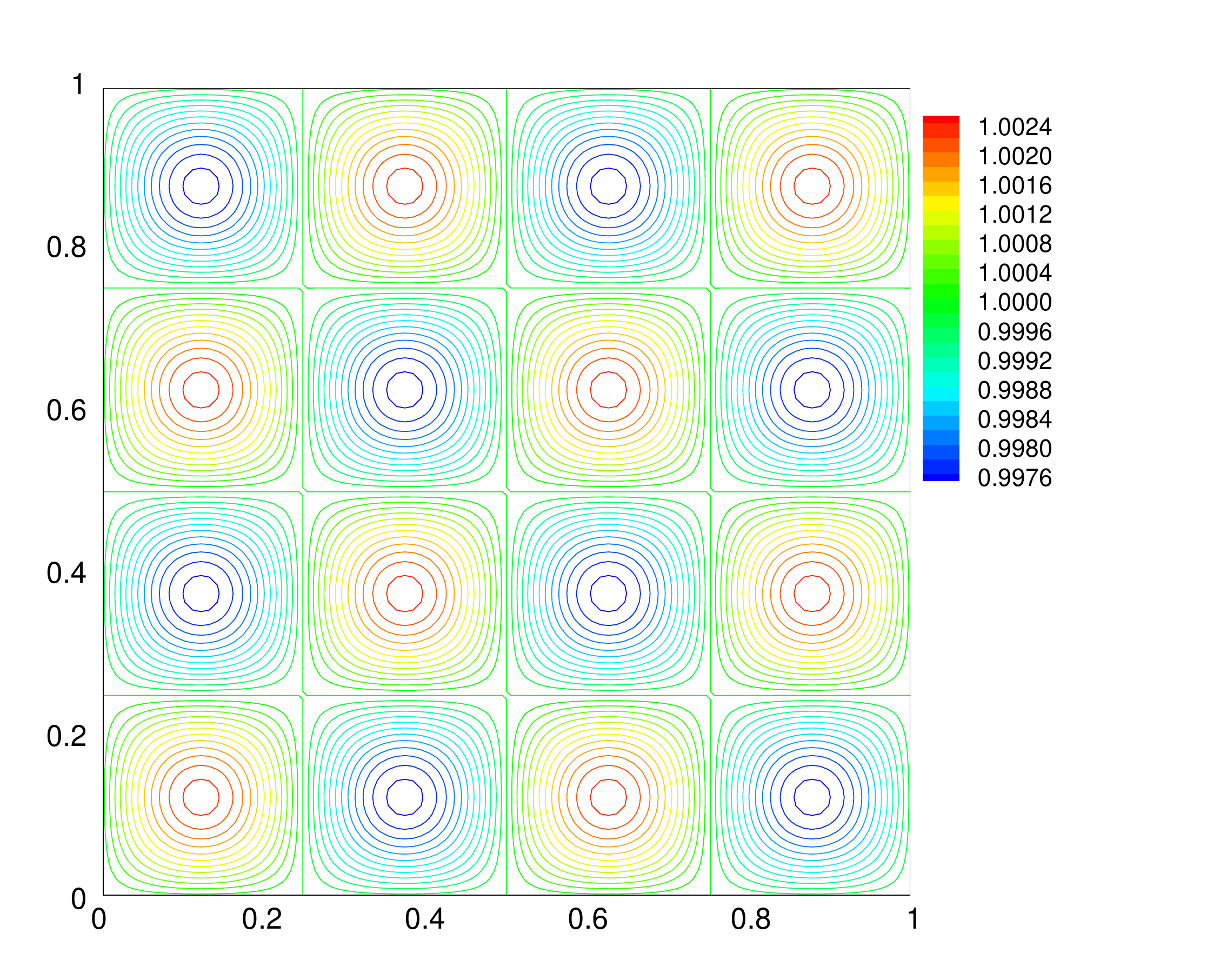}} 
   
\caption{Initial data in two space dimensions.}
\label{2d-initial2}
\end{figure}

The two-dimensional system is solved with the standard HLL scheme and our proposed one. In Figures~\ref{bxy-nwb} and~\ref{bxy-wb}, we plot the solution $\rho$ and velocity magnitude $V$ at $t=8$, $16$, $50$, $60$, respectively. The results of both schemes demonstrate that the solution $\rho \to Cb^2(x,y)$  (where $C$ is a positive constant) as $t$ increases. The solutions of the well-balanced scheme converge to the steady state solution drastically faster and show significant improvements in comparison to the non-well-balanced ones. This shows the importance of using the well-balanced scheme. 

\begin{figure}[htbp]
\centering
{\includegraphics[height=2.1in]{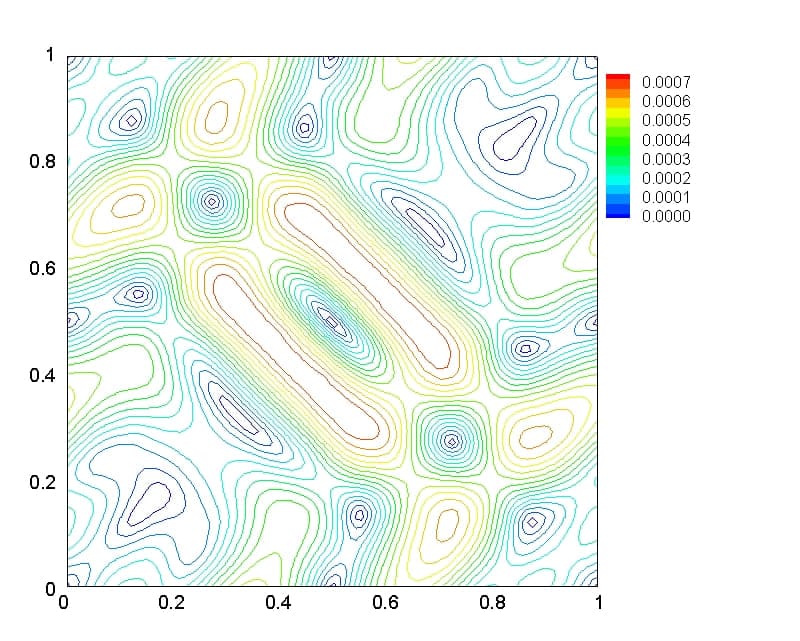}}
 {\includegraphics[height=2.1in]{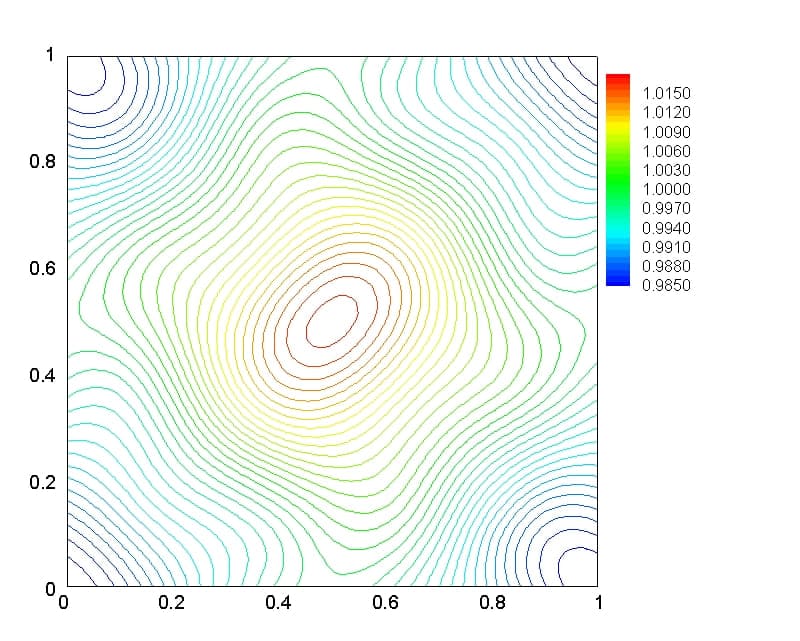}} \\
  
\centering
{\includegraphics[height=2.1in]{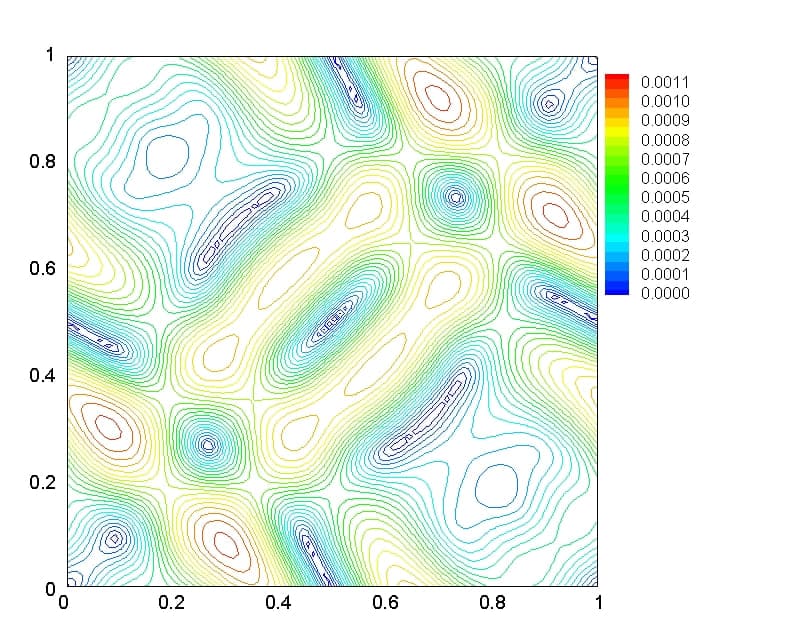}}
 {\includegraphics[height=2.1in]{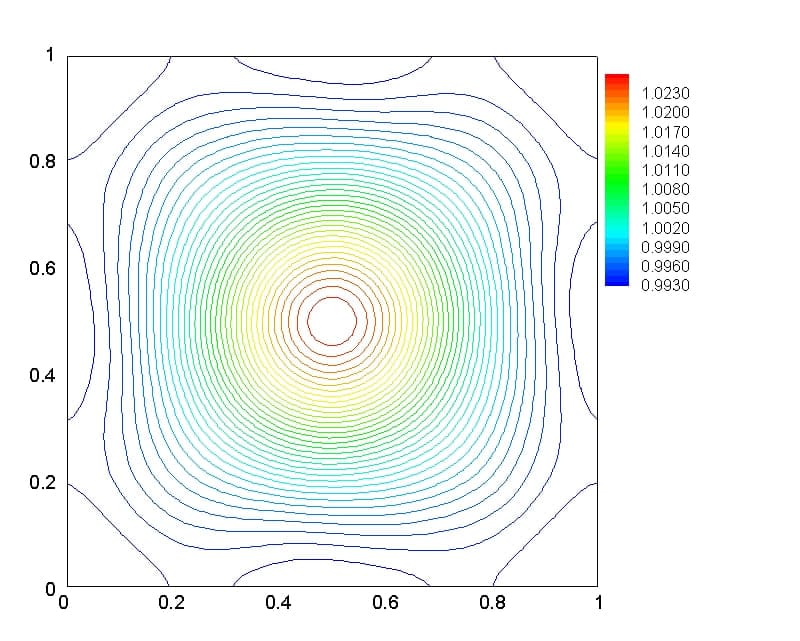}} \\

\centering
{\includegraphics[height=2.1in]{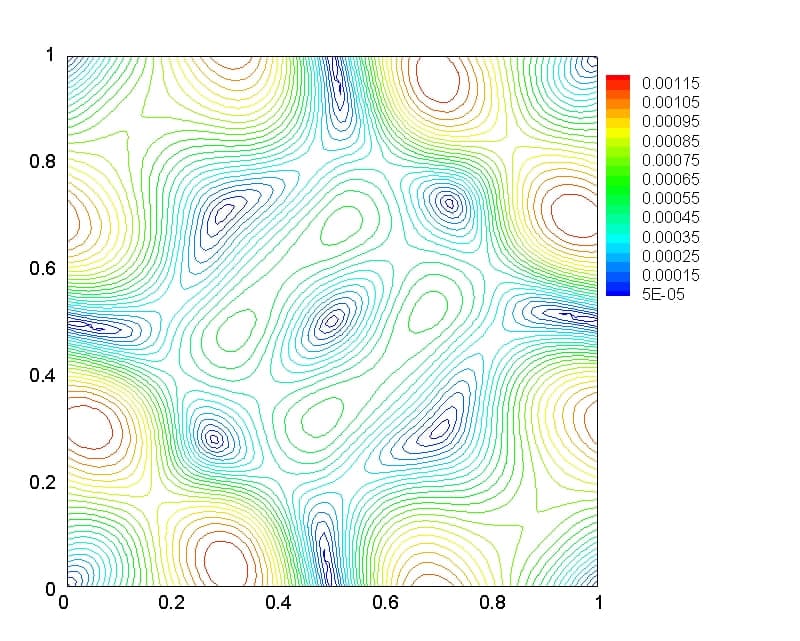}}
 {\includegraphics[height=2.1in]{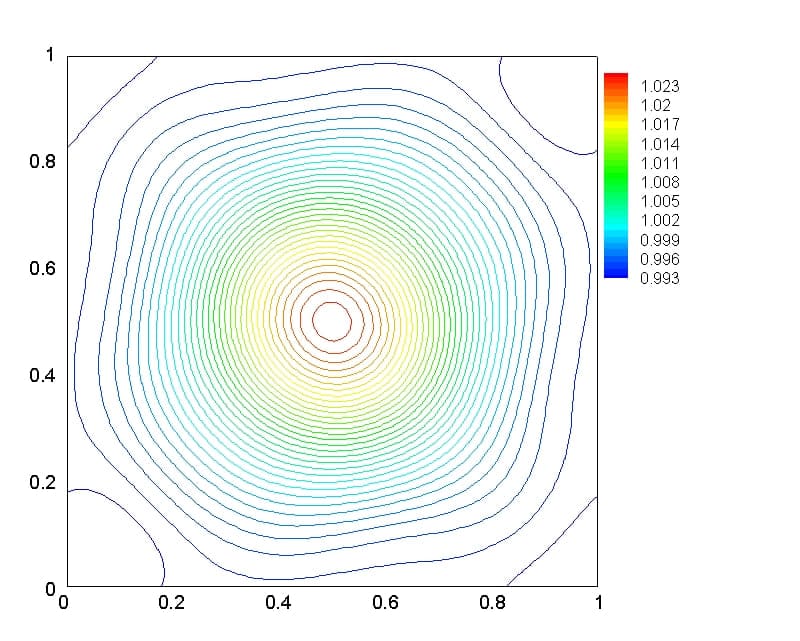}} \\

\centering
{\includegraphics[height=2.1in]{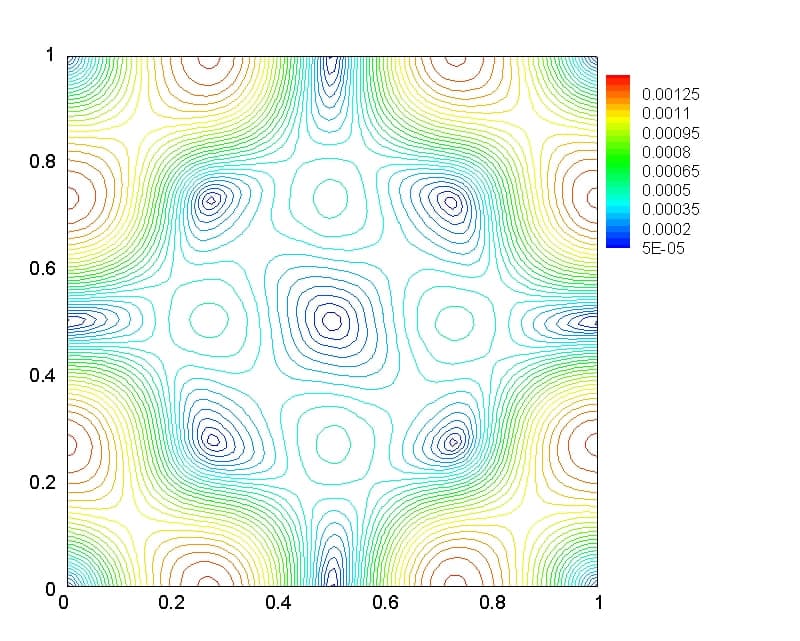}}
 {\includegraphics[height=2.1in]{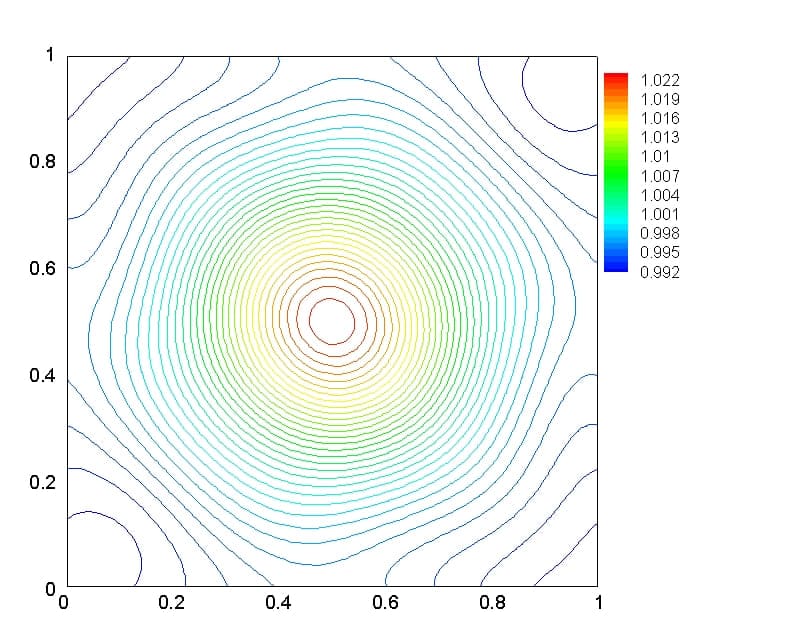}} 
    
\caption{Non-well-balanced solutions of the 2-D system with $a(t)\equiv 1$ at $t=8$, $16$, $50$, $60$. Right column: Solution $\rho$. Left column: Velocity magnitude $V$.}
\label{bxy-nwb}
\end{figure}

\begin{figure}[htbp]
\centering
{\includegraphics[height=2.1in]{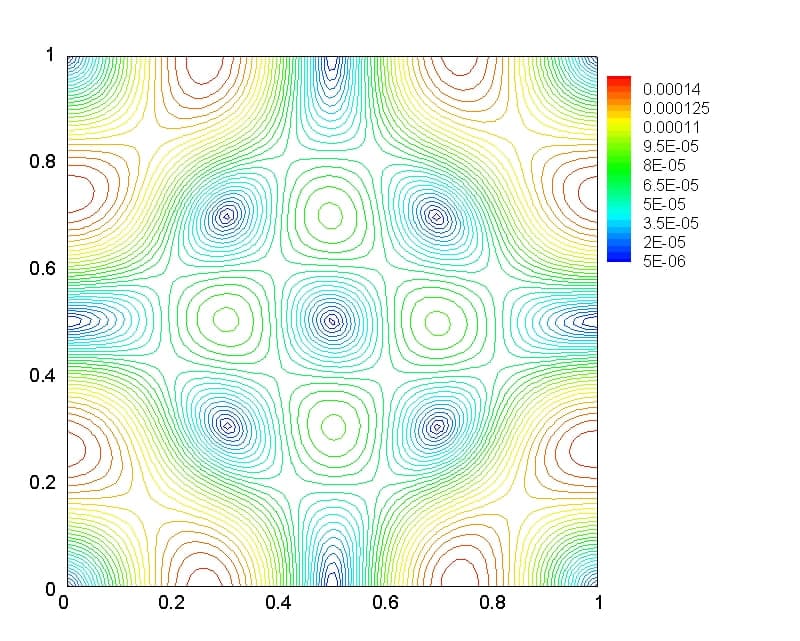}}
 {\includegraphics[height=2.1in]{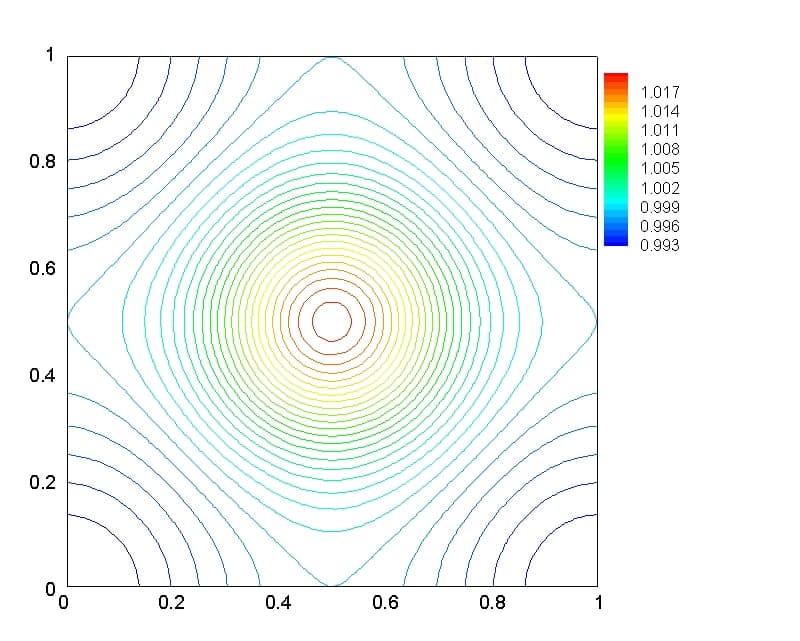}} \\
  
\centering
{\includegraphics[height=2.1in]{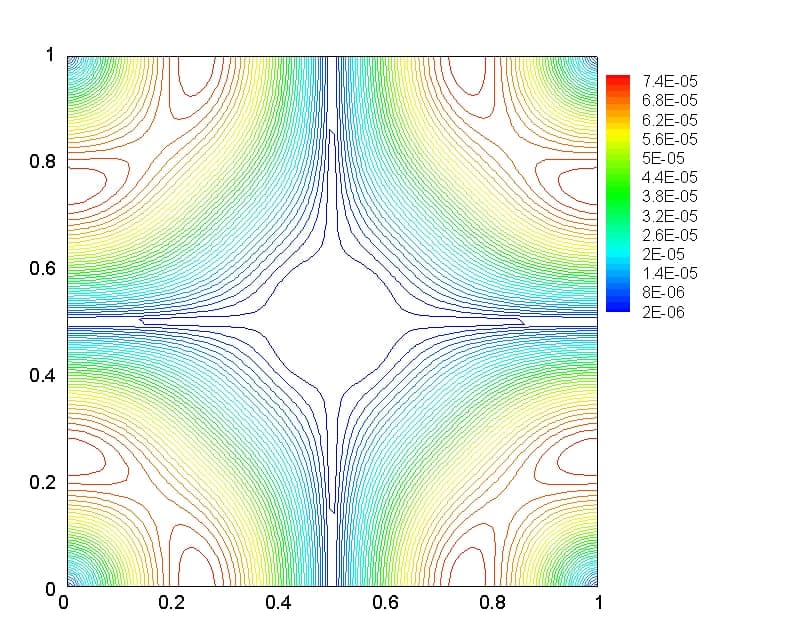}}
 {\includegraphics[height=2.1in]{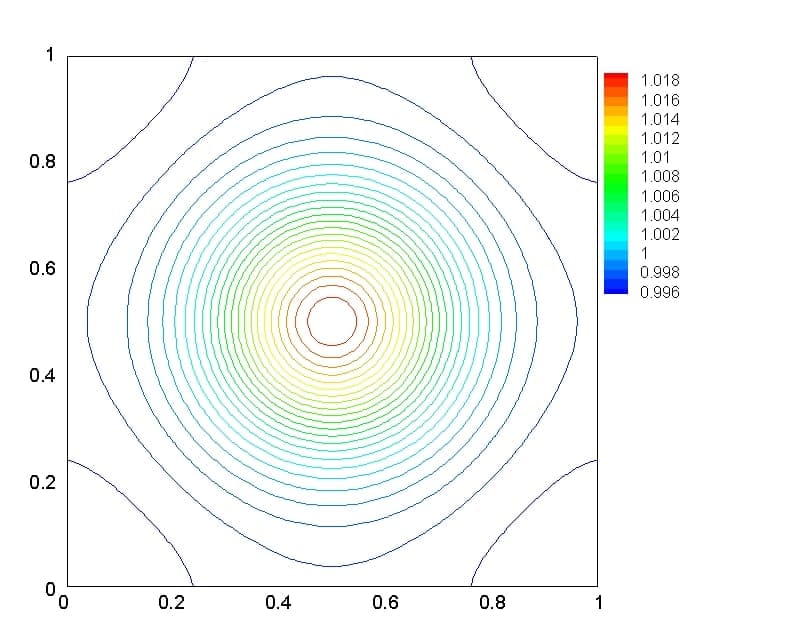}} \\

\centering
{\includegraphics[height=2.1in]{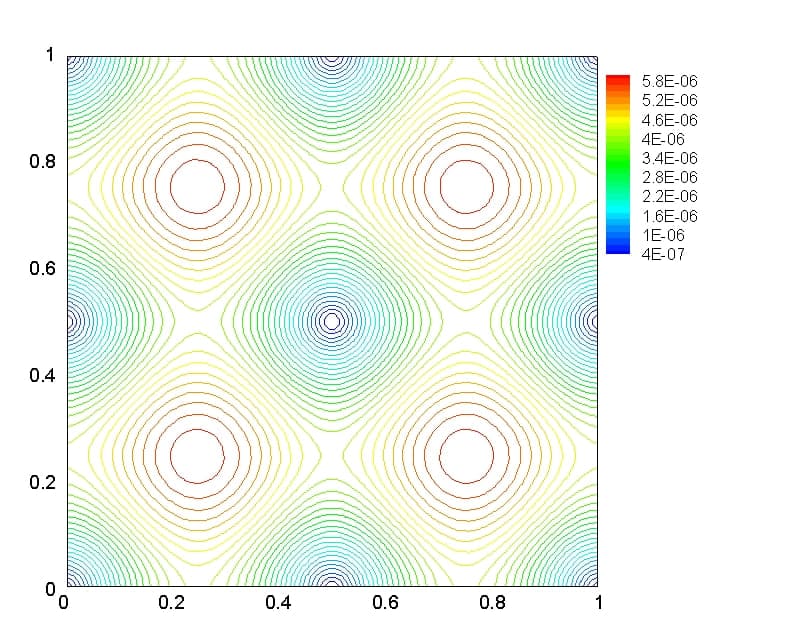}}
 {\includegraphics[height=2.1in]{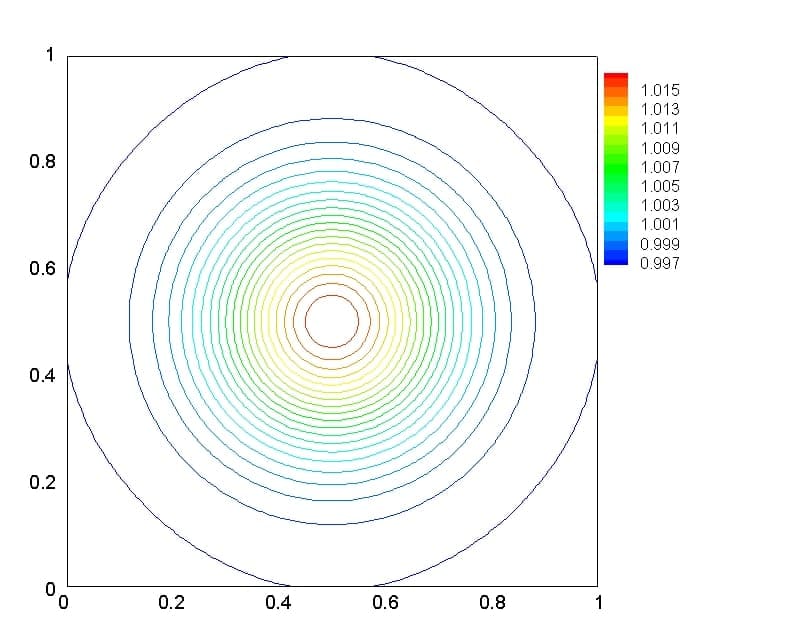}} \\

\centering
{\includegraphics[height=2.1in]{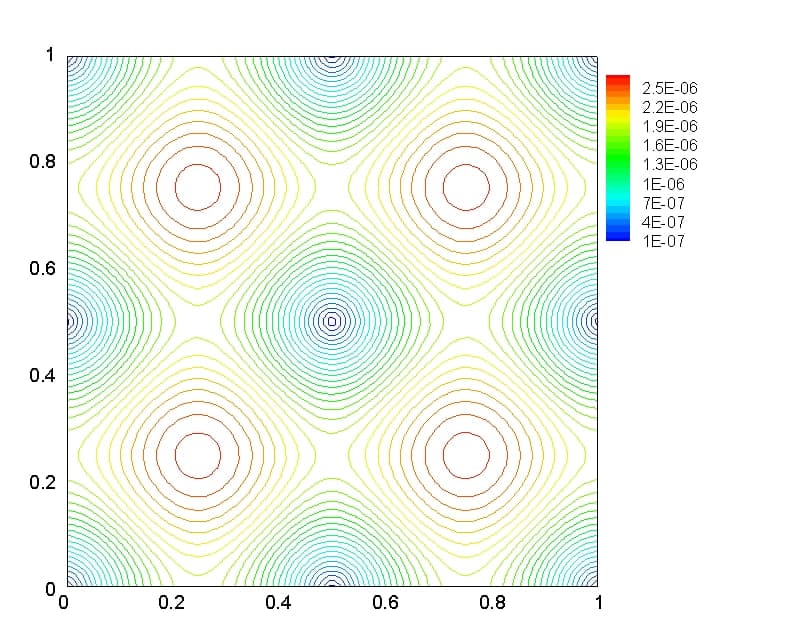}}
 {\includegraphics[height=2.1in]{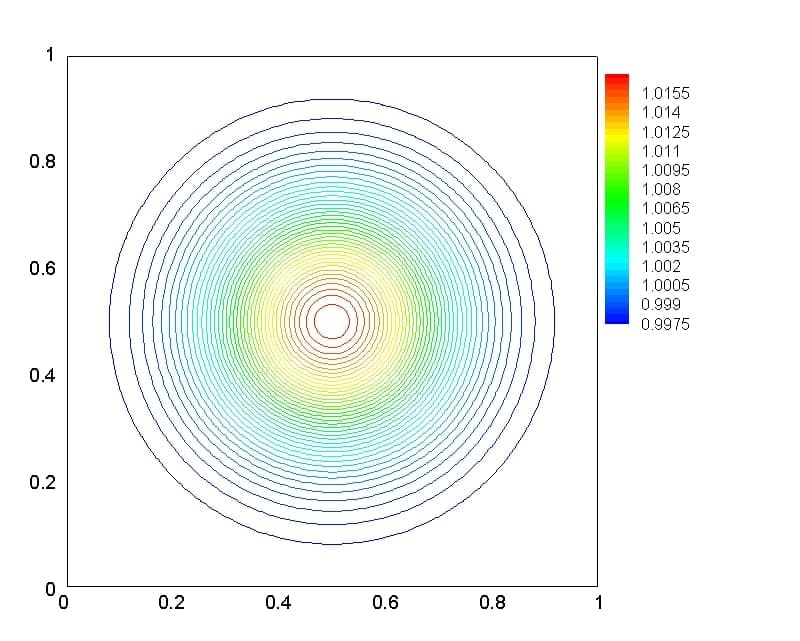}} 
    
\caption{Well-balanced solutions of the 2-D system with $a(t)\equiv 1$ at $t=8$, $16$, $50$, $60$. Right column: Solution $\rho$. Left column: Velocity magnitude $V$.}
\label{bxy-wb}
\end{figure}

\paragraph*{Test 2: Point symmetrical initial data in an expanding background.}

In this test we demonstrate the effect of $a(t) = t^\kappa$ when $t \in [1, +\infty)$ over the background geometry $b(x,y)$ in \eqref{bxy2d}. We solve the two-dimensional system similar to the previous test in this section and plot the rescaled solution $\rhot$ and velocity magnitude $V$ at $t=8$, $16$, $50$, $60$, respectively. The results of both schemes demonstrate that the solution $\rho \to 0$, $V \to 0$ and the rescaled density $\rhot$ converges to  $Cb^2(x,y)$ (where $C$ is a positive constant) as $t$ increases. Similar to the previous test, the rescaled solutions $\rhot$ of the well-balanced scheme converge to the steady state solution faster and show improvements in comparison to the non-well-balanced ones.

\begin{figure}[htbp]
\centering
{\includegraphics[height=2.1in]{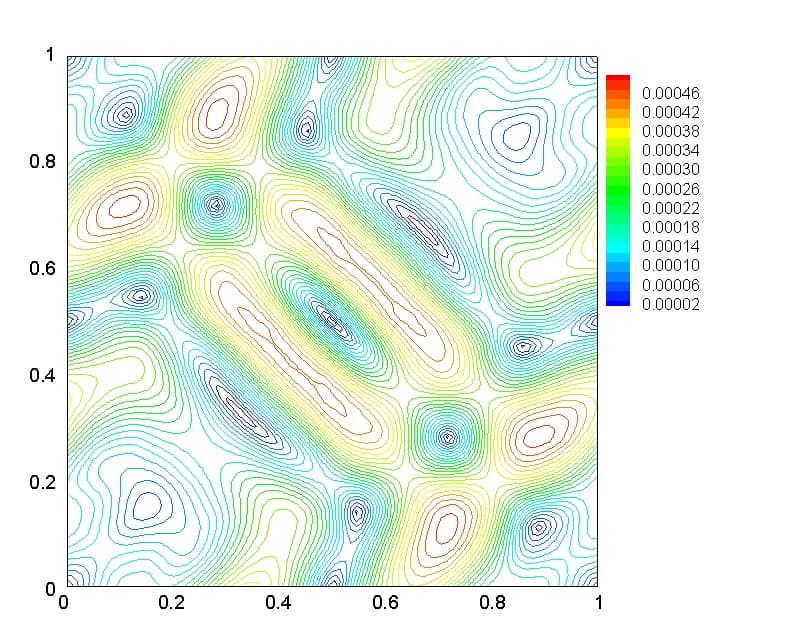}}
 {\includegraphics[height=2.1in]{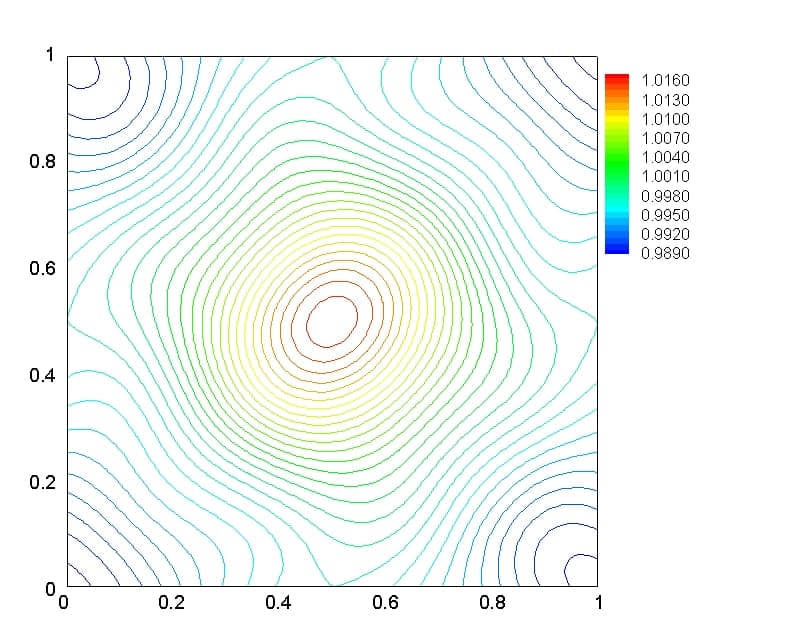}} \\
  
\centering
{\includegraphics[height=2.1in]{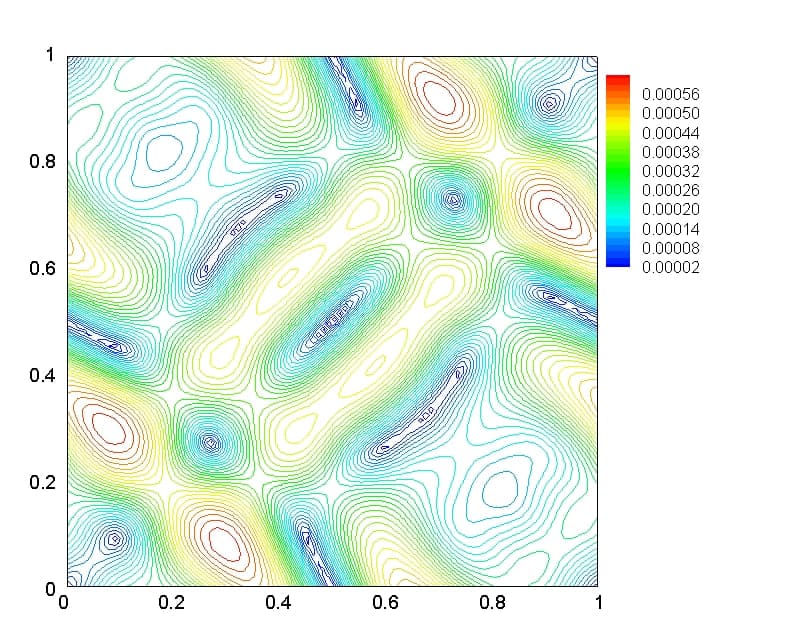}}
 {\includegraphics[height=2.1in]{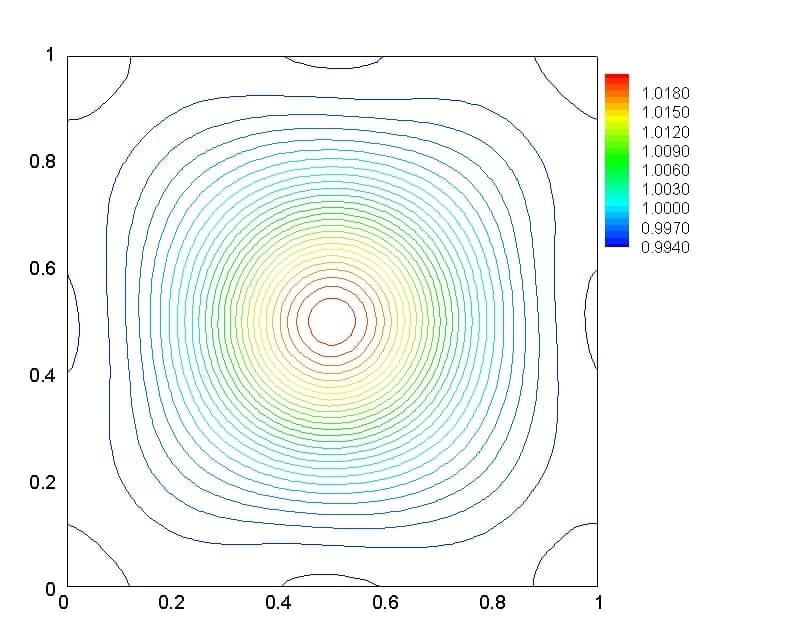}} \\

\centering
{\includegraphics[height=2.1in]{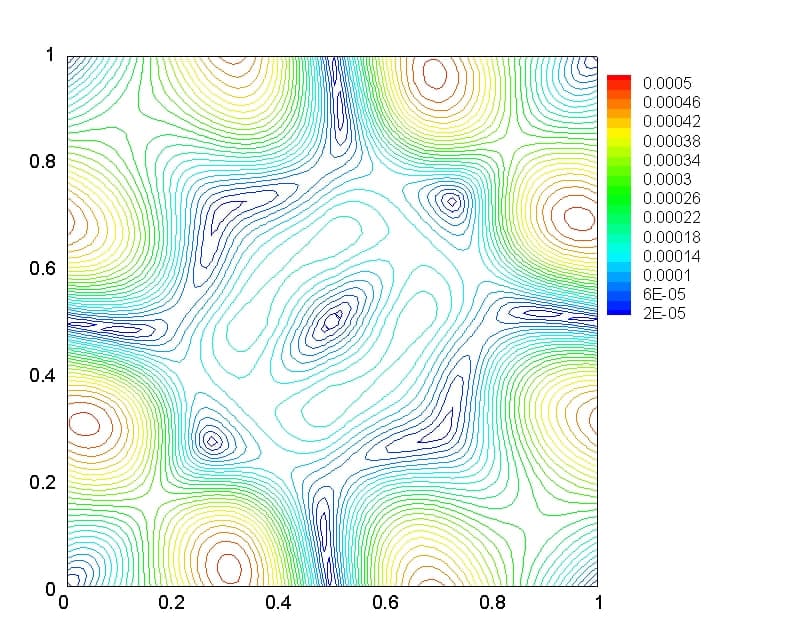}}
 {\includegraphics[height=2.1in]{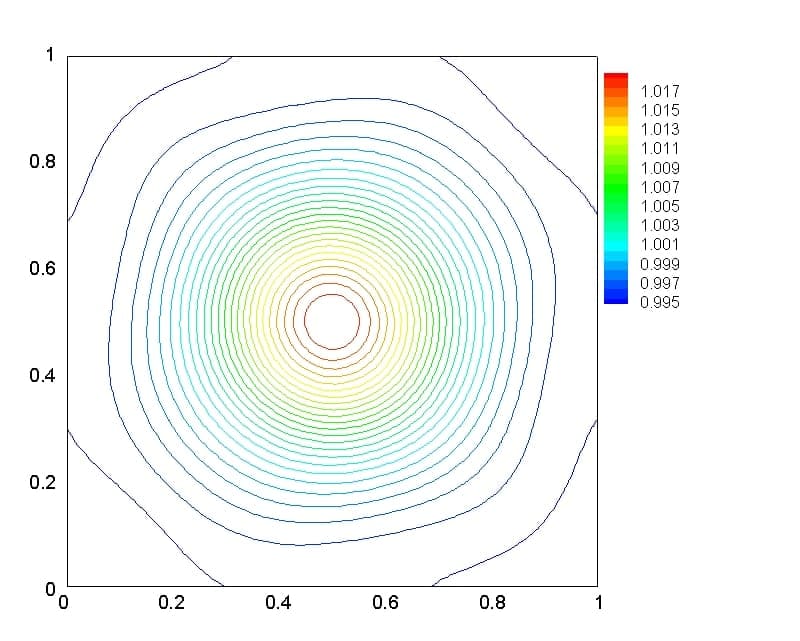}} \\

\centering
{\includegraphics[height=2.1in]{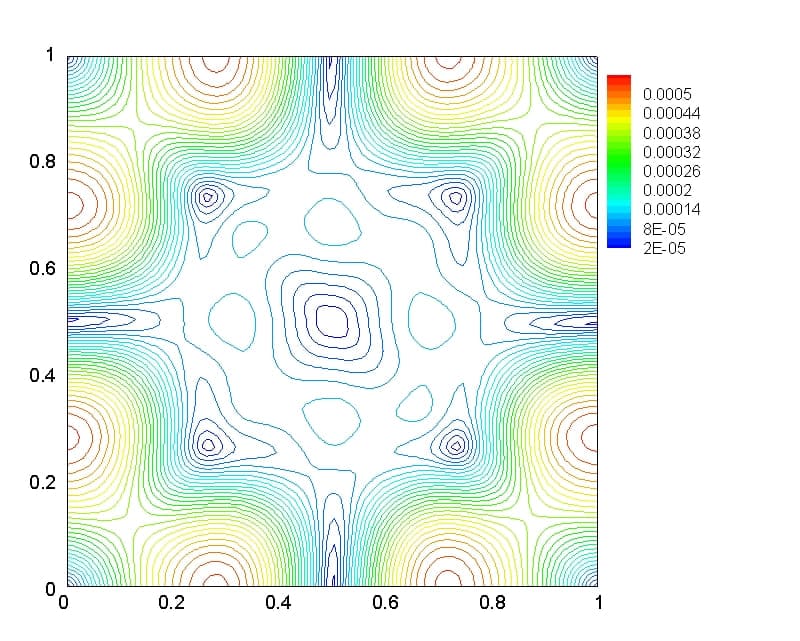}}
 {\includegraphics[height=2.1in]{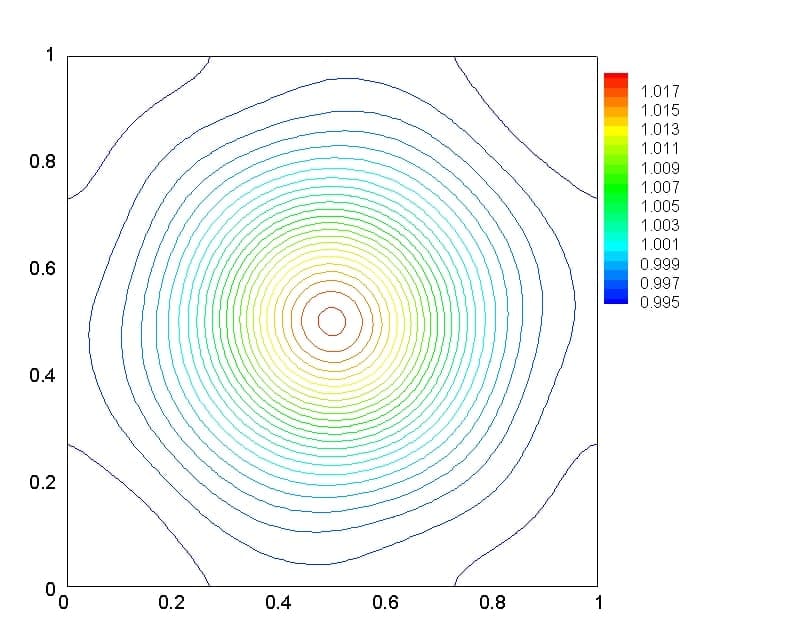}} 
    
\caption{Non-well-balanced solutions of the 2-D system in an expanding background with the background geometry $b(x,y)$ at $t=8$, $16$, $50$, $60$. Right column: Rescaled solution $\rhot$. Left column: Velocity magnitude $V$.}
\label{bxy-nwb-exp}
\end{figure}

\begin{figure}[htbp]
\centering
{\includegraphics[height=2.1in]{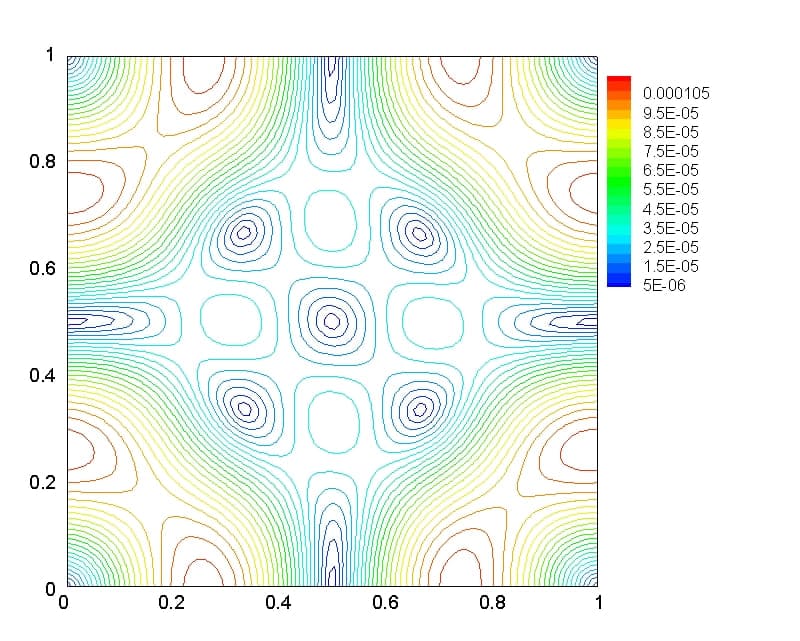}}
 {\includegraphics[height=2.1in]{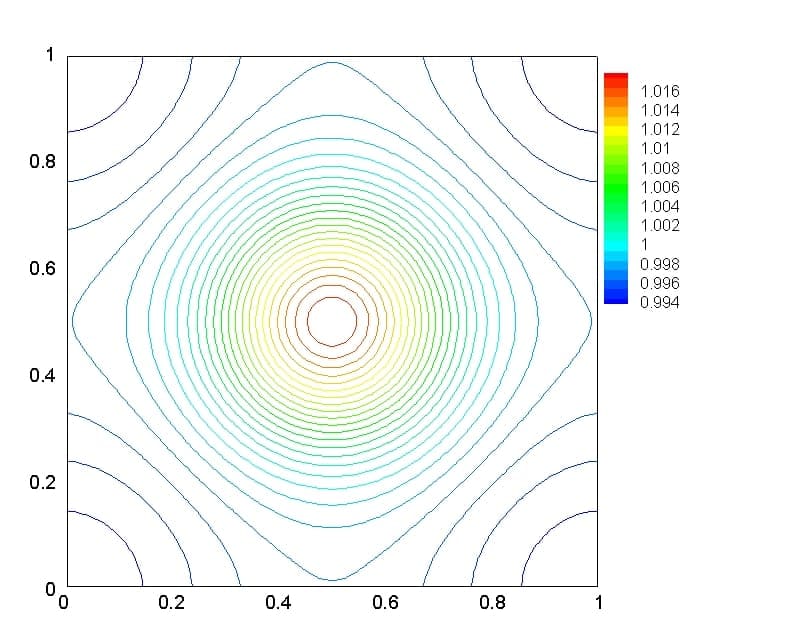}} \\
  
\centering
{\includegraphics[height=2.1in]{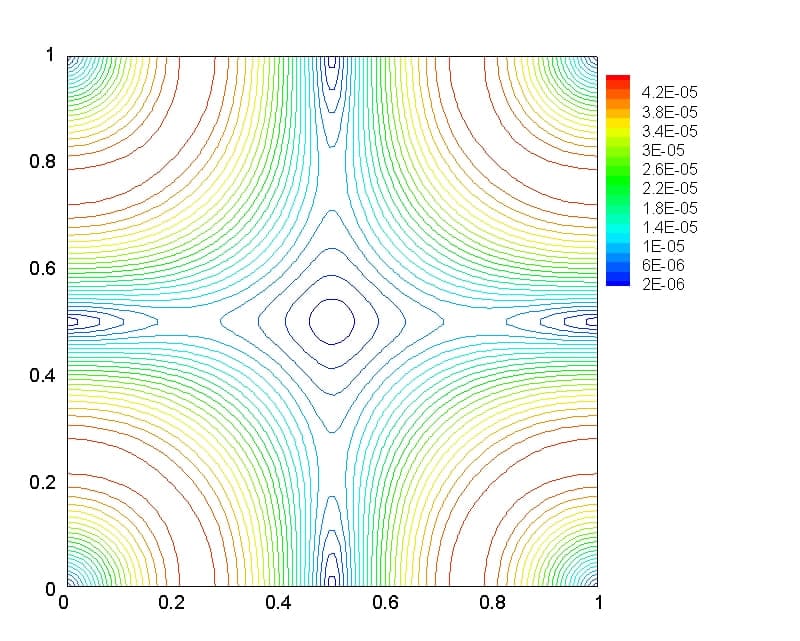}}
 {\includegraphics[height=2.1in]{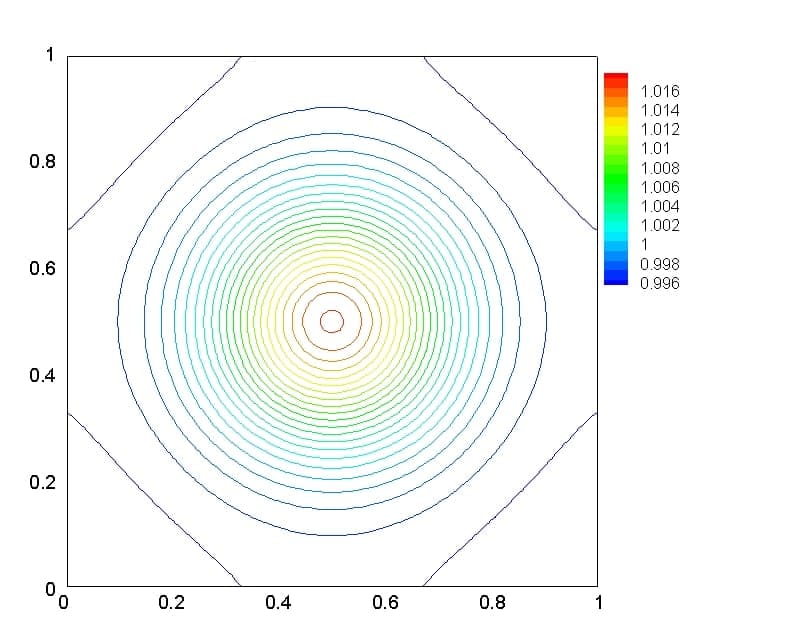}} \\

\centering
{\includegraphics[height=2.1in]{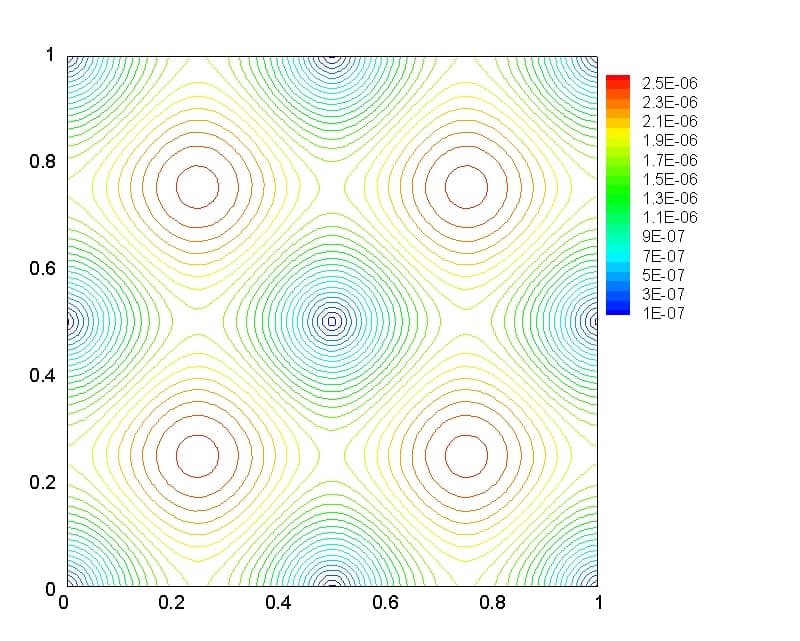}}
 {\includegraphics[height=2.1in]{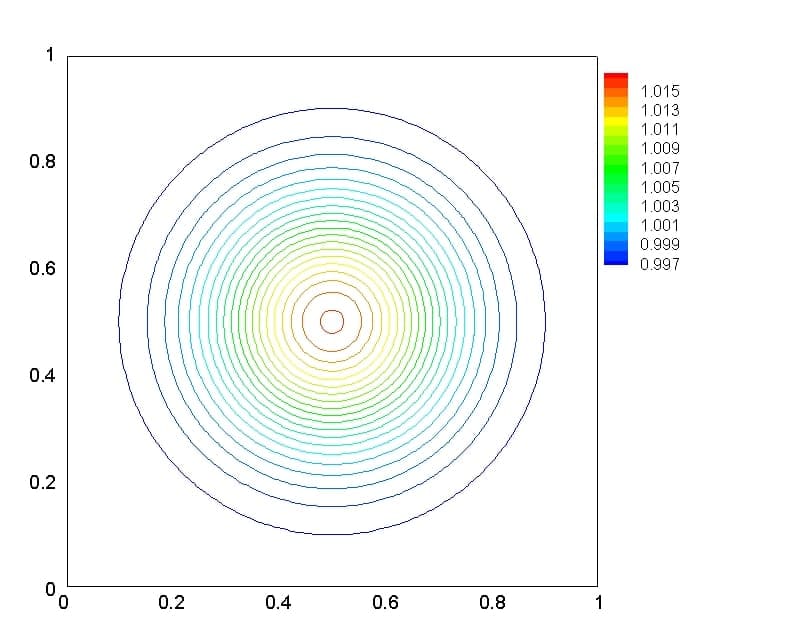}} \\

\centering
{\includegraphics[height=2.1in]{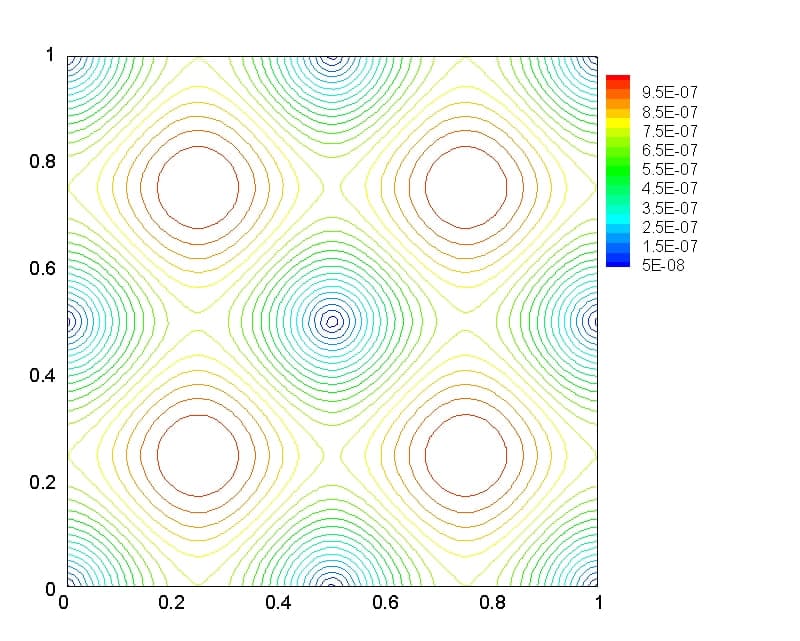}}
 {\includegraphics[height=2.1in]{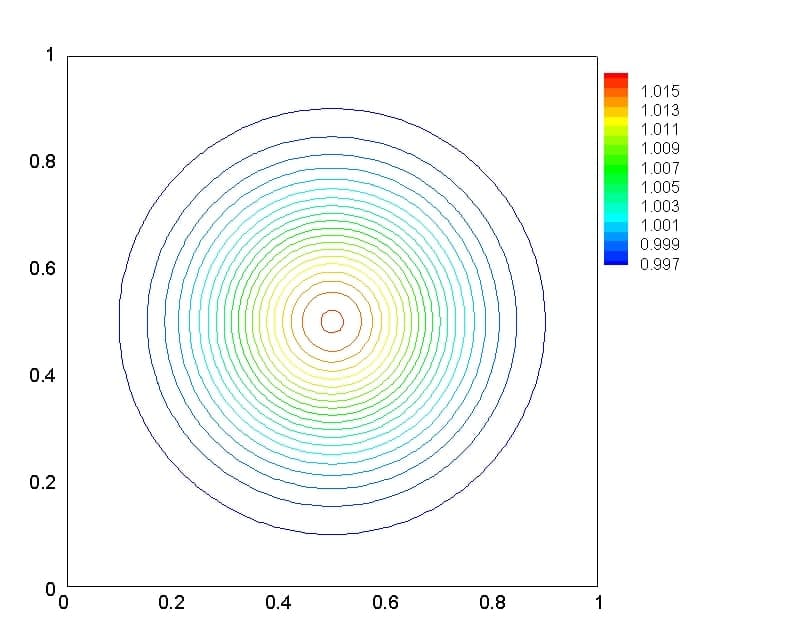}} 
    
\caption{Well-balanced solutions of the 2-D system in an expanding background with the background geometry $b(x,y)$ at $t=8$, $16$, $50$, $60$. Right column: Rescaled solution $\rhot$. Left column: Velocity magnitude $V$.}
\label{bxy-wb-exp}
\end{figure}


\subsection{Conclusion for an expanding background}

Based on the numerical experiments in this section, we have the following observations: 
\bei 

\item On a homogeneous geometry, both in one space and two space dimensions the rescaled density $\rhot$ converges to a constant, while the rescaled velocity $\ut$ converges to zero. 

\item On a non-homogeneous geometry described by a function $b=b(x)$ in one space dimension and $b=b(x,y)$ in two space dimensions, the asymptotic solution after rescaling coincides with this geometric function up to a multiplicative constant. 

\eei 

We thus reach the following conclusion and conjecture.

\begin{claim}[Compressible fluid flows on a future-expanding cosmological background] 
The asymptotic behavior of the solutions to the fluid model \eqref{EulerAbc700} posed on a future-expanding cosmological background is described as follows: 

\bei 

\item The solution $(\rho, u) = (\rho,u)(t, x)$ (with $t >0$) decays to zero as $t \to +\infty$:
\be
\lim_{t \to + \infty} \rho(t, x) =0, \qquad 
\lim_{t \to + \infty} u(t, x) =0, \qquad x \in [0,1]. 
\ee

\item {\bf Spatially homogeneous background.}
When the function $b$ is a constant, the asymptotic rescaled solution defined in \eqref{equa-urho-exp}
is a constant with vanishing velocity: $(\rhob, \ub) = (\rhob, 0)$.  
For sufficiently large times, the solution is not stationary but is approximately time-periodic. The solution propagates at the sound wave speed $\pm k$.

\bei 

\item {\bf One space dimension.}  The rescaled density defined in \eqref{eq: Reu01} looks like two constant density states, both converging to the constant density $\rhot$, while the velocity $\ut$ looks like two linear parts separated by two discontinuities and both linear pieces are converging to $\ub=0$.  

\item {\bf Two space dimensions.} Convergence to constant states is also observed. 

\eei 

\item {\bf General background.} On a spatially non-homogeneous background the rescaled density $\rhot$ defined in \eqref{equa-urho-exp} approaches a non-trivial limit as $t \to + \infty$ of the form 
\be
\rhob(x) = \lim_{t \to + \infty} \rhot(t, x) = C_1 b^2(x), 
\ee
where $C_1>0$ is a constant.
\eei 
\end{claim}

  
\section{Global dynamics on a future-contracting background}
\label{thesection:6}

\subsection{Spatially homogeneous background in one space dimension} 

In the future-contracting case, we now demonstrate that the density $\rho \to +\infty$ blows up, while the velocity $u$ approaches the light speed value or zero. 

\paragraph{A test with variable density data} We choose the same initial data as test $2$ of  the expanding background to be
\bel{eq:Ini0}
u_0(x)=0, \qquad \rho_0(x) = 1+ \sin\Big({6\over7} \pi x\Big) \cos\Big({7\over 2} \pi x).
\ee

We take the exponent $\kappa =2$, $N = 500$ with CFL $= 0.3$, and light speed to be unit. In Figure~\ref{fig:Con0} and ~\ref{fig:Con1}, we plot the evolution of the solution $u$ and $\rho$ as $t \to 0$ with sound speed $k = 0.5$.
We find that $\rho \to +\infty$ and $u \to \pm 1$, as $t \to 0$. We also plot the solution $u$ and $\rho$ when the sound speed $k =0.1$ in Figure~\ref{Fig: Euler-HLL-Con4}, which also shows that $\rho \to +\infty$ and $u \to \pm 1$, as $t \to 0$.
However, $u \to  \pm 1$ when $k = 0.1$  appears earlier than the case $k = 0.5$.  In Figure~\ref{Fig: Euler-HLL-Con5}, we plot solution $u$ and $\rho$ with $k = 0.9$ at $t = -10^{-6}$.
Observe that the density $\rho$ blows up, while the velocity 
$u$ closes to $0$.

\paragraph{Rescaling the numerical solution}

We now plot the rescaled solution  $\rhot$ defined in \eqref{equa:resc499} with the initial data given by \eqref{eq:Ini0} and $k = 0.5$; see Figure~\ref{Fig: Euler-HLL-ResCon}. We observe that the asymptotic solution $\rhot$ approaches a bounded and stationary limit.

\begin{figure}
\centering
\begin{minipage}{0.5\linewidth}
  \centerline{\includegraphics[width=5.6in]{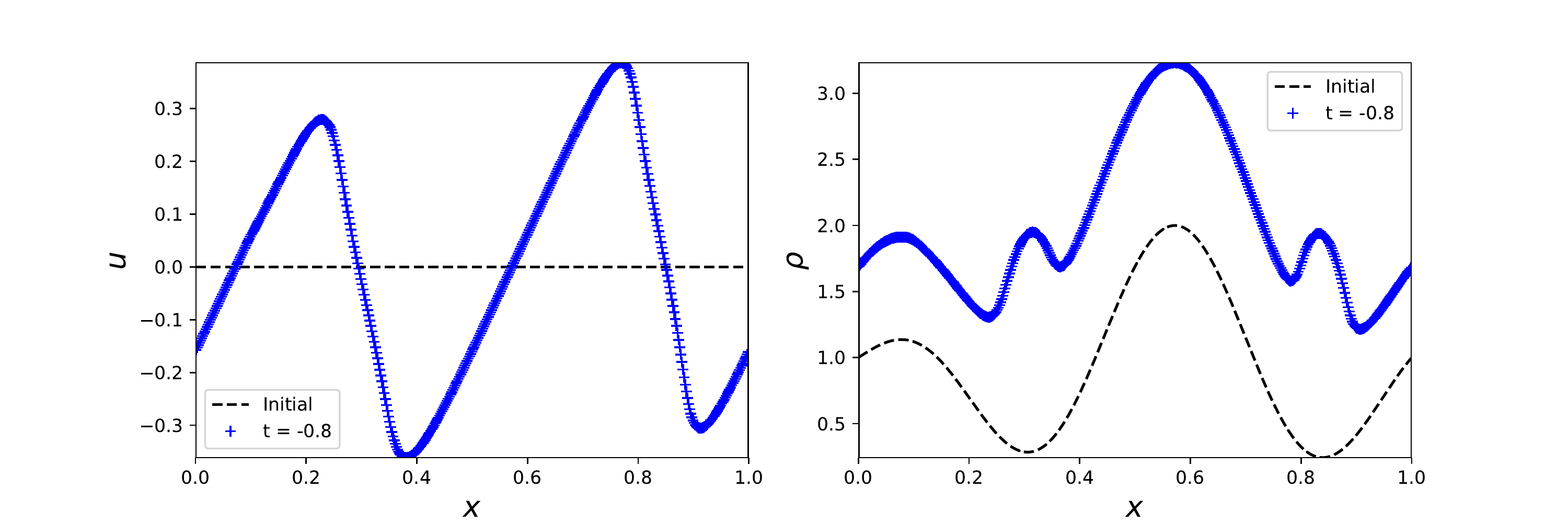}}
\end{minipage}
\vfill
\begin{minipage}{.5\linewidth}
  \centerline{\includegraphics[width=5.6in]{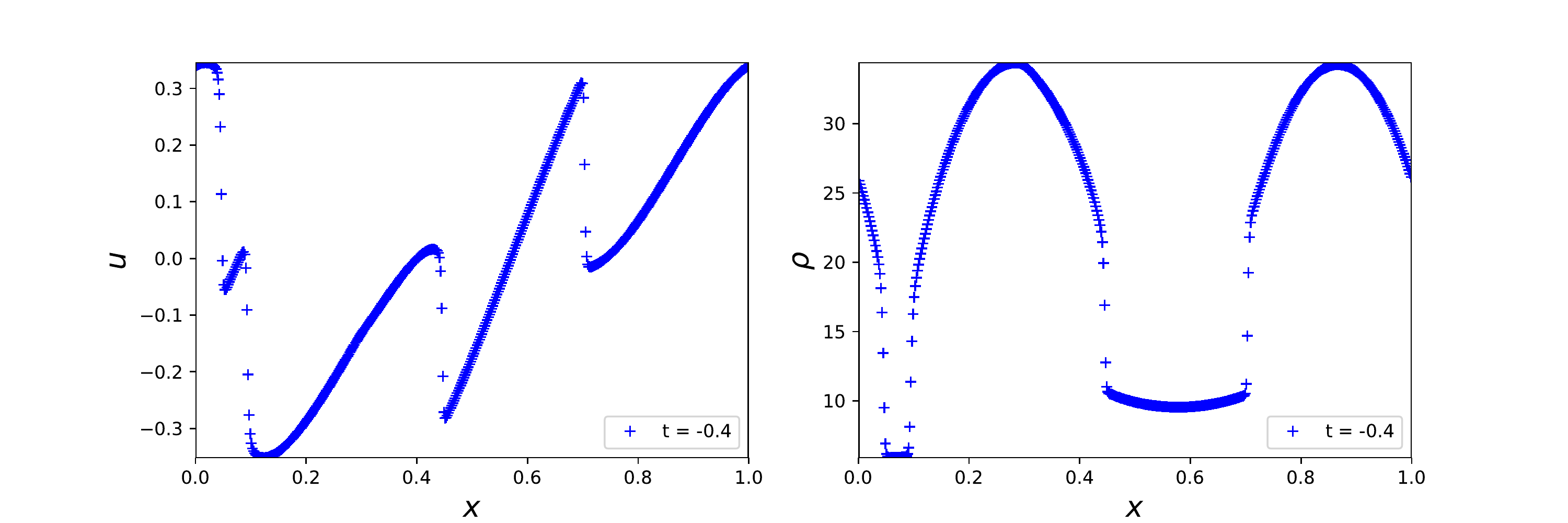}}
\end{minipage}
\vfill
\begin{minipage}{.5\linewidth}
  \centerline{\includegraphics[width=5.6in]{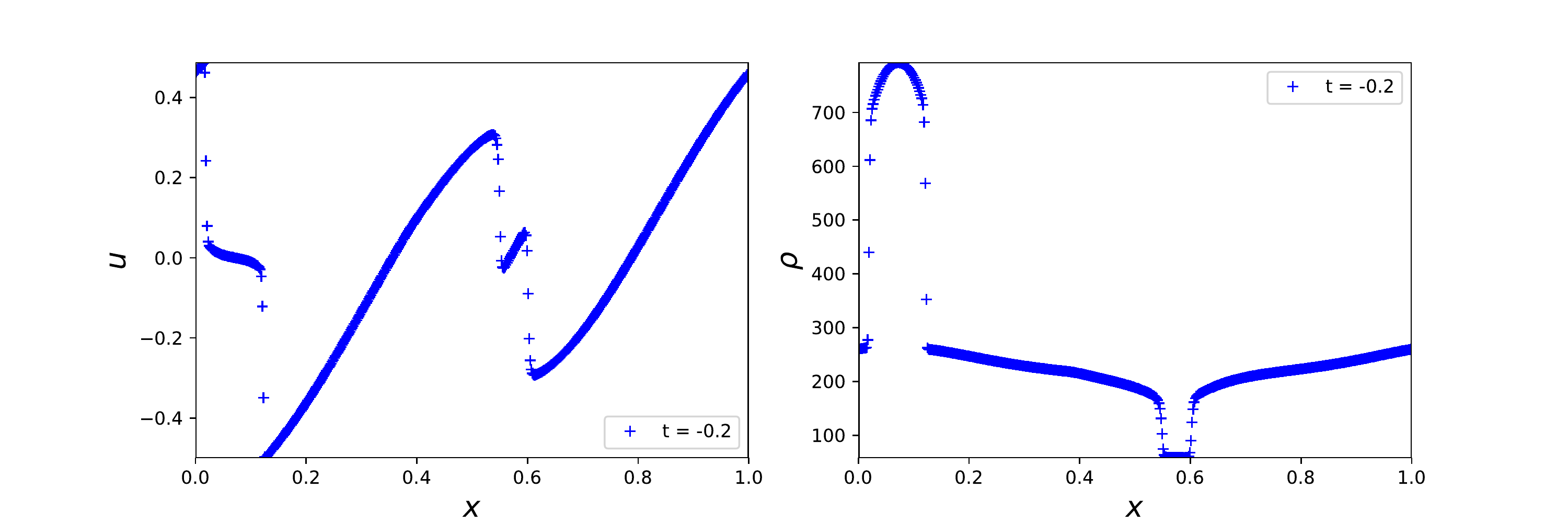}}
\end{minipage}

\vfill
\begin{minipage}{.5\linewidth}
  \centerline{\includegraphics[width=5.6in]{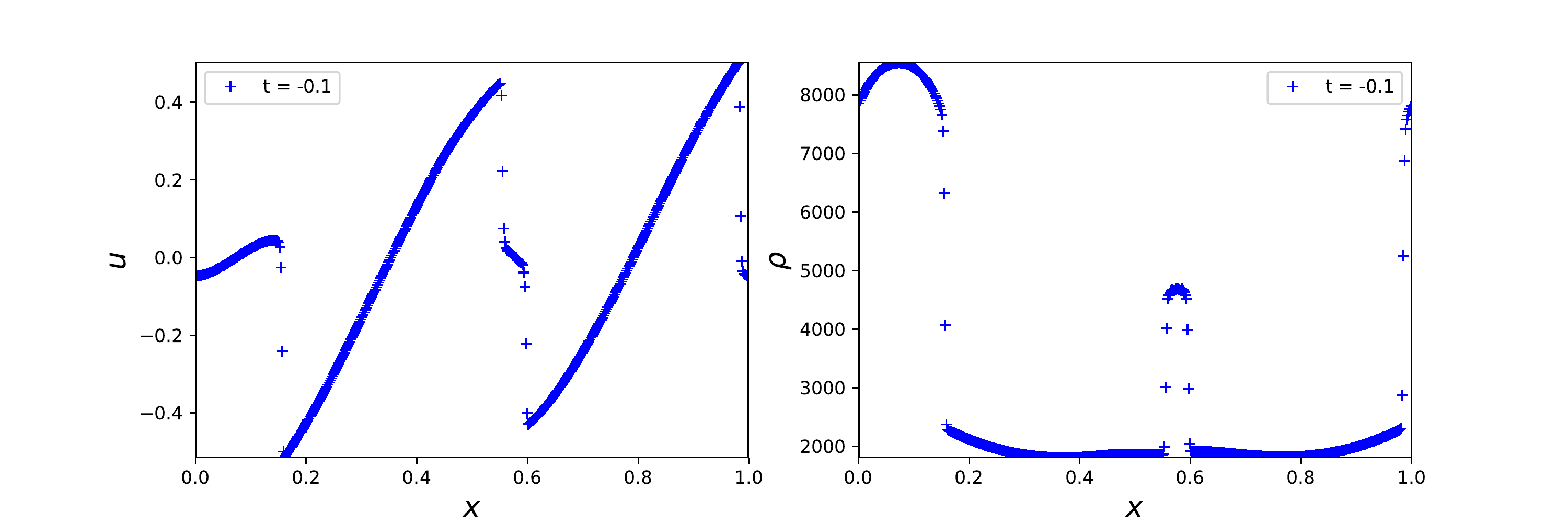}}
\end{minipage}

\caption{The evolution of solution on a contracting background with $k =0.5$.}
\label{fig:Con0}
\end{figure}

\begin{figure}
\centering

\begin{minipage}{.5\linewidth}
  \centerline{\includegraphics[width=5.6in]{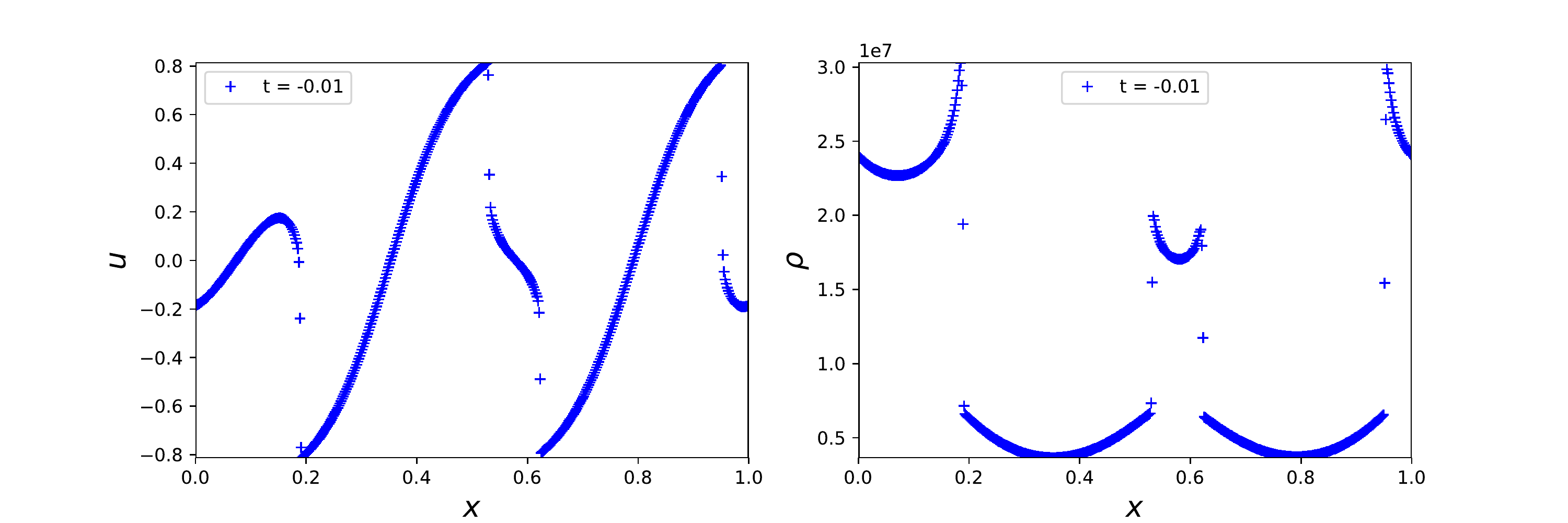}}
\end{minipage}

\vfill
\begin{minipage}{.5\linewidth}
  \centerline{\includegraphics[width=5.6in]{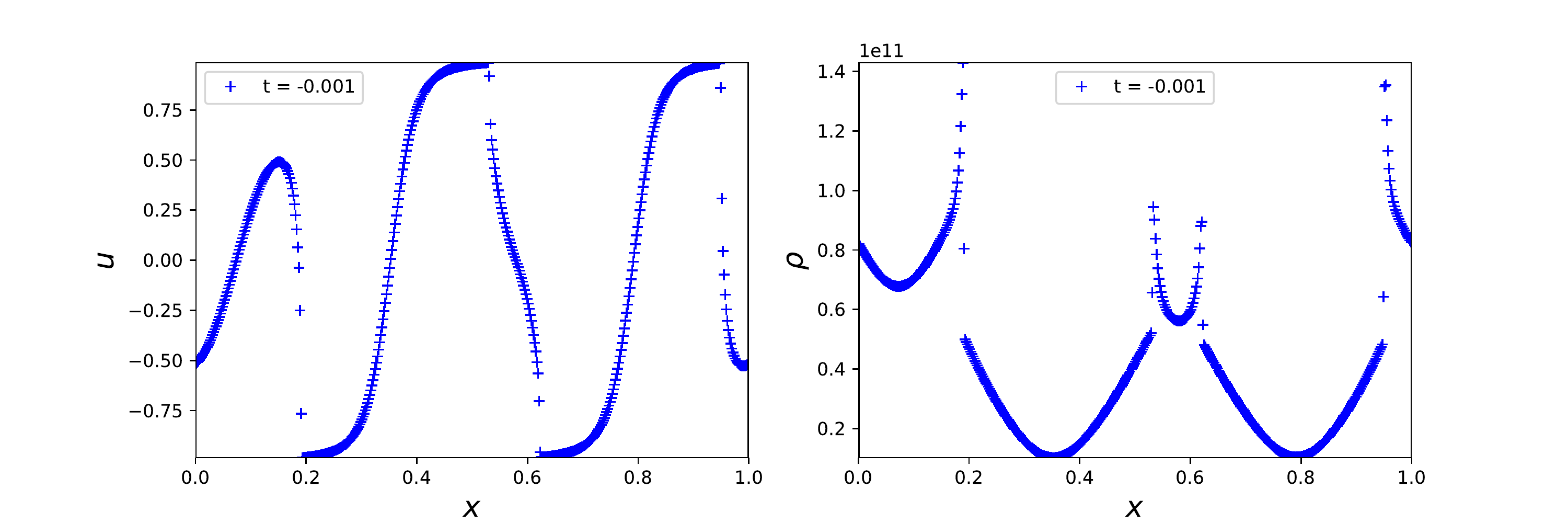}}
\end{minipage}

\vfill
\begin{minipage}{.5\linewidth}
  \centerline{\includegraphics[width=5.6in]{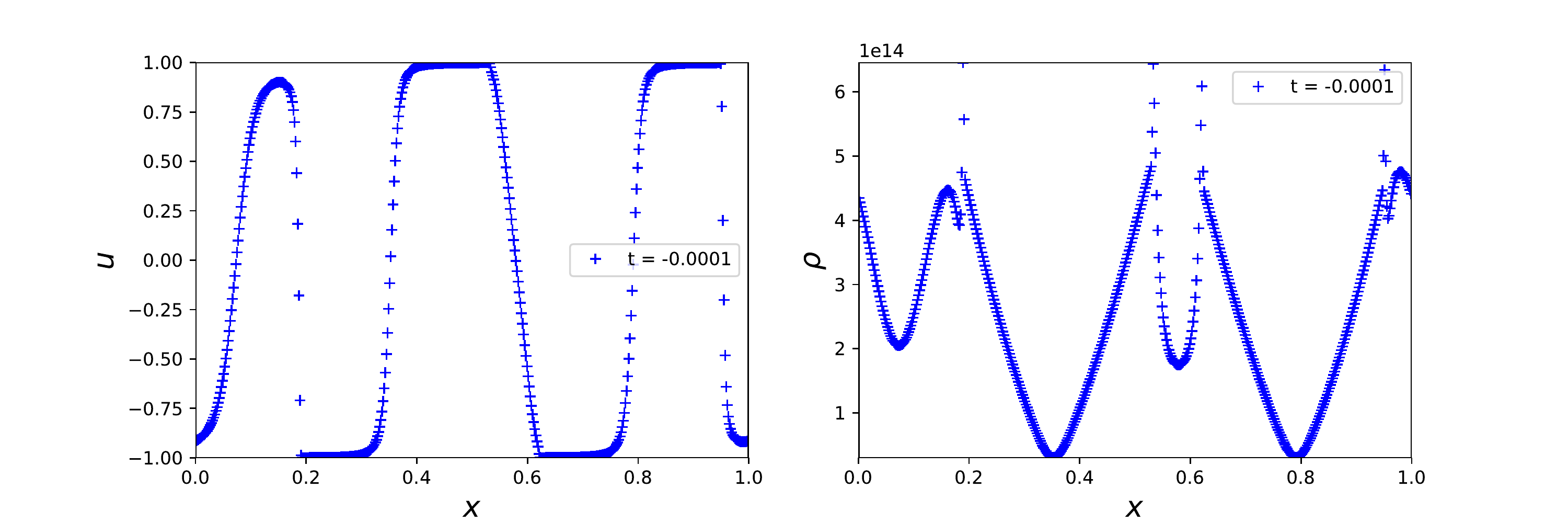}}
\end{minipage}

\vfill
\begin{minipage}{.5\linewidth}
  \centerline{\includegraphics[width=5.6in]{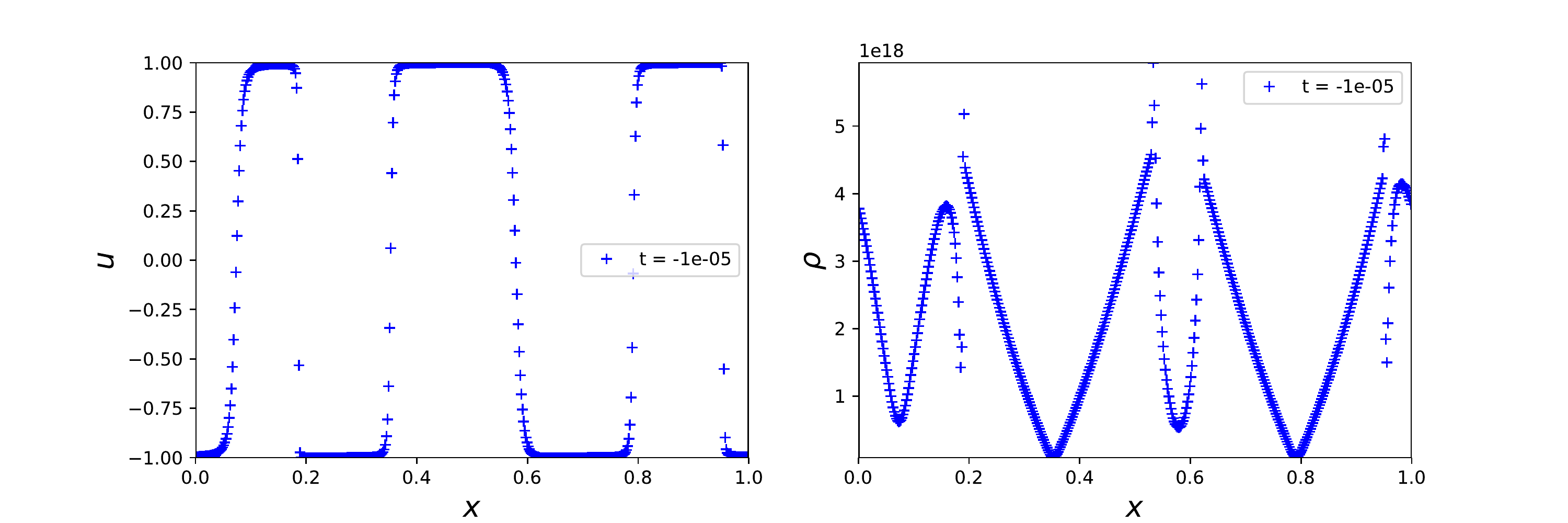}}
\end{minipage}

\caption{The evolution of solution on a contracting background with $k =0.5$.}
\label{fig:Con1}
\end{figure}




\begin{figure}[htbp]
\centering
\epsfig{figure = 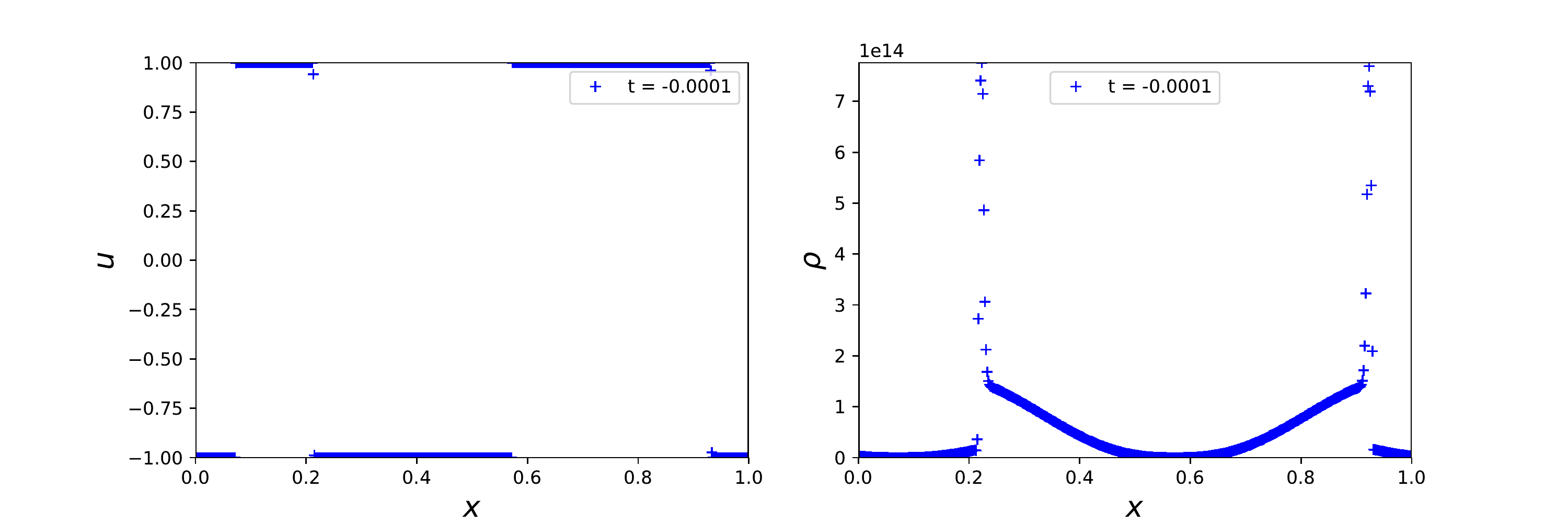,height = 2.1 in}  
\caption{The solution on a contracting background at $t = -10^{-4}$ with $k =0.1$.}
\label{Fig: Euler-HLL-Con4}
\end{figure}

\begin{figure}[htbp]
\centering
\epsfig{figure = 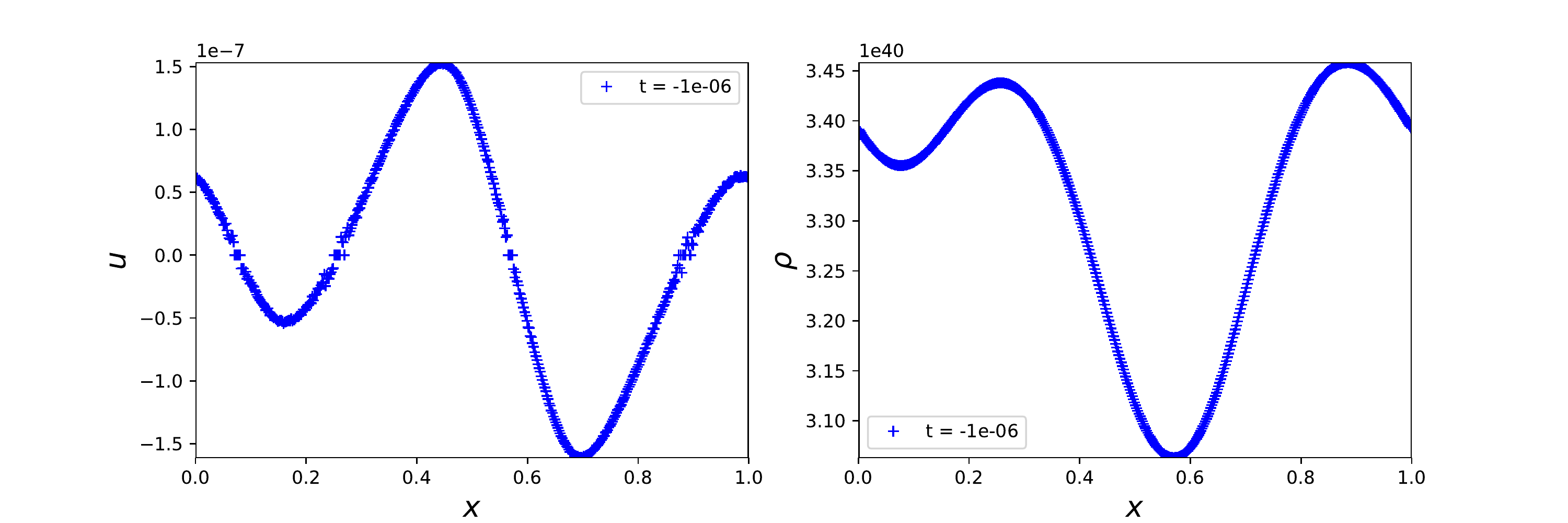,height = 2.1 in}  
\caption{The solution on a contracting background at $t = - 10^{-6}$ with $k =0.9$.}
\label{Fig: Euler-HLL-Con5}
\end{figure}


\begin{figure}[htbp]
\centering
\epsfig{figure = 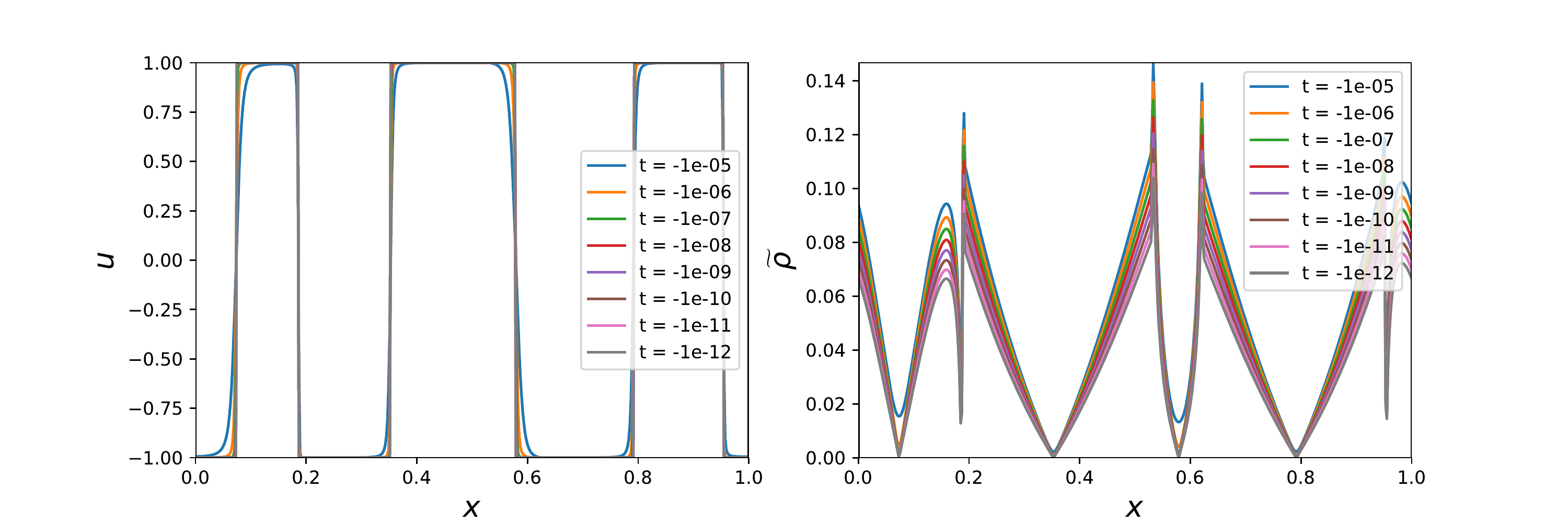,height = 2.1 in}  
\caption{The solution $u$ and rescaled solution $\rhot$ on a contracting background with $k = 0.5$.}
\label{Fig: Euler-HLL-ResCon}
\end{figure}


\subsection{Flows on a spatially homogeneous background in two space dimensions} 

In the two-dimensional future-contracting case, we now demonstrate that the density $\rho \to +\infty$ blows up, while the velocity magnitude $V$ approaches the light speed value. 

\paragraph*{Test 1: Symmetrical initial data.}

We choose the initial data \eqref{eq:initial2D} posed at $t_0 = -1$ and defined in the domain $[0,1]\times[0,1]$. Recall that in all two-dimensional tests, the exponent $\kappa =2$, the sound speed $k= 0.5$, CFL $= 0.5$, the light speed to be a unit, and the grid is $[100\times100]$.
In Figure~\ref{a(t)-con}, we plot the rescaled solution $\rhot$ and velocity magnitude  $t=-0.5$, $-10^{-1}$, $-10^{-3}$, $-10^{-5}$, respectively. The results which are obtained by the standard HLL scheme demonstrate that the solution $\rho \to +\infty$ and $V \to 1$ as $t$ increases.

\begin{figure}[htbp]
\centering
 {\includegraphics[height=2.1in]{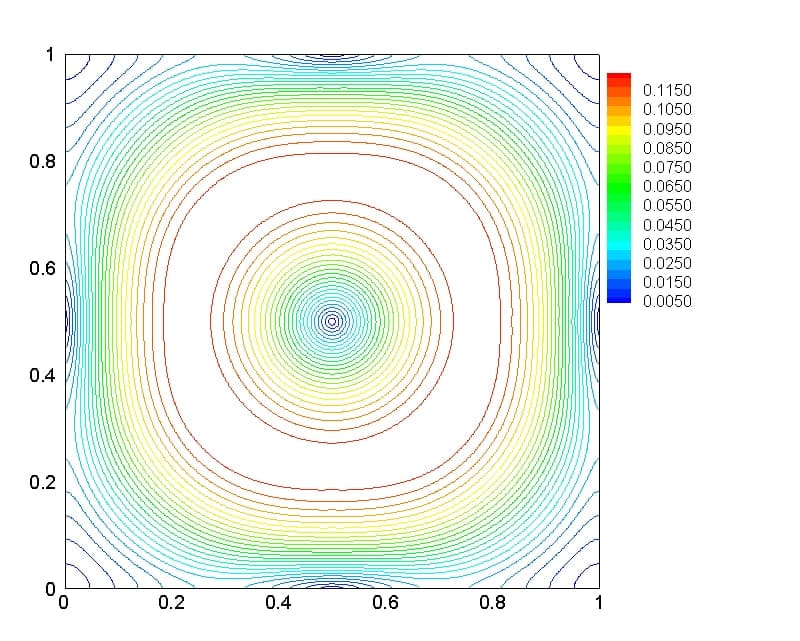}}
 {\includegraphics[height=2.1in]{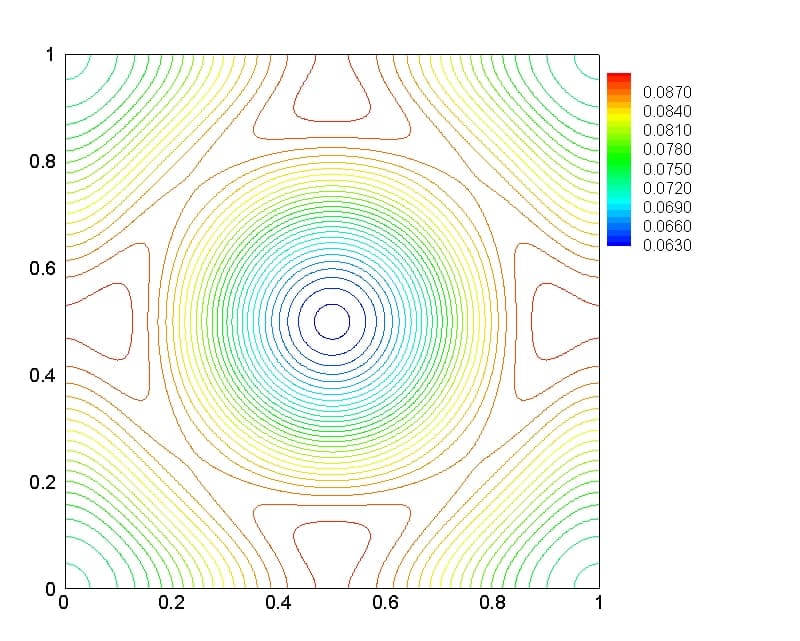}} \\
  
\centering
{\includegraphics[height=2.1in]{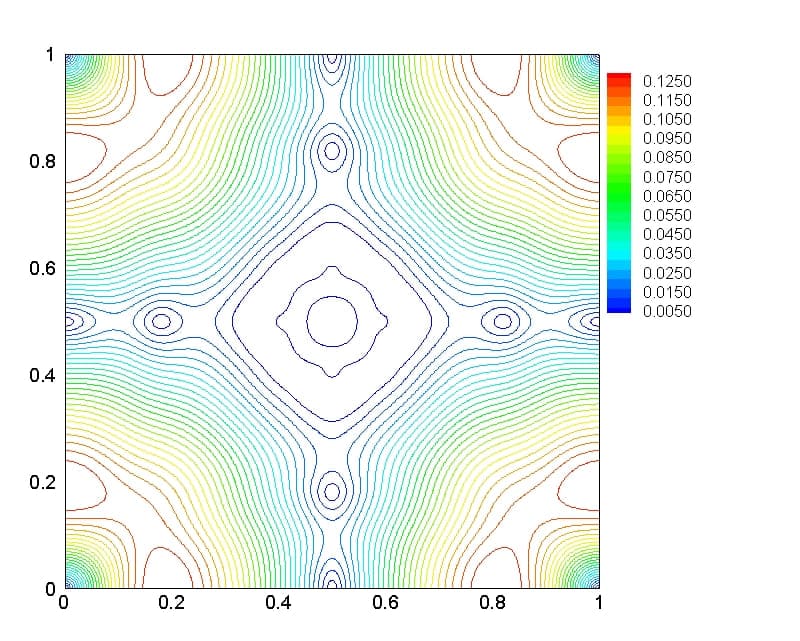}}
 {\includegraphics[height=2.1in]{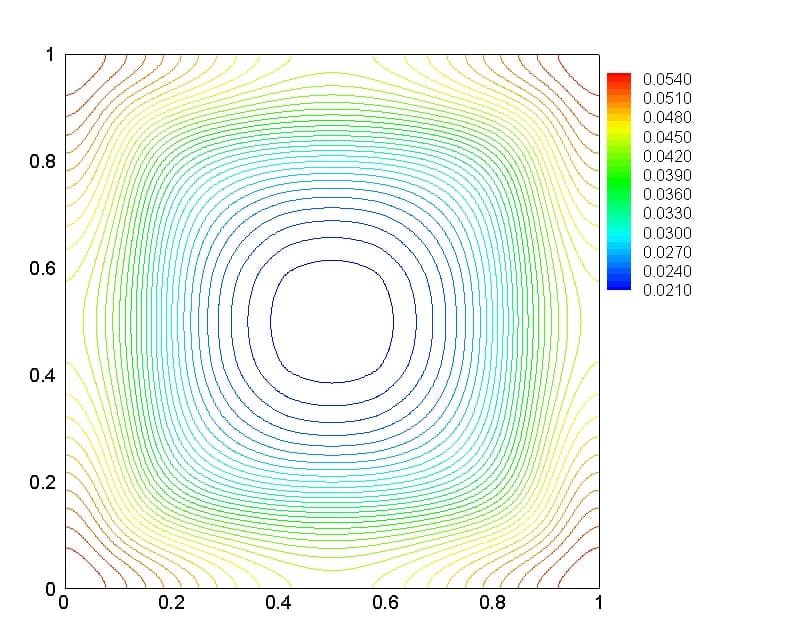}} \\

\centering
{\includegraphics[height=2.1in]{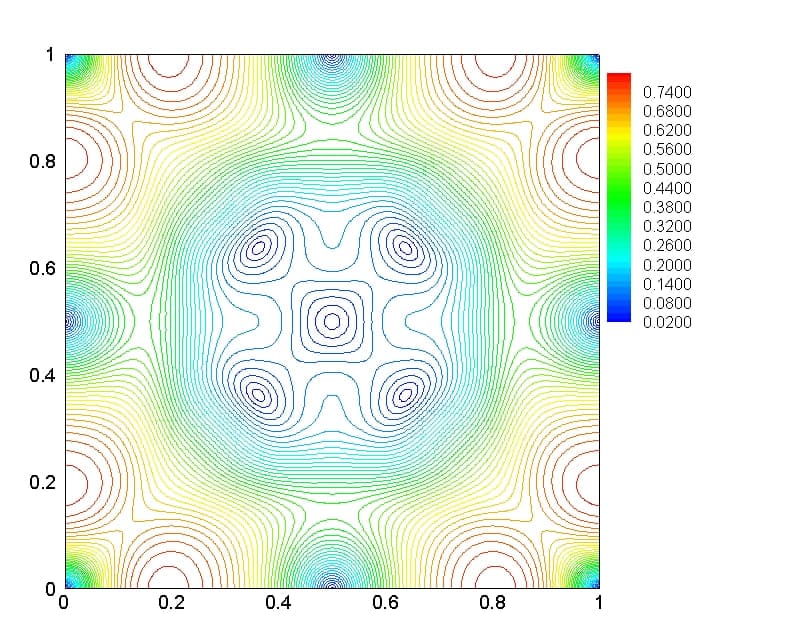}}
 {\includegraphics[height=2.1in]{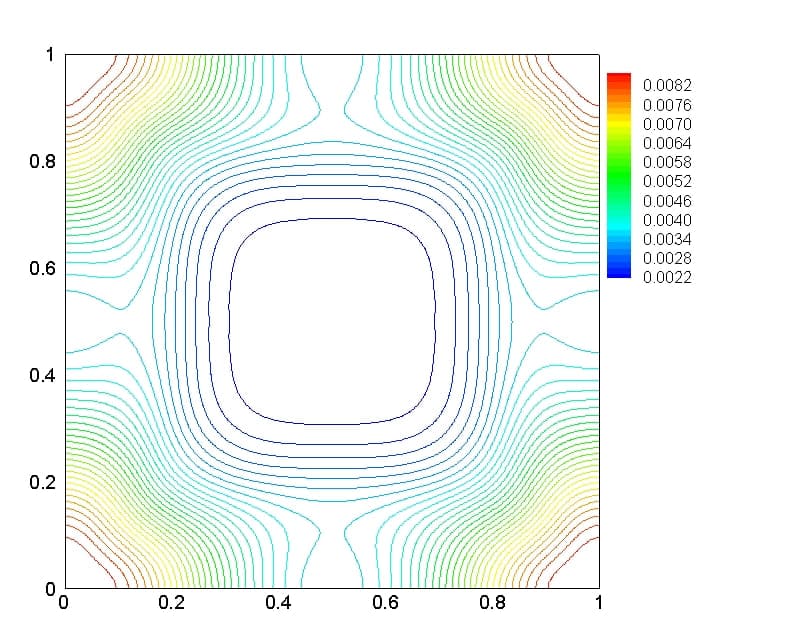}} \\

\centering
{\includegraphics[height=2.1in]{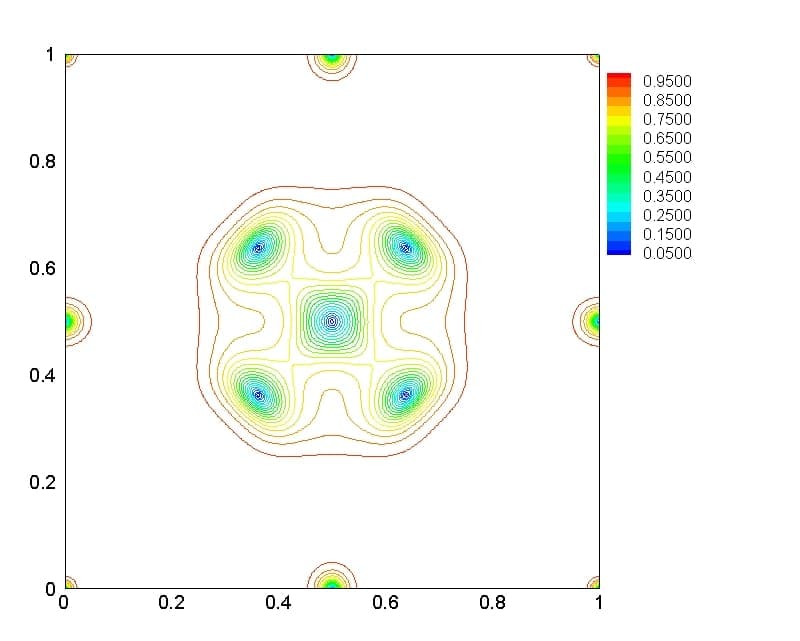}}
 {\includegraphics[height=2.1in]{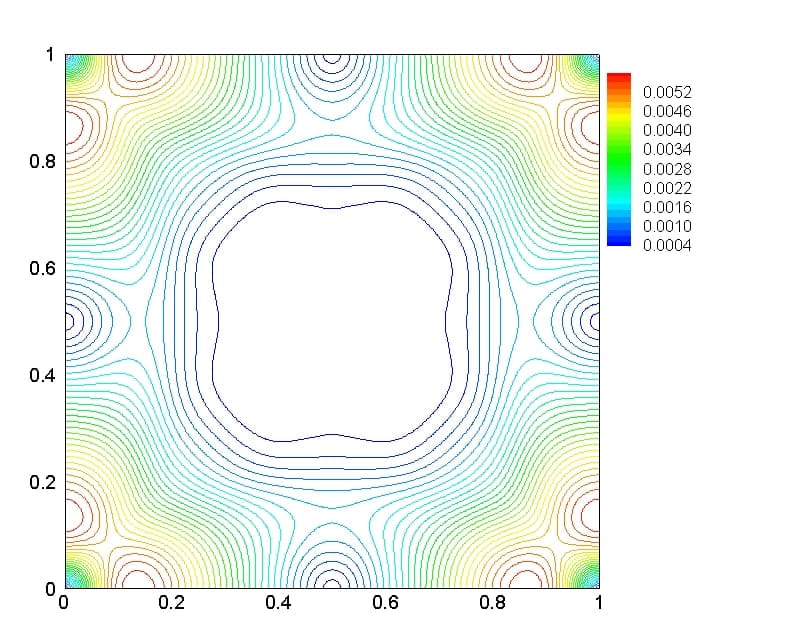}} 
    
\caption{Solutions of the 2-D spatially homogeneous system in a contracting background at $t=-0.5$, $-10^{-1}$, $-10^{-3}$, $-10^{-5}$. Right column: Rescaled solution $\rhot$. Left column: Velocity magnitude $V$.}
\label{a(t)-con}
\end{figure}


\subsection{Flows on an non-homogeneous background in one space dimension}  

\paragraph{A test with variable density data} 
We choose the function $b(x)$ as follows:
\bel{eq:bxb11}
b(x) = 1+ 0.01 \big(\sin(6 \pi x) + \cos(2\pi x)\big),
\ee
 and the initial data to be
\bel{eq: sssv000}
u_0=  0, \qquad \rho_0(x) = b^2(x).
\ee
at $t = -1$.
We take $\kappa = 2$. We can obtain the same result as the homogeneous case  on a contracting  background, that is 
$\rho \to +\infty$ and $u \to \pm 1$ or $0$, as $t \to 0$. We plot the solution $u$ and $\rho$ at $t = -10^{-7}$ with $k = 0.3$, 
see Figure~\ref{Fig: Euler-wb-Con}. Observe that in this case $\rho \to +\infty$ and $u \to \pm 1$, as $t \to 0$. Moreover, in Figure~\ref{Fig: Euler-wb-Con1},
we plot the solution at $t = -10^{-7}$ with $k = 0.9$. We observe that $\rho \to +\infty$ and $u \to 0$.

\begin{figure}[htbp]
\centering
\epsfig{figure = 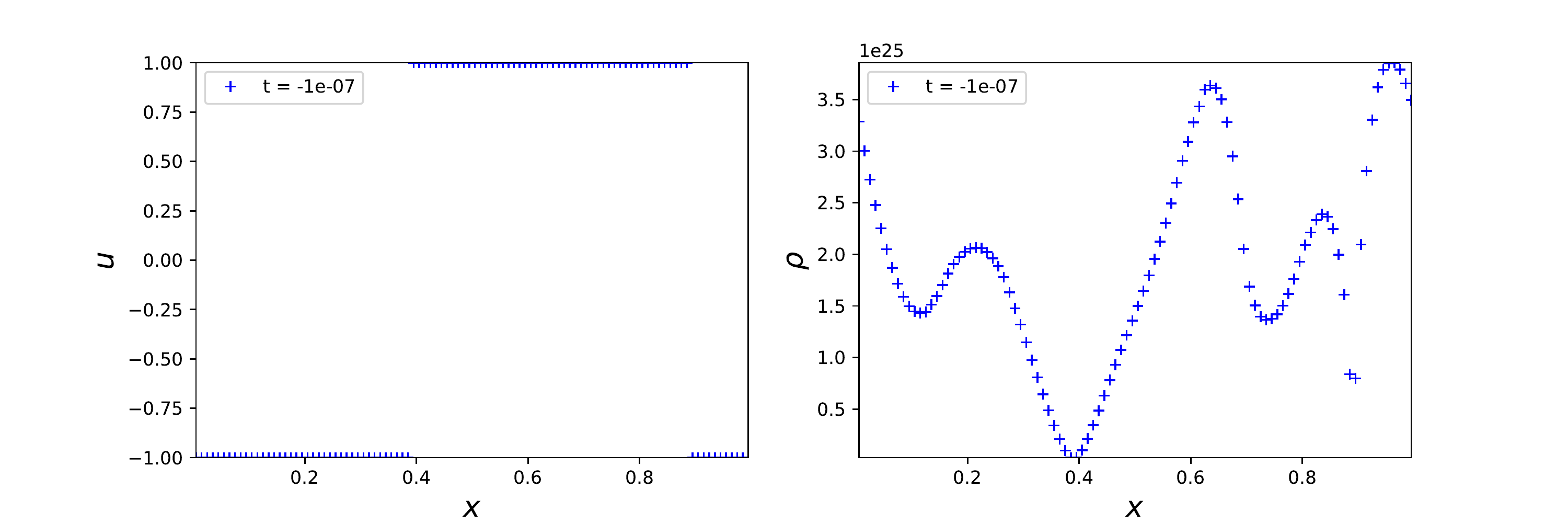,height = 2.1 in}  
\caption{The solution $u$ and $\rho$ on a contracting background with $k = 0.3$.}
\label{Fig: Euler-wb-Con}
\end{figure}

\begin{figure}[htbp]
\centering
\epsfig{figure = 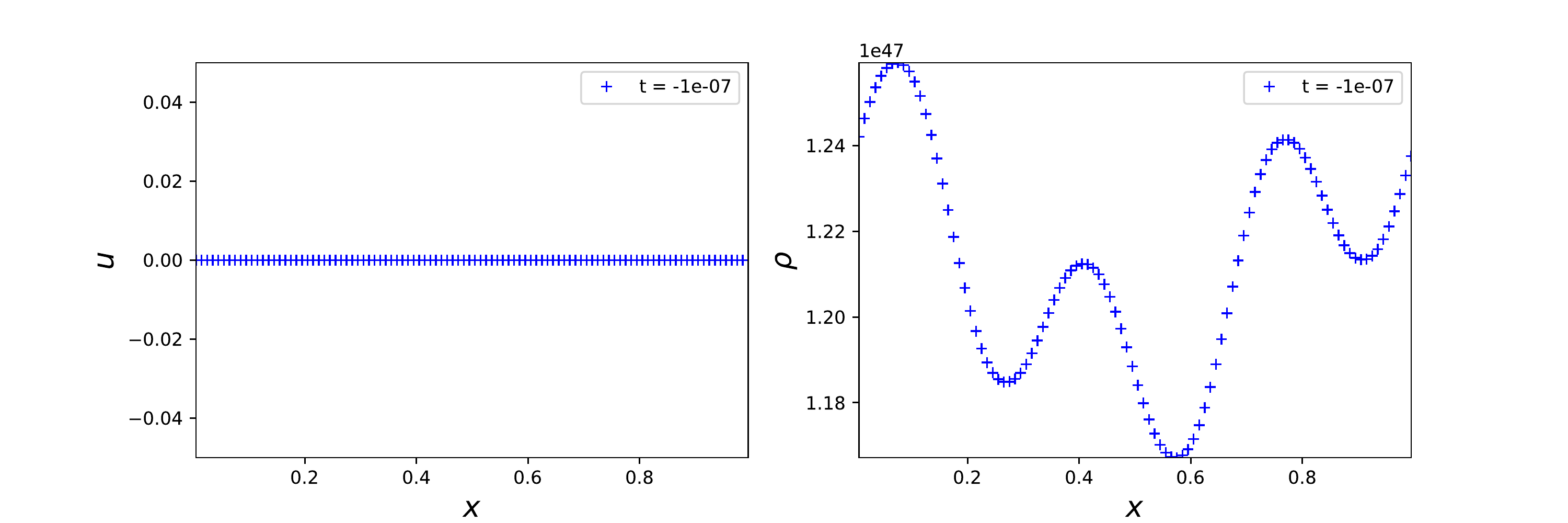,height = 2.1 in}  
\caption{The solution $u$ and $\rho$ on a contracting background with $k = 0.9$.}
\label{Fig: Euler-wb-Con1}
\end{figure}


\paragraph{Rescaling the numerical solution}

We now plot the rescaled solution  $\rhot$ defined in \eqref{Fig: Euler-Res-Con}. We observe that the asymptotic solution $\rhot$ approaches a bounded and stationary limit.

\begin{figure}[htbp]
\centering
\epsfig{figure = 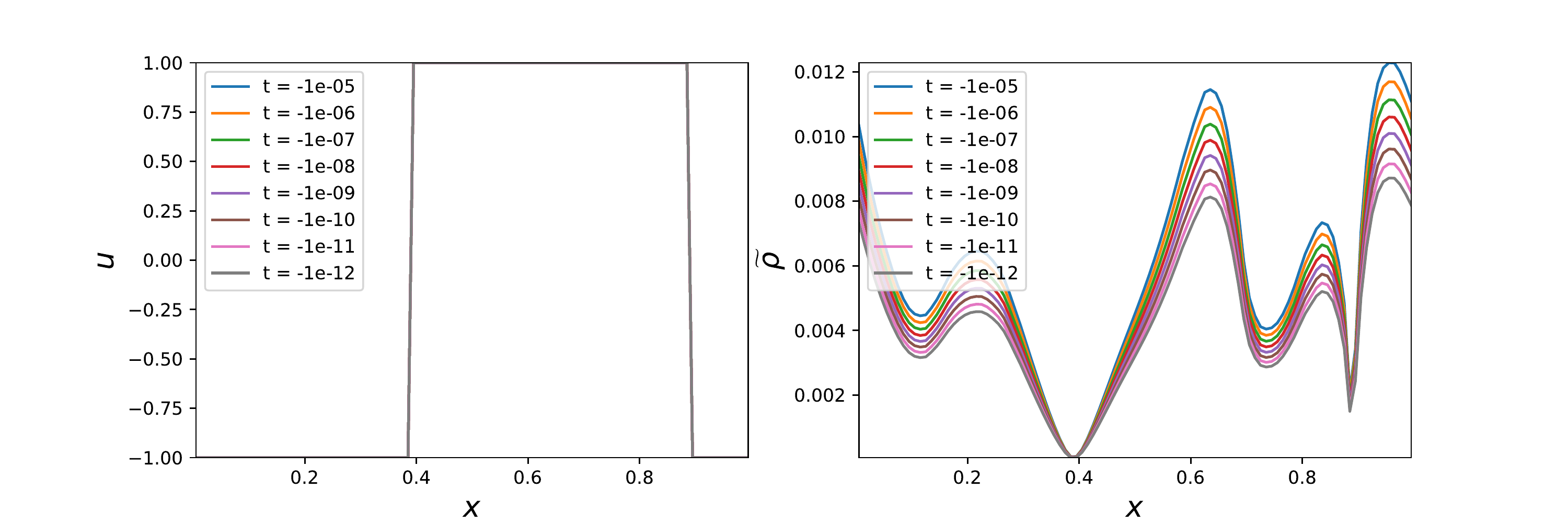,height = 2.1 in}  
\caption{The solution $u$ and rescaled solution $\rhot$ on a contracting background with $k = 0.3$.}
\label{Fig: Euler-Res-Con}
\end{figure}



\subsection{Flows on an non-homogeneous background in two space dimensions} 

\paragraph*{Test 1: Point symmetrical initial data in a contracting background.}

In this test we demonstrate the effect of $a(t) = |t|^\kappa$ when $t \in [-1, 0)$ over the background geometry $b(x,y)$ in \eqref{bxy2d}. The two-dimensional system is solved with the standard HLL scheme and our proposed one. In Figures~\ref{bxy-nwb-con} and~\ref{bxy-wb-con}, we plot the rescaled solution $\rhot$ and velocity magnitude $V$ at $-10^{-1}$, $-10^{-3}$, $-10^{-5}$, $-10^{-8}$, respectively. The results of both schemes demonstrate that the solution $\rho \to +\infty$, and velocity magnitude $V \to 1$ as $t$ increases. 

\begin{figure}[htbp]
\centering
{\includegraphics[height=2.1in]{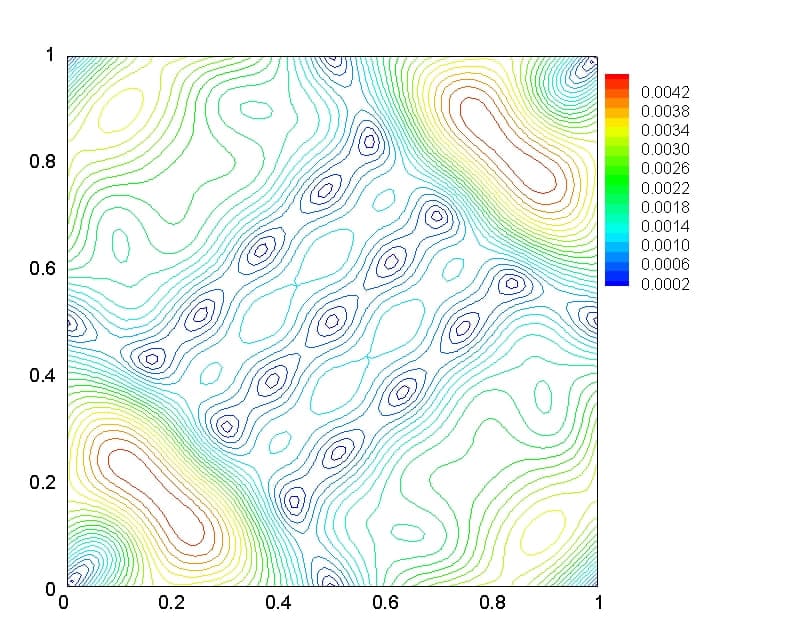}}
 {\includegraphics[height=2.1in]{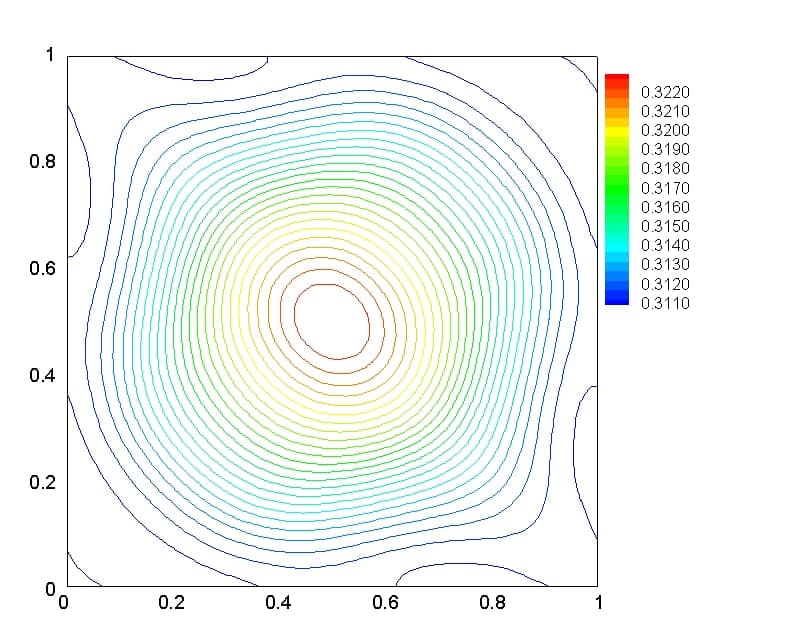}} \\
  
\centering
{\includegraphics[height=2.1in]{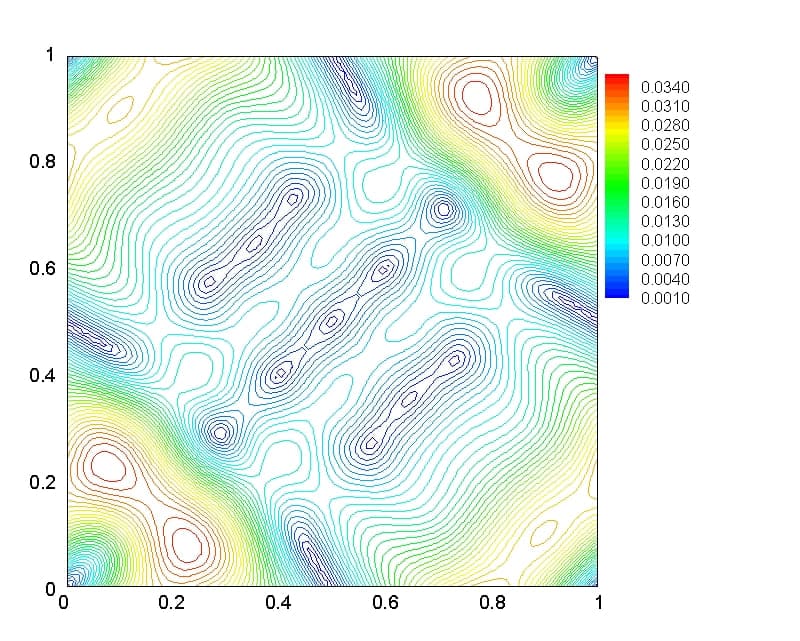}}
 {\includegraphics[height=2.1in]{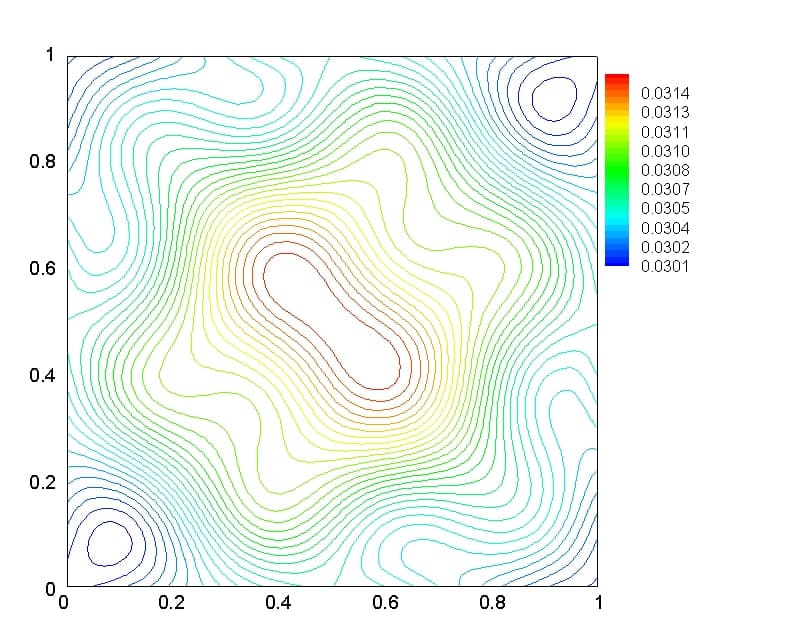}} \\

\centering
{\includegraphics[height=2.1in]{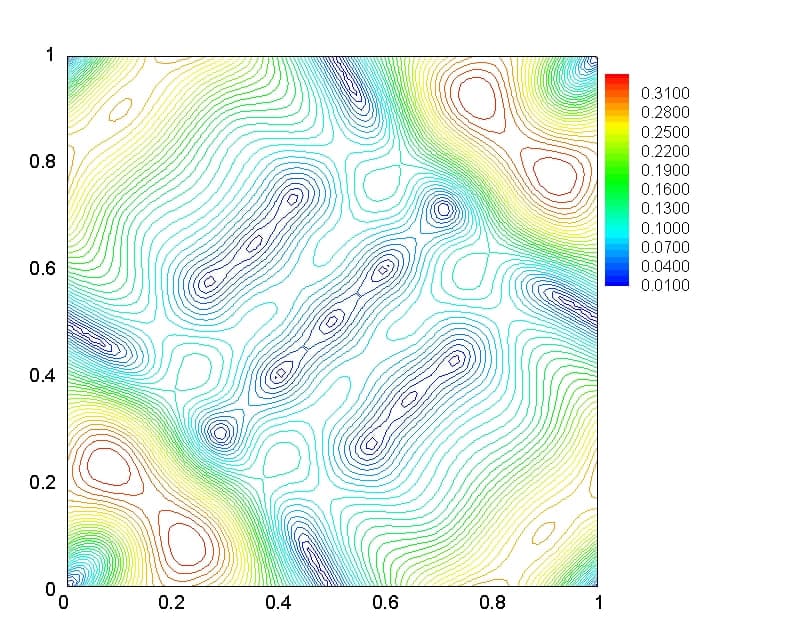}}
 {\includegraphics[height=2.1in]{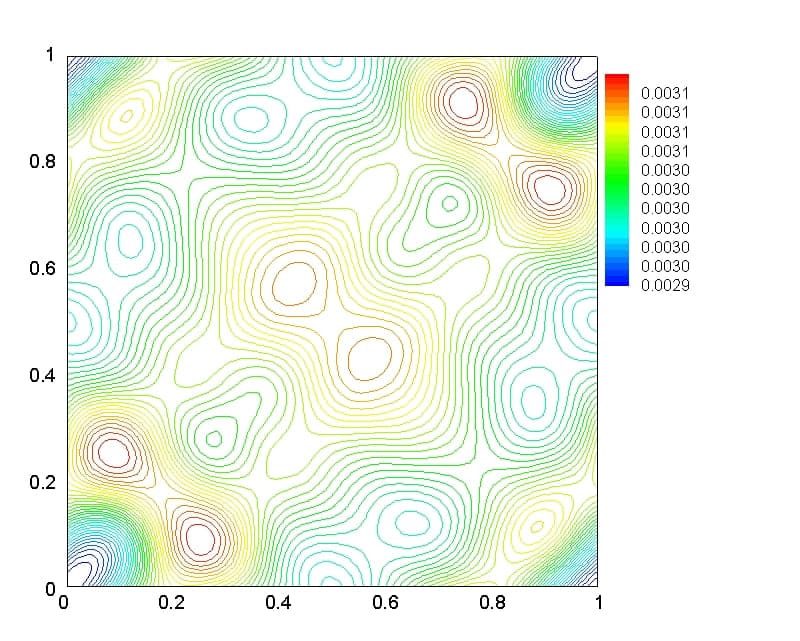}} \\

\centering
{\includegraphics[height=2.1in]{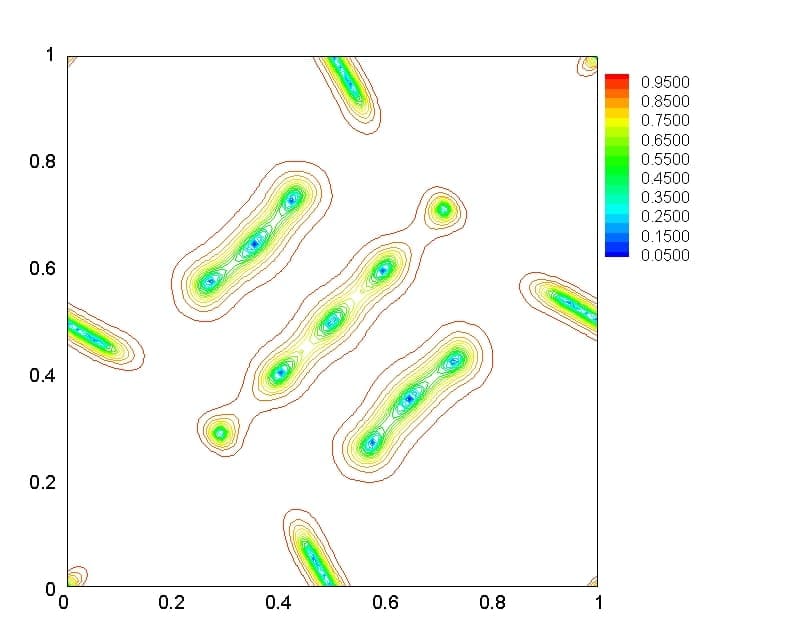}}
 {\includegraphics[height=2.1in]{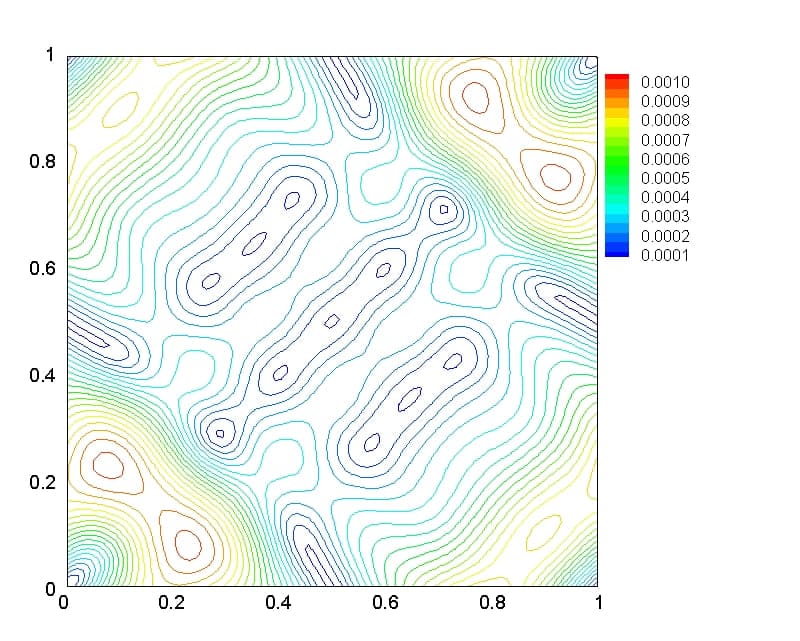}} 
    
\caption{Non-well-balanced solutions of the 2-D system in a contracting  background with the background geometry $b(x,y)$ at $t=-10^{-1}$, $-10^{-3}$, $-10^{-5}$, $-10^{-8}$. Right column: Rescaled solution $\rhot$. Left column: Velocity magnitude $V$.}
\label{bxy-nwb-con}
\end{figure}

\begin{figure}[htbp]
\centering
{\includegraphics[height=2.1in]{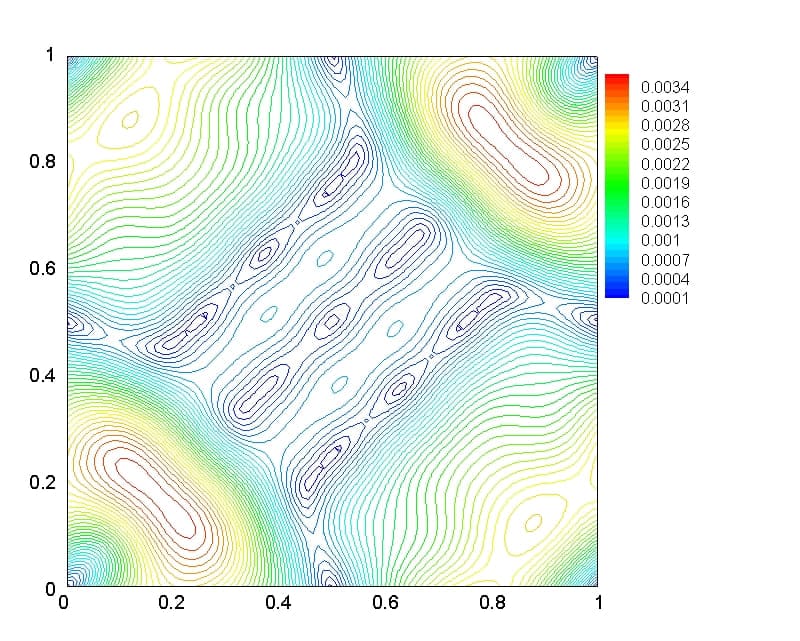}}
 {\includegraphics[height=2.1in]{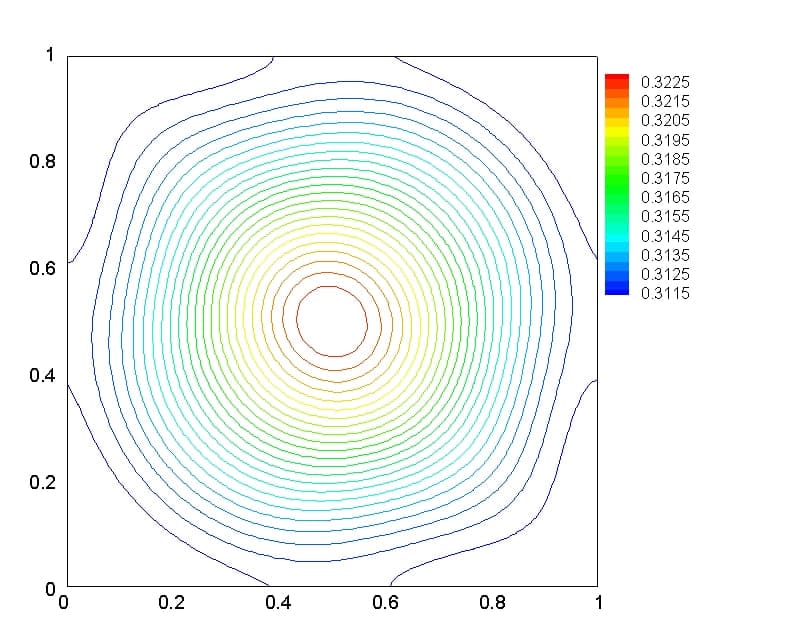}} \\
  
\centering
{\includegraphics[height=2.1in]{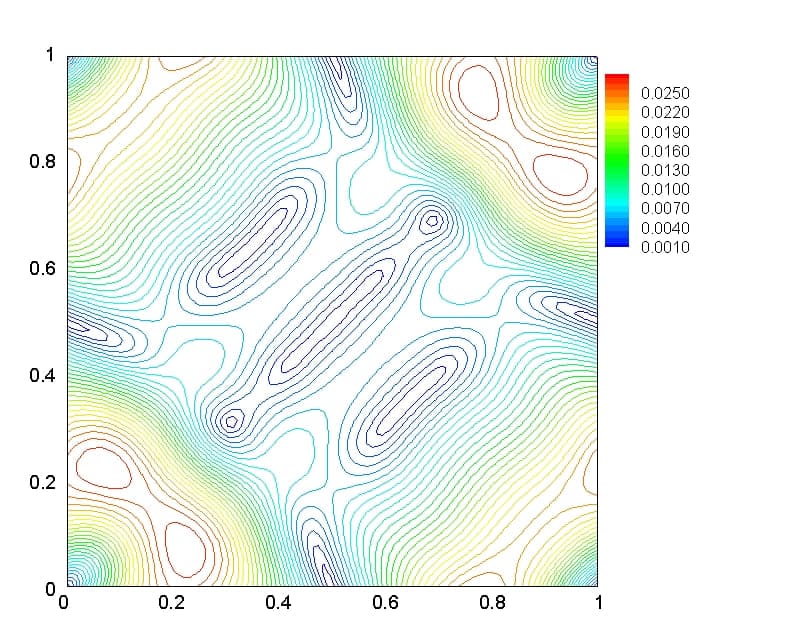}}
 {\includegraphics[height=2.1in]{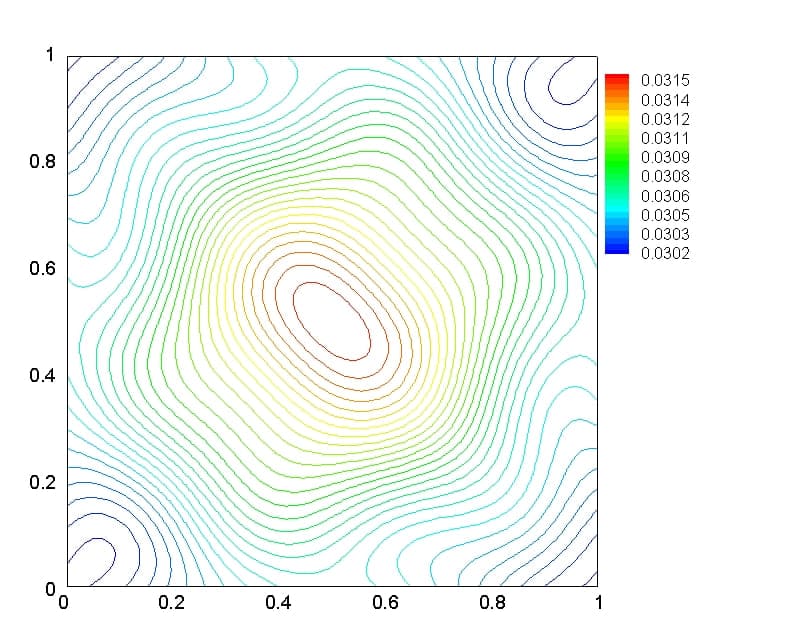}} \\

\centering
{\includegraphics[height=2.1in]{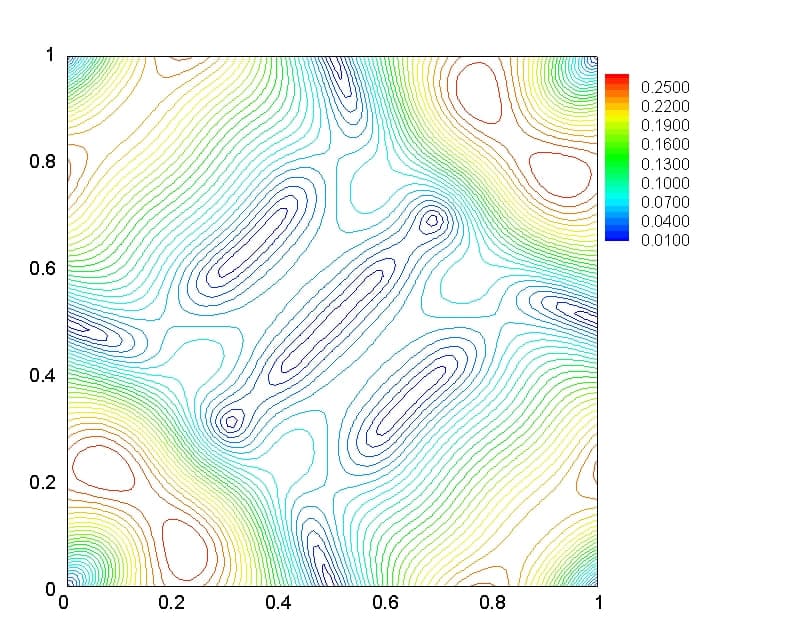}}
 {\includegraphics[height=2.1in]{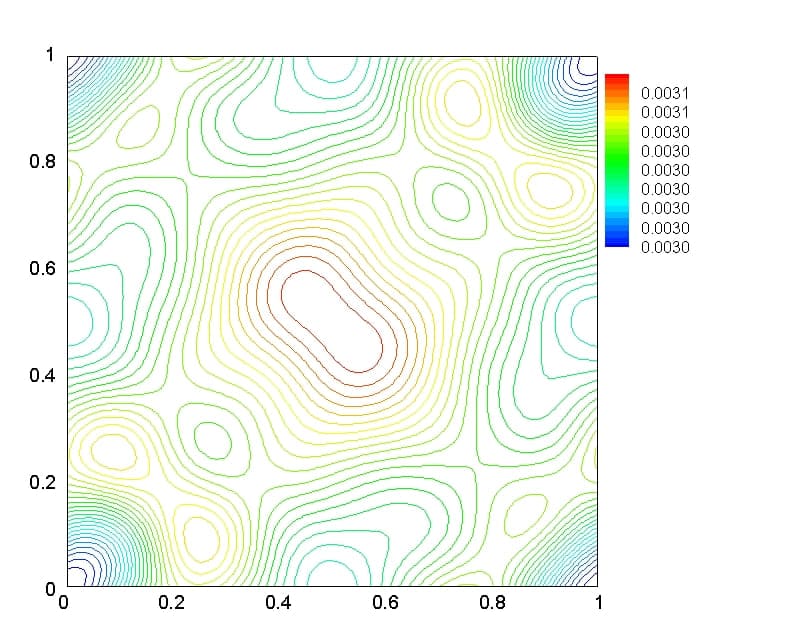}} \\

\centering
{\includegraphics[height=2.1in]{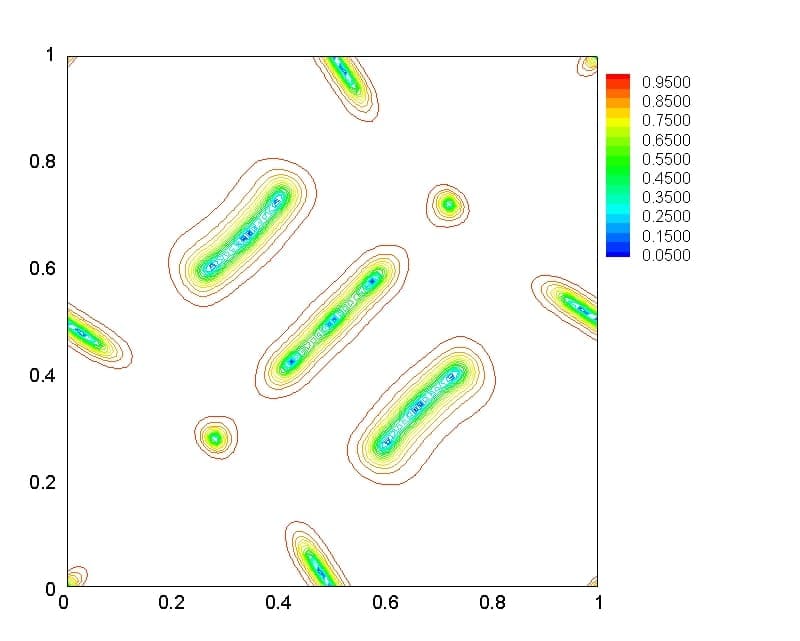}}
 {\includegraphics[height=2.1in]{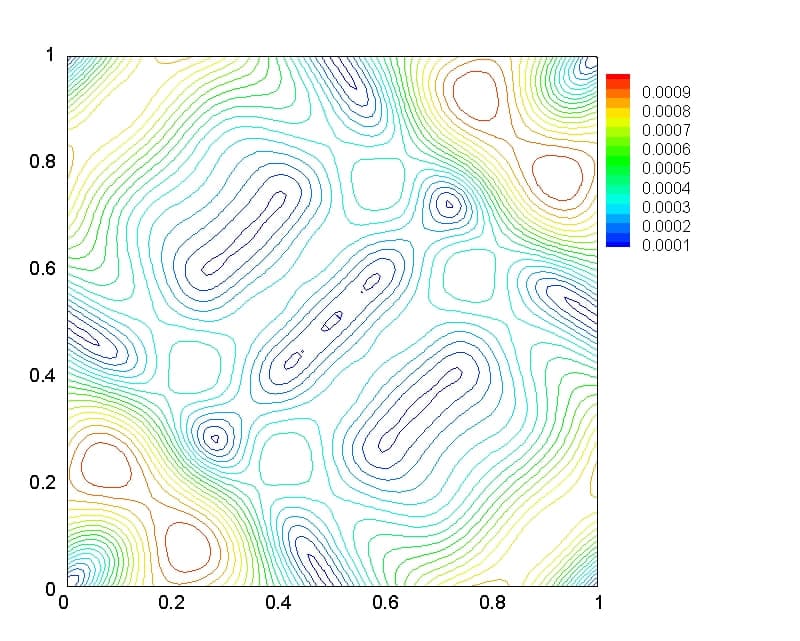}} 
    
\caption{Well-balanced solutions of the 2-D system in a contracting  background with the background geometry $b(x,y)$ at $t=-10^{-1}$, $-10^{-3}$, $-10^{-5}$, $-10^{-8}$. Right column: Rescaled solution $\rhot$. Left column: Velocity magnitude $V$.}
\label{bxy-wb-con}
\end{figure}


\subsection{Conclusion for a contracting background}

Again we are able to ``validate'' the exponents that were derived theoretically. 

\begin{claim}[Compressible fluid flows on a future-contracting cosmological background] 
The asymptotic behavior of solutions to the cosmological fluid model on a future-contracting background is as follows: 
\bei 

\item The density $\rho = \rho(t, x)$ blows up as $t \to 0$ while the velocity approaches zero or the light speed: 
\be
\lim_{t \to 0} \rho(t, x) = +\infty, \qquad 
\lim_{t \to 0} u(t, x) \in \big\{-1, 0, +1 \big\}, 
\qquad x \in [0,1]. 
\ee

\item {\bf Spatially homogeneous and non-homogeneous background.} The rescaled density $\rhot$ defined in \eqref{equa:resc499} approaches a bounded and stationary limit. 

\eei 
\end{claim}


\noindent{\bf Acknowledgments.} 
The authors were partially supported by the Innovative Training Network (ITN) ModCompShock under the grant 642768. Parts of this paper were written while the third author (PLF) was a visiting fellow at the Courant Institute of Mathematical Sciences at New York University during the Academic year 2018--2019. 



\begin{thebibliography}{99}

\bibitem{ALO} \auth{P. Amorim, P.G. LeFloch, and B. Okutmustur},
Finite volume schemes on Lorentzian manifolds,
Comm. Math. Sc. 6 (2008), 1059--1086.

\bibitem{BaezaBMPZ} 
\auth{A. Baeza, S Boscarino, P. Mulet, G. Russo, and D. Zor\'io,}
Approximate Taylor methods for ODEs,
Comput. \& Fluids 159 (2017), 156--166. 

\bibitem{BLF} \auth{Y. Bakhtin and P.G. LeFloch,}
Ergodicity and Hopf-Lax-Oleinik formula for fluid flows evolving around a black hole under a random forcing,
Stoch. Partial Differ. Equ. Anal. Comput. 6 (2018), 746--785.  

\bibitem{BL} \auth{A. Beljadid and P.G. LeFloch,}
A central-upwind geometry-preserving method for hyperbolic conservation laws on the sphere,
Commun. Appl. Math. Comput. Sci. 12 (2017), 1, 81--107. 
 
\bibitem{BLM1} \auth{A. Beljadid, P.G. LeFloch, and M. Mohamadian,} 
Late-time asymptotic behavior of solutions to hyperbolic conservation laws on the sphere,
Comput. Methods Appl. Mech. Engrg. 349 (2019), 285--311. 

\bibitem{BoscarinoLR} \auth{S. Boscarino, P.G. LeFloch,  and G. Russo,}
High-order asymptotic-preserving methods for fully nonlinear relaxation problems,
SIAM J. Sci. Comput. 36 (2014), A377--A395. 

\bibitem{BoscarinoRS} \auth{S. Boscarino, G. Russo, and M. Semplice,}
High-order finite volume schemes for balance laws with stiff relaxation,
 Comput. \& Fluids 169 (2018), 155--168.  

\bibitem{CGL2} 
\auth{Y. Cao, M.A. Ghazizadeh, and P.G. LeFloch,} 
Asymptotic structure of cosmological Burgers flows in one and two space dimensions: a numerical study, Preprint July 2019, ArXiv. 

\bibitem{CLO} \auth{T. Ceylan, P.G. LeFloch, and B. Okutmustur,} 
A finite volume method for the relativistic Burgers equation on a FLRW background spacetime, 
Commun. Comput. Phys. 23 (2018), 500--519.  

\bibitem{Chertock2} \auth{A. Chertock, S. Cui,  A. Kurganov, S.N. \"Ozcan, and E. Tadmor,}
Well-balanced schemes for the Euler equations with gravitation: conservative formulation using global fluxes,
J. Comput. Phys. 358 (2018), 36--52. 

\bibitem{Gie09} \auth{J. Giesselmann},
A convergence result for finite volume schemes on Riemannian manifolds,
Math. Model. Numer. Anal. 43 (2009), 929--955.

\bibitem{GLF} \auth{J. Giesselmann and P.G. LeFloch,} 
Formulation and convergence of the finite volume method for conservation laws on spacetimes with boundary,  Numerische Mathematik (2019). 

\bibitem{GM14} \auth{J. Giesselmann and T. M\"uller,}
Geometric error of finite volume schemes for conservation laws on evolving surfaces, 
Numer. Math. 128 (2014), 489--516.

\bibitem{HLL} \auth{A. Harten, P. D. Lax, and B. V. Leer,} 
On Upstream Differencing and Godunov-Type Schemes for Hyperbolic Conservation Laws,
SIAM Rev., 25(1983), 35--61.

\bibitem{PLF1} \auth{P.G. LeFloch,}
Structure-preserving shock-capturing methods: late-time asymptotics, curved geometry, small-scale dissipation, and nonconservative products,  in ``Lecture Notes of the XV 'Jacques-Louis Lions' Spanish-French school'', 
SEMA SIMAI Springer Series,
Springer Verlag, Switzerland, 2014, pp.~179--222.


\bibitem{PLFM} \auth{P.G. LeFloch and H. Makhlof,}
A geometry-preserving finite volume method for compressible fluids on Schwarzschild spacetime,
Commun. Comput. Phys. 15 (2014), 827--852.

\bibitem{LMO} \auth{P.G. LeFloch, H. Makhlof, and B. Okutmustur,}
Relativistic Burgers equations on a curved spacetime. Derivation and finite volume approximation,
SIAM J. Num. Anal. 50 (2012), 2136--2158. 

\bibitem{LFX1}  \auth{P.G. LeFloch and S. Xiang,}
Weakly regular fluid flows with bounded variation on the domain of outer communication of a Schwarzschild black hole spacetime,
J. Math. Pures Appl. 106 (2016), 1038--1090.

\bibitem{LFX2}  \auth{P.G. LeFloch and S. Xiang,}
A numerical study of the relativistic Burgers and Euler equations on a Schwarzschild black hole exterior, 
Commun. Appl. Math. Comput. Sci. 13 (2018), 271--301.  

\bibitem{LM13} \auth{D. Lengeler and T. M\"uller,}
Scalar conservation laws on constant and time-dependent Riemannian manifolds,
J. Differential Equations 254 (2013), 1705--1727.

\bibitem{MVCS2016}
 \auth{V. Michel-Dansac, C. Berthon, S. Clain, and F. Foucher,}
A well-balanced scheme for the shallow-water equations with topography,
Comput. Math. Appl. 72 (2016), 568--593.

\bibitem{Russo1} \auth{G. Russo,}
Central schemes for conservation laws with application to shallow water equations, S. Rionero, G. Romano Ed., Trends and Applications of Mathematics to Mechanics: STAMM 2002, Springer Verlag, Italy, 2005, pp.~225--246.

\bibitem{Russo2} \auth{G. Russo,}
High-order shock-capturing schemes for balance laws, in
``Numerical solutions of partial differential equations'', Adv. Courses Math. CRM Barcelona, Birkh\"auser, Basel, 2009, pp.~59--147.

\end{thebibliography}
\end{document}